\documentclass[a4paper,11pt]{article}

\usepackage[T1]{fontenc}
\usepackage{lmodern}

\usepackage{dsfont}
\usepackage{color}
\usepackage[latin1]{inputenc}
\usepackage{amsmath}
\usepackage{amsthm}
\usepackage{amssymb}
\usepackage{mathrsfs}
\usepackage{graphicx}
\usepackage[all]{xy}
\usepackage{hyperref}
\usepackage{mathtools}

\usepackage{makeidx}

\usepackage{stmaryrd}
\usepackage{caption}

\usepackage{abstract} 

\newtheorem{thm}{Theorem}[section]
\newtheorem*{theorem*}{Theorem}
\newtheorem{cor}[thm]{Corollary}
\newtheorem{claim}[thm]{Claim}

\newtheorem{lemma}[thm]{Lemma}
\newtheorem{prop}[thm]{Proposition}

\theoremstyle{definition}
\newtheorem{definition}[thm]{Definition}
\newtheorem{ex}[thm]{Example}

\newtheorem{remark}[thm]{Remark}

\newtheorem{conj}[thm]{Conjecture}

\def\rquotient#1#2{%
	\makeatletter
	\raise.3ex\hbox{$#1$}/\lower.3ex\hbox{$#2$}%
	\makeatother
}	

\makeatletter
\newcommand{\subjclass}[2][2010]{%
	\let\@oldtitle\@title%
	\gdef\@title{\@oldtitle\footnotetext{#1 \emph{Mathematics subject classification.} #2}}%
}
\newcommand{\keywords}[1]{%
	\let\@@oldtitle\@title%
	\gdef\@title{\@@oldtitle\footnotetext{\emph{Key words and phrases.} #1.}}%
}
\makeatother

\newcommand{\Address}{{
		\bigskip
		\small
		
		\textsc{University of Montpellier\\ 
Institut Math\'ematiques Alexander Grothendieck\\
Place Eug\`ene Bataillon\\
34090 Montpellier (France)}\par\nopagebreak
		\textit{E-mail address}: \texttt{anthony.genevois@umontpellier.fr}
		\bigskip
		\par\nopagebreak
		\textsc{Institute for Mathematics, TU Berlin, Germany \& \\  Dept. of Mathematics, Heriot-Watt University \& \\ Maxwell Institute for Mathematical Sciences, Edinburgh} \par\nopagebreak
\textit{E-mail address}: \texttt{L.Ciobanu@hw.ac.uk, ciobanu@math.tu-berlin.de} 
		
}}

\makeindex

\title{Contracting elements and conjugacy growth in Coxeter groups, graph products, and further groups}
\date{\today}
\author{Laura Ciobanu and Anthony Genevois}

\begin{document}

\maketitle

\begin{abstract}
In this article we construct contracting elements in the standard Cayley graphs of the so-called periagroups, a family of groups introduced by the second-named author which include Coxeter groups, graph products, and Dyer groups. As a consequence, we deduce that, unless they virtually split as direct products, periagroups are acylindrically hyperbolic and their conjugacy growth series, with respect to standard generating sets, are transcendental.
\end{abstract}

\tableofcontents

\section{Introduction}

The aim of this paper is to establish the existence of contracting elements in several important classes of groups when considering the action on their standard Cayley graphs. The existence of such elements has two main applications: the first is the acylindrical hyperbolicity of the groups, and the second is explicit asymptotics of the conjugacy growth of the groups, asymptotics which imply that the conjugacy growth series of the groups are transcendental. Recall that, given a metric space $X$, an isometry $g$ in $\mathrm{Isom}(X)$ is \emph{contracting} when some (or equivalently, any) orbit $\langle g \rangle \cdot o$, for $o \in X$, is quasi-isometrically embedded and \emph{contracting}; that is, there exists $D \geq 0$ such that, for every ball $B \subset X$ disjoint from $\langle g \rangle \cdot o$, the nearest-point projection of $B $ on $\langle g \rangle \cdot o$ has diameter $\leq D$. One can think of contracting isometries as having a behaviour similar to isometries in negatively curved spaces, even when $X$ is not negatively curved. 

\medskip \noindent
The concept of contracting isometry has been essential for the understanding of an array of topics, including quasi-morphisms and bounded cohomology \cite{MR2507218}, (co)growth in finitely generated groups \cite{MR3404665}, random walks \cite{MR3849623}, counting problems \cite{MR4081104, GeYang}, marked length spectrum rigidity \cite{MR4891022}, and dynamics \cite{MR4803663, CoulonErg}. Contracting isometries also play an important role in the theory of acylindrically hyperbolic groups \cite{MR3415065, MR3849623}. 

\medskip \noindent
Our main contribution is to construct contracting elements in the standard Cayley graphs for \emph{periagroups}. Roughly speaking, given a graph $\Gamma$, a collection of groups $\mathcal{G}$ indexed by $V(\Gamma)$, and a labelling $\lambda : E(\Gamma) \to \mathbb{N}_{\geq 2}$ `compatible' with the size of the vertex-groups, a periagroup is defined by a presentation that interpolates between that of a graph product and a Coxeter group based on $\Gamma, \lambda, \mathcal{G}$, as explained in Section~\ref{sec:periagroups} (see Definition \ref{def:periagroup} and Example \ref{ex:periagroup}).

\begin{thm}\label{thm:main}
Let $\Pi := \Pi(\Gamma, \lambda, \mathcal{G})$ be a finitely generated periagroup. For each factor $G \in \mathcal{G}$, fix a finite generating set $S_G \subset G$. Then $\Pi$ contains a contracting element in its Cayley graph $\mathrm{Cay}(\Pi, \bigcup_{G \in \mathcal{G}} S_G)$ if and only if $\Pi$ is infinite and not virtually a product of two infinite groups. 
\end{thm}

\noindent
We refer to Theorem~\ref{thm:PeriagroupsContracting} for a more precise statement. The focus on Cayley graphs, instead of other metric spaces, will be justified by our applications to conjugacy growth. 

\medskip \noindent
Since periagroups include Coxeter groups (when every factor is cyclic of order $2$), graph products (when $\lambda \equiv 2$), and Dyer groups (when every factor is cyclic), we can deduce from Theorem~\ref{thm:main}, and its proof, the following statements:

\begin{cor}[Theorem~\ref{thm:CoxeterContracting}]\label{cor:ContCoxeter}
A finitely generated Coxeter group $(W,S)$ has a contracting element with respect to its Cayley graph $\mathrm{Cay}(W,S)$, where $S$ is the standard generating set, if and only if it decomposes as a product $W_1 \times \cdots \times W_n$ of irreducible Coxeter groups such that:
\begin{itemize}	
	\item[(i)] $W_2, \ldots, W_n$ are all finite, and
	\item[(ii)] $W_1$ is either non-affine or an infinite dihedral group.
\end{itemize}
\end{cor}

\noindent
Corollary~\ref{cor:ContCoxeter} should be compared with \cite{MR2585575}, which characterises precisely when a Coxeter group contains a rank-one isometry with respect to its action on the corresponding Davis complex. However, we emphasize that, if a group acts on a metric space with a contracting isometry, then there is no guarantee that this element remains contracting in a Cayley graph. 

\begin{cor}
Let $\Gamma$ be a finite graph and $\mathcal{G}= \{ G_u \mid u \in V(\Gamma) \}$ a collection of non-trivial, finitely generated, groups. For every $u \in V(\Gamma)$, fix a finite generating set $S_u \subset G_u$. The graph product $\Gamma \mathcal{G}$ contains a contracting element in its Cayley graph $\mathrm{Cay}(\Gamma \mathcal{G}, \bigcup_{u \in V(\Gamma)} S_u)$ if and only if $\Gamma$ is not complete and does not split as a join $\Gamma_1 \ast \Gamma_2$ such that each $\Gamma_1, \Gamma_2$ contains either a vertex indexed by an infinite group or two non-adjacent vertices.
\end{cor}

\noindent
We refer to Theorem~\ref{thm:ContractingGP} for a more precise statement, which characterises exactly when an element of the graph product is contracting in the Cayley graph under consideration. As a byproduct, we also get a characterisation of Morse elements in graph products in Corollary~\ref{cor:Morse}.

\begin{cor}
A Dyer group $D:= \Pi(\Gamma, \lambda, \mathcal{G})$ contains a contracting element in its Cayley graph $\mathrm{Cay}(D, V(\Gamma))$ if and only if it is infinite and is not virtually a product of two infinite groups.
\end{cor}

\noindent
We refer to Corollary~\ref{cor:DyerAcylHyp} for a more precise statement.

\paragraph{Acylindrical hyperbolicity.} A group $G$ is \emph{acylindrically hyperbolic} if its admits an action on some hyperbolic space $X$ that is non-elementary and \emph{acylindrical}, that is, for every $D \geq 0$, there exist $L,N \geq 0$ such that
$$\# \{ g \in G \mid d(x,gx) , d(y,gy)\} \leq N$$
for all $x,y \in X$ satisfying $d(x,y) \geq L$. Roughly speaking, this means that there is a uniform upper bound on the size of the quasi-stabiliser of any two points that are sufficiently far from each other in $X$. Acylindrical hyperbolicity is a generalisation of hyperbolicity with wide-reaching applications (\cite{MR3430352}).

\medskip \noindent
By \cite{MR3415065}, a group acting properly on a metric space with contracting isometries is acylindrically hyperbolic or virtually cyclic; it thus follows from Theorem~\ref{thm:main} that we can determine precisely when periagroups - and in particular Coxeter groups, graph products, and Dyer groups - are acylindrically hyperbolic. Section~\ref{section:AcylHyp} presents more precise statements and related results already available in the literature.

\paragraph{Conjugacy growth.} Our main motivation for exhibiting contracting elements in Cayley graphs is because their existence determines the conjugacy growth asymptotics and series for the groups we consider, with respect to their standard generating sets.

\medskip \noindent
Let $G$ be a finitely generated group with generating set $X$. For any $n \geq 0$, the \emph{conjugacy growth function} $c(n)=c_{G,X}(n)$ counts the number of conjugacy classes with a minimal length representative of length $n$ with respect to $X$. The \emph{conjugacy growth series} of $G$ with respect to $X$ is then defined as the generating function for $c(n)$ (see Section~\ref{section:ConjGrowth}). Conjugacy growth has been studied in a variety of different groups \cite{AC2017, Guba2010, Hull2013, Mercier2016,  Rivin2010}, most recently including soluble Baumslag-Solitar groups \cite{CiobanuE2020}, graph products \cite{Ciobanu2023} and dihedral Artin groups \cite{CC25}. Virtually abelian groups are known to have rational conjugacy growth series with respect to all generating sets \cite{Evetts2019}, and otherwise all known results support the following conjecture.

\begin{conj}\label{conj:evetts}\cite[Conjecture 7.2]{CiobanuE2020}
Conjugacy growth series of finitely presented groups that are not virtually abelian are transcendental.
\end{conj}

 \noindent
The connection between contracting elements in Cayley graphs and conjugacy growth is as follows:

\begin{theorem*}[Corollary 1.8 in \cite{GeYang}]
Let $G$ be a non-elementary group with finite generating set $S$. If $G$ has a contracting element with respect to the action on the Cayley graph $\mathrm{Cay}(G,S)$, then the conjugacy growth series of $G$ with respect to $S$ is transcendental.
\end{theorem*}

\medskip \noindent
Theorem \ref{thm:main} together with \cite[Corollary 1.8]{GeYang} imply:

\begin{cor}\label{cor:main}
The conjugacy growth series of the groups in Theorem \ref{thm:main}, with respect to the standard generating set for each type of group, are transcendental.
\end{cor}
\medskip \noindent
By applying Corollary \ref{cor:dirprod} to Theorem \ref{thm:main}, Corollary \ref{cor:main} can be generalised to the case when the groups are not irreducible.

\begin{cor}\label{cor:products}
Let $G$ be a graph product, Coxeter group or periagroup that splits non-trivially as a direct product. If the growth rate of one of the factors is strictly larger than the growth rates of the other factors, then the group has transcendental conjugacy growth series with respect to its standard generating set.
\end{cor}
\medskip \noindent
We believe the hypothesis in Corollary \ref{cor:products} that one growth rate dominates the others is in fact not necessary, and all the non-virtually abelian Coxeter groups, graph products and periagroups have transcendental conjugacy growth.

\paragraph{On the proof of Theorem~\ref{thm:main}.} As shown in \cite{Mediangle}, the Cayley graph $\mathrm{Cay}(\Pi, \bigcup \mathcal{G})$ of a periagroup $\Pi := \Pi(\Gamma, \lambda, \mathcal{G})$ admits a very specific structure: it is a \emph{mediangle graph} (see Definition~\ref{def:Mediangle}). However, the generating set $\bigcup \mathcal{G}$ will typically be infinite, so, to work with finite generating sets, we make the following key observation: once we have fixed a finite generating set $S_G \subset G$ for every factor $G \in \mathcal{G}$, it is possible to recover the geometry of $\mathrm{Cay}(\Pi, \bigcup_{G \in \mathcal{G}} S_G)$ from the geometry of $\mathrm{Cay}(\Pi, \bigcup \mathcal{G})$ thanks to a family of \emph{local metrics}, which allows us to keep track of the mediangle geometry of $\mathrm{Cay}(\Pi, \bigcup \mathcal{G})$ on the Cayley graph $\mathrm{Cay}(\Pi, \bigcup_{G \in \mathcal{G}} S_G)$ we are interested in.

\medskip \noindent
More precisely, we introduce \emph{paraclique graphs} in Section~\ref{section:Paraclique}, a new family of graphs that generalise the mediangle graphs previously mentioned; generalising the second-named author's work \cite[Section~3]{QM} for quasi-median graphs, we show in Section~\ref{section:Unfolding} how some system of local metrics on paraclique graphs can be extended to global metrics. Then, in Section~\ref{section:Contracting}, we state and prove the main criterion, namely Theorem~\ref{thm:Contracting}, that allows us to recognise contracting elements for groups acting on paraclique graphs endowed with global metrics. Our criterion is inspired by results for right-angled Artin groups \cite{MR2874959}, CAT(0) cube complexes \cite{MR3339446}, and median graphs \cite{MR4071367}. 

\medskip \noindent
In Section~\ref{section:FirstApplications}, we apply our criterion to Coxeter groups (whose Cayley graphs are bipartite mediangle graphs) and graph products (whose Cayley graphs are quasi-median). Section~\ref{section:periagroups} is dedicated to arbitrary periagroups (whose Cayley graphs are mediangle). Despite the fact that the strategy is similar, the proof of Theorem~\ref{thm:main} for periagroups that are neither Coxeter groups nor graph products is more delicate, and it relies on a careful analysis of some semidirect product decompositions highlighted in \cite{Mediangle}.

\section{Graphs and parallel cliques}

\noindent
In this section, our goal is to introduce paraclique graphs, a new family of graphs, and to show how they are related to other families of graphs. In a nutshell, we have
$$\text{quasi-median} \Rightarrow \text{mediangle} \Rightarrow \text{paraclique} \Rightarrow \begin{array}{c} \text{isometrically embeddable} \\ \text{into a Hamming graph} \end{array}.$$
We first introduce paraclique graphs in Section~\ref{section:Paraclique}, and later recall the concepts of quasi-median (Section \ref{section:QM}) and mediangle graphs (Section \ref{section:MediangleGeometry}).

\medskip \noindent
We start by recording a few basic definitions related to graphs that will be used later.

\begin{definition}
A \emph{graph} is a set $X$ endowed with a symmetric and non-reflexive relation. The elements of $X$ are \emph{vertices}, and two vertices in relation are \emph{adjacent}. An \emph{edge} is the data of two adjacent vertices, and a \emph{path} in $X$ is a sequence of successively adjacent vertices. A graph is \emph{connected} if any two vertices are connected by a path. The \emph{length} of a path is the number of edges it consists of, and a path of minimal length between its endpoints is a \emph{geodesic}. The \emph{distance} between two vertices is the length of any geodesic that connects them. 
\end{definition}

\begin{definition}
A \emph{complete graph} is one where any two vertices are adjacent, and a \emph{clique} is a complete subgraph that is maximal with respect to inclusion. 
\end{definition}

\begin{definition}
Let $X$ be a graph. A(n \emph{induced}) \emph{subgraph} is a subset of $X$ endowed with the restriction of the adjacency relation. A subgraph $Y \subset X$ is \emph{gated} if, for each $x \in X$, there exists a vertex $x_0 \in Y$, referred to as the \emph{gate} of $x$, satisfying the following condition: for every $y \in Y$, there is at least one geodesic from $x$ to $y$ passing through~$x_0$. 
\end{definition}

\noindent
An immediate consequence of the definition is that, given a fixed gated $Y\subset X$, every $x\in X$ has a unique gate with respect to $Y$, also referred to as the \emph{projection on $Y$}; we denote by $\mathrm{proj}_Y : X \to Y$ the map that sends every vertex of $X$ to its projection on $Y$.

\subsection{Paraclique graphs}\label{section:Paraclique}

\noindent
Paraclique graphs are clique-gated graphs that have a good notion of parallelism between cliques, as follows. 

\begin{definition}[\cite{MR1420527}]
A connected graph is \emph{clique-gated} if its cliques are gated. 
\end{definition}

\noindent
A key property of clique-gated graphs, which actually characterises them, is:
 
\begin{lemma}[\cite{MR1420527}]\label{lem:CliqueGated}
Let $X$ be a clique-gated graph. For any two cliques $C_1,C_2 \subset X$, either $C_1$ projects to a single vertex in $C_2$ or $\mathrm{proj}_{C_2}$ induces a bijection $C_1 \to C_2$. 
\end{lemma}

\noindent
This property allows us to define a \emph{parallelism} relation between cliques.

\begin{definition}
In a clique-gated graph, two cliques $C_1$ and $C_2$ are \emph{parallel} if $\mathrm{proj}_{C_2}$ induces a bijection $C_1 \to C_2$. 
\end{definition}

\noindent
It is worth noticing that the parallelism relation between cliques in clique-gated graphs is reflexive and symmetric, but may not be transitive: an example with such behaviour is the bipartite complete graph $K_{2,3}$. 

\begin{definition}
A \emph{paraclique} graph is a clique-gated graph for which parallelism between cliques is transitive.
\end{definition}

\noindent
See Proposition~\ref{prop:Paraclique} for an alternative characterisation of paraclique graphs.

\medskip \noindent
\begin{minipage}{0.3\linewidth}
\includegraphics[width=0.9\linewidth]{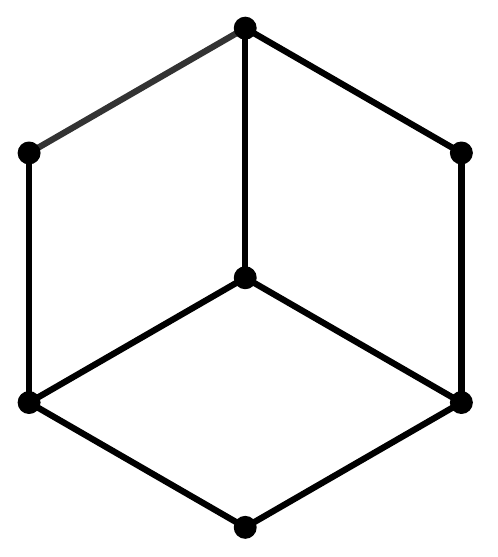}
\end{minipage}
\begin{minipage}{0.69\linewidth}
 The main examples of paraclique graphs we will be interested in are \emph{mediangle graphs} (see Section~\ref{section:MediangleGeometry}), which include (quasi-)median graphs (see Section~\ref{section:QM}), Cayley graphs of Coxeter groups, one-skeleta of some small cancellation polygonal complexes, and hypercellular graphs \cite{Mediangle}. Other examples are given by arbitrary \emph{partial cubes} (i.e.\ isometrically embedded subgraphs in hypercubes). For instance, the wheel of three $4$-cycles illustrated on the left is a paraclique graph, as a partial cube, but it is not mediangle.
\end{minipage}

\medskip \noindent 
It follows from the definition that paraclique graphs have no odd cycles other than triangles, so graphs with odd cycles (of length $> 3$) are not paraclique.
\medskip

 \noindent
We next introduce \emph{hyperplanes}, extending the definition from (quasi-)median graphs.

\begin{definition}\label{def:hyper}
In a paraclique graph, a \emph{hyperplane} is the union of all the edges from a parallelism class of cliques.
\end{definition}

\noindent
An easy consequence of Lemma~\ref{lem:CliqueGated} is that hyperplanes separate a clique-gated graph into convex  components. More precisely:

\begin{prop}\label{prop:HypCliqueGated}
Let $X$ be a paraclique graph. Given a hyperplane $J$, fix a clique $C \subset J$, and let $X \backslash\backslash J$ denote the graph obtained from $X$ by removing the edges in $J$. Then the following hold:
\begin{itemize}
	\item[(i)] Any two vertices of $C$ belong to distinct connected components of $X \backslash\backslash J$.
	\item[(ii)] Conversely, every connected component of $X \backslash\backslash J$ contains a vertex of $C$. 
	\item[(iii)] The connected components of $X\backslash \backslash J$ are convex in $X$.
\end{itemize}
\end{prop}

\noindent
The connected components of $X\backslash \backslash J$, thought of as subgraphs of $X$, are referred to as \emph{sectors}. We emphasize that our graph may (and will often) be locally infinite, so a hyperplane may delimit infinitely many sectors.

\begin{proof}[Proof of Proposition~\ref{prop:HypCliqueGated}.]
Let $x,y \in C$ be two distinct vertices and let $\gamma$ be an arbitrary path connecting $x$ and $y$. Because $\mathrm{proj}_C(x)=x$ and $\mathrm{proj}_C(y)=y$ are distinct, there must exist two adjacent vertices $a,b \in  \gamma$ having distinct projections on $C$. Let $Q$ denote the (unique) clique containing $a$ and $b$. Then, because the projection of $Q$ on $C$ is not reduced to a single vertex, $Q$ must be parallel to $C$, which means that the edge of $\gamma$ that connects $a$ and $b$ belongs to $J$. Thus, $J$ separates $x$ and $y$, and (i) is proved. 

\medskip \noindent
In order to prove the two other items, it suffices to show that, for a fixed vertex $o \in X$, $W:= \mathrm{proj}_C^{-1}(o)$ is convex. 

\medskip \noindent
If $W$ is not convex, then we can find a geodesic $\gamma$ connecting two vertices $x,y \in W$ that is not contained in $W$, or equivalently, that crosses $J$. Fix two consecutive vertices $a,b \in \gamma$ such that $a \in W$ but $b \notin W$. Then we find the contradiction
$$d(a,y)= d(a,b)+d(b,y) = 1+d(b,y) > 1+ d(a,y),$$
where the inequality is justified by the fact that $a$ is the projection of $y$ on the clique containing the edge $[a,b]$. 
\end{proof}

\noindent
Similarly to (quasi-)median graphs, the metric in a paraclique graph can be recovered from hyperplanes only. More precisely:

\begin{prop}\label{prop:ParaGeod}
Let $X$ be a paraclique graph. A path in $X$ is a geodesic if and only if it crosses each hyperplane at most once. Consequently, the distance between two vertices coincides with the number of hyperplanes separating them. 
\end{prop}

\begin{proof}
Let $x,y \in X$ be two vertices and $\alpha$ a path connecting $x$ to $y$. Necessarily, $\alpha$ crosses all the hyperplanes separating $x$ and $y$. The hyperplanes that do not separate $x$ and $y$ must be crossed an even number of times by $\alpha$. Consequently, if $\alpha$ does not cross a hyperplane twice, it follows that it has the least possible length, namely the number of hyperplanes separating $x$ and $y$. Conversely, it follows from the convexity of halfspaces given by Proposition~\ref{prop:HypCliqueGated} that a geodesic crosses each hyperplane at most once. 
\end{proof}

\noindent
We conclude this section by noticing that, as a consequence of our previous observations about hyperplanes, one has the following alternative characterisation of paraclique graphs, which we record because of its independent interest, but which will not be used in the article.

\begin{prop}\label{prop:Paraclique}
A connected graph $X$ is paraclique if and only if 
\begin{itemize}
	\item[(i)] it does not contain an induced copy of $K_4^-$ (i.e.\ the graph obtained from the complete graph $K_4$ by removing an edge), and
	\item[(ii)] it satisfies the \emph{triangle condition}, i.e.\ for all vertices $o,x,y \in X$ with $x,y$ adjacent and $d(o,x)=d(o,y)$, there exists a common neighbour $z$ of $x,y$ such that $d(o,z)<d(o,x)$, and
	\item[(iii)] it can be isometrically embedded into a Hamming graph (i.e.\ a Cartesian product of complete graphs).
\end{itemize}
\end{prop}

\begin{proof}
First note that (i) and (ii) fully characterise clique-gated graphs, that is:

\begin{claim}
A connected graph is clique-gated if and only if it has no $K_4^-$ and it satisfies the triangle condition.
\end{claim}

\noindent
This can be proved easily, and is done in \cite[Theorem~3.1]{MR1420527}.

\medskip \noindent
Therefore, our proposition reduces to proving that a clique-gated graph is paraclique if and only if it can be isometrically embedded into a Hamming graph.

\medskip \noindent
Let $X$ be an isometrically embedded clique-gated subgraph in some Hamming graph $\prod_{i \in I} (K_i,o_i)$, where $K_i$ are complete graphs. Given a clique $C$ of $X$, there exist some $j \in I$, some complete subgraph $C_j \subset K_j$, and some vertices $u_i \in K_i$ for $i \neq j$ such that 
$$C = \prod\limits_{i \in I} C_i \text{ where } C_i:= \{u_i\} \text{ for every } i \neq j.$$
Let $Q$ be another clique of $X$. Similarly, there exist some $k \in I$, some complete subgraph $Q_k \subset K_k$, and some vertices $v_i \in K_i$ for $i \neq k$ such that
$$Q= \prod\limits_{i \in I} Q_i \text{ where } Q_i:= \{v_i\} \text{ for every } i \neq k.$$
We distinguish two cases. First, assume that $j=k$. Up to reindexing the factors in our Hamming graph, we can write
$$C = C_j \times \prod\limits_{i \in I\backslash \{j\}} \{u_i\} \text{ and } Q= Q_j \times \prod\limits_{i \in I \backslash \{j\}} \{v_i\}.$$
Because $Q$ is gated, for every $x \in C_j$, there must exist a unique vertex of $Q$ minimising the distance to $(x,(u_i)_i) \in C$. This implies that $x \in Q_j$ and that the projection of $(x,(u_i)_i)$ on $Q$ is $(x,(v_i)_i)$. Since this is true for every $x \in C_j$ and since the same argument applies to $C$, we conclude that $C_j=Q_j$ and that the projection of $C$ to $Q$ is given by $(x,(u_i)_i) \mapsto (x,(v_i)_i)$. It is bijective, so $C$ and $Q$ are parallel in $X$. Next, assume that $j \neq k$. Up to reindexing the factors in our Hamming graph, we can write
$$C= C_j \times \{ u_k\} \times \prod\limits_{i \in I \backslash \{j,k\}} \{u_i\} \text{ and } Q= \{v_j\} \times Q_k \times \prod\limits_{i \in I\backslash \{j,k\}} \{v_i\}.$$
Then, the projection of $C$ to $Q$ must be $(x,u_k,(u_i)_i) \mapsto (v_j, \mathrm{proj}_{Q_k}(u_k), (v_i)_i)$. In particular, the projection of $C$ to $Q$ is reduced to a single vertex. We conclude that our two cliques $C$ and $Q$ are parallel in $X$ if and only if $j=k$. From this characterisation, it clearly follows that parallelism is transitive.

\medskip \noindent
Conversely, let $X$ be a paraclique graph. Fix a collection of representatives of cliques $(C_i)_{i \in I}$ up to parallelism. Also, fix a vertex $o \in X$; and, for every $i \in I$, let $o_i$ denote the projection of $o$ on $C_i$. We claim that
$$\eta : \left\{ \begin{array}{ccc} X & \to & \prod\limits_{i \in I} (C_i,o_i) \\ x & \mapsto & (\mathrm{proj}_{C_i}(x))_{i \in I} \end{array} \right.$$
is an isometric embedding. First, the map is well-defined because, given a vertex $x \in X$, there are only finitely many cliques $C_i$ on which $x$ and $o$ have distinct projections, namely the cliques that belong to the (finitely many) hyperplanes separating $x$ and $o$. Then, for all vertices $x,y \in X$, we have
$$\begin{array}{lcl} d(\eta(x),\eta(y)) & = & \# \{ i \in I \mid \mathrm{proj}_{C_i}(x) \neq \mathrm{proj}_{C_i}(y) \} \\ \\ & = & \# \{ i \in I \mid \text{the hyperplane containing $C_i$ separates $x$ and $y$}\} \\ \\ & = & \# \{ \text{hyperplanes separating $x$ and $y$}\} = d(x,y) \end{array}$$
proving that $\eta$ is indeed an isometric embedding. 
\end{proof}

\subsection{Quasi-median graphs}\label{section:QM}

\noindent
Recall that a connected graph is \emph{median} if, for any three vertices $x_1,x_2,x_3 \in X$, there exists a unique vertex $m \in X$, referred to as the \emph{median point}, satisfying
$$d(x_i,x_j)=d(x_i,m)+d(m,x_j) \text{ for all } i \neq j.$$
For instance, trees are simple examples of median graphs, as are products of trees (including hypercubes). Alternatively, median graphs can be characterised as retracts of hypercubes \cite{MR766499}. From this perspective, \emph{quasi-median graphs}, defined as retracts of Hamming graphs (i.e.\ Cartesian products of complete graphs), naturally generalise median graphs. The intuition is that, in the same way that median graphs can be thought of as made of cubes, quasi-median graphs can be thought of as made of Hamming graphs (or \emph{prisms}). A formal motivation of this idea is  that the prism-completion of a quasi-median graph is always contractible (and, even stronger, it can be endowed with a CAT(0) metric \cite[Theorem~2.120]{QM}). 

\medskip \noindent
Let us mention an alternative characterisation of quasi-median graphs, which can be used to motivate the definition (Def. \ref{def:Mediangle}) of mediangle graphs in Section~\ref{section:MediangleGeometry}. In the definition below, $I(\cdot,\cdot)$ denotes the \emph{interval} between two vertices, i.e.\ the union of all the geodesics connecting the two vertices under consideration. 

\begin{prop}\label{prop:QMweaklyModular}
A connected graph $X$ is quasi-median if and only if all of the following conditions are satisfied:
\begin{description}
	\item[(Triangle Condition)] For all vertices $o,x,y \in X$ satisfying $d(o,x)=d(o,y)$ and $d(x,y)=1$, there exists a common neighbour $z \in X$ of $x,y$ such that $z \in I(o,x) \cap I(o,y)$.
	\item[(Intersection of Triangles)] $X$ does not contain an induced copy of $K_4^-$.
	\item[(Quadrangle Condition)] For all vertices $o ,x,y,z \in X$ satisfying $d(o,x)=d(o,y)=d(o,z)-1$ and $d(x,z)=d(y,z)=1$, there exists a common neighbour of $x,y$ that belongs to $I(o,x) \cap I(o,y)$.
	\item[(Intersection of $4$-cycles)] $X$ does not contain an induced copy of $K_{3,2}$. 
\end{description}
\end{prop}

\noindent
We refer the reader to \cite{MR1297190} for various characterisations of quasi-median graphs, including the one given above. 

\medskip \noindent
Quasi-median graphs are examples of mediangle graphs, which will be considered in Section~\ref{section:MediangleGeometry}. Therefore, it follows from Proposition~\ref{prop:MediangleParaclique} that quasi-median graphs are also paraclique graphs. As a consequence, all the properties stated and proved in Section~\ref{section:Paraclique} also hold for quasi-median graphs. To sum up,
$$\text{quasi-median} \Rightarrow \text{mediangle} \xRightarrow{\text{Prop.~\ref{prop:MediangleParaclique}} } \text{paraclique} \xRightarrow{\text{Prop.~\ref{prop:Paraclique}} } \begin{array}{c} \text{isometrically embeddable} \\ \text{into a Hamming graph} \end{array}.$$
\begin{minipage}{0.2\linewidth}
\includegraphics[width=0.9\linewidth]{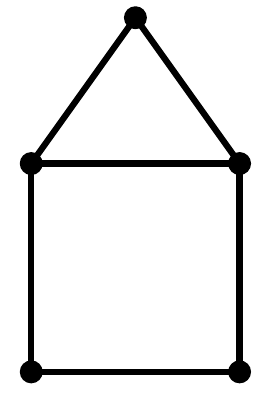}
\end{minipage}
\begin{minipage}{0.79\linewidth}
The classes above are distinct. For example, the graph on the left can be isometrically embedded into a Hamming graph (namely, $K_2 \times K_3$) but is not paraclique. As mentioned in Section~\ref{section:Paraclique}, a wheel of three $4$-cycles is paraclique but not mediangle. Finally, an even cycle of length $\geq 6$ is mediangle but not quasi-median. More conceptually, a major difference between quasi-median and mediangle graphs is that sectors in quasi-median graphs are always gated. 
\end{minipage}

\medskip 

\noindent
In all our types of graphs, there is a good notion of how the hyperplanes may interact. For instance:

\begin{definition}
In a quasi-median/mediangle/paraclique graph, two hyperplanes $J$ and $H$ are \emph{transverse} whenever every sector delimited by $J$ intersects every sector delimited by $H$. 
\end{definition}

\noindent
For example, in a quasi-median graph, the hyperplanes crossing a prism are pairwise transverse.

\subsection{Quasi-median closures}

\noindent
Generalising the construction of a median graph (or equivalently, a CAT(0) cube complex) from a \emph{space with walls}, the second author introduced \emph{spaces with partitions} and showed how to construct quasi-median graphs from them (in \cite{QM}). Here we describe this construction and its application to paraclique graphs.

\medskip \noindent
Let $X$ be a set and $\mathfrak{P}$ a collection of partitions of $X$. 
We refer to the subsets or pieces of the partitions from $\mathfrak{P}$ as \emph{sectors}, and say that: 
\begin{itemize}
	\item Two partitions $\mathcal{P},\mathcal{Q}$ are \emph{nested} if there exist sectors $A \in \mathcal{P}$ and $B \in \mathcal{Q}$ such that $D \subset A$ for every $D \in \mathcal{Q} \backslash \{B\}$ and $D \subset B$ for every $D \in \mathcal{P} \backslash \{A\}$.
	\item Two partitions $\mathcal{P},\mathcal{Q}$ are \emph{transverse} if, for all $P \in \mathcal{P}$ and $Q\in \mathcal{Q}$, $P$ and $Q$ are not $\subset$-comparable.
	\item A partition $\mathcal{P}$ \emph{separates} two points $x,y \in X$ if $x$ and $y$ belong to two distinct sectors of $\mathcal{P}$.
\end{itemize}
The typical example to keep in mind is when $X$ is a quasi-median graph and the partitions in $\mathfrak{P}$ come from the hyperplanes of $X$ cutting $X$ into sectors. Then $X$ together with $\mathfrak{P}$ is a `space with partitions', in that it satisfies the following properties.

\begin{definition}
A \emph{space with partitions} $(X,\mathfrak{P})$ is the data of a set $X$ and a collection of partitions $\mathfrak{P}$ satisfying the following conditions:
\begin{itemize}
	\item for every $\mathcal{P}\in \mathfrak{P}$, $\# \mathcal{P} \geq 2$ and $\emptyset \notin \mathcal{P}$;
	\item for any two (different) partitions $\mathcal{P},\mathcal{Q} \in \mathfrak{P}$, if there exist two sectors $A\in \mathcal{P}$ and $B \in \mathcal{Q}$ such that $A \subset B$, then $\mathcal{P}$ and $\mathcal{Q}$ are nested;
	\item any two points of $X$ are separated by only finitely many partitions of $\mathfrak{P}$. 
\end{itemize}
\end{definition}

\noindent
This generalizes the construction of CAT(0) spaces from spaces with walls, in that we replace half-spaces by sectors, and accordingly adjust the definition of orientation: this will be a map that picks out coherently one sector for each partition, instead of one half-space in the case of walls.

\begin{definition}
Let $(X,\mathfrak{P})$ be a space with partitions. An \emph{orientation} $\sigma$ is a map $\mathfrak{P} \to \{ \text{sectors}\}$ satisfying the following conditions:
\begin{itemize}
	\item $\sigma(\mathcal{P}) \in \mathcal{P}$ for every partition $\mathcal{P} \in \mathfrak{P}$;
	\item $\sigma(\mathcal{P}) \cap \sigma(\mathcal{Q}) \neq \emptyset$ for any two partitions $\mathcal{P},\mathcal{Q} \in \mathfrak{P}$.
\end{itemize} 
For every point $x \in X$, 
$$\sigma_x : \mathcal{P} \mapsto \text{sector in $\mathcal{P}$ containing $x$}$$
is the \emph{principal orientation at $x$}.
\end{definition}

\noindent
Analogous to the setup of spaces with walls, where building a cubulation starts with choosing orientations as vertices and adding edges when two orientations differ on a single wall, we can build a `quasi-cubulation' for spaces with partitions.

\begin{definition}
Let $(X,\mathfrak{P})$ be a space with partitions. The \emph{quasi-cubulation} $\mathrm{QM}(X,\mathfrak{P})$ is the connected component containing the principal orientations of the graph whose vertices are the orientations and whose edges connect two orientations whenever they differ on a single partition.
\end{definition}

\noindent
The statement below summarises the basic properties satisfied by the quasi-cubulation of a space with partitions. See \cite[Proposition~5.46, Theorem~2.56 and its proof, Lemma~2.60, Corollary~2.51]{QM} for more details.

\begin{thm}\label{thm:Popset}
Let $(X,\mathfrak{P})$ be a space with partitions and let 
$\mathrm{QM}:= \mathrm{QM}(X,\mathfrak{P})$ be its the quasi-cubulation. The following assertions hold:
\begin{itemize}
	\item $\mathrm{QM}$ is a quasi-median graph;
	\item the distance between two orientations $\mu,\nu \in \mathrm{QM}$ coincides with the number of partitions of $\mathfrak{P}$ on which they differ;
	\item for every partition $\mathcal{P} \in \mathfrak{P}$ there is a corresponding hyperplane of $\mathrm{QM}$ defined as $$J_\mathcal{P}:= \left\{ \{\mu,\nu \} \mid \mu(\mathcal{P}) \neq \nu(\mathcal{P}) \text{ but } \mu(\mathcal{Q})= \nu(\mathcal{Q}) \text{ for every } \mathcal{Q} \in \mathfrak{P} \backslash \{ \mathcal{P} \} \right\};$$
	\item the map $\mathcal{P} \mapsto J_\mathcal{P}$ induces a bijection from $\mathfrak{P}$ to the hyperplanes of $\mathrm{QM}$ that preserves (non-)transversality. 
\end{itemize}
\end{thm}

\noindent
Since for every paraclique graph there is a decomposition into sectors (coming from the hyperplanes), and therefore a natural space with partitions, we can show that every paraclique graph embeds canonically into a quasi-median graph, which we will refer to as its \emph{quasi-median closure}.

\begin{prop}\label{prop:QMextension}
Let $X$ be a paraclique graph. There exist a quasi-median graph $M$ and an isometric embedding $\iota : X \hookrightarrow M$ such that, for every isometric embedding $\eta$ of $X$ into a quasi-median graph $QM$, one can find an isometric embedding $\xi : M \hookrightarrow QM$ satisfying $\eta = \xi \circ \iota$. Moreover, thinking of $X$ as a subgraph of $M$,
\begin{itemize}
	\item[(i)] the cliques in $X$ are cliques in $M$;
	\item[(ii)] every clique in $M$ is parallel to some clique in $X$;
	\item[(iii)] two cliques in $X$ are parallel in $M$ if and only if they are parallel in $X$;
	\item[(iv)] two hyperplanes are transverse in $M$ if and only if they are transverse in $X$. 
\end{itemize} 
\end{prop}

\begin{proof}
Let $\mathfrak{P}$ denote the set of partitions of $X$ produced by the decompositions into sectors induced by the hyperplanes of $X$. Let $M$ be the corresponding quasi-cubulation and $\iota : X \to M$ the map that sends a vertex of $X$ to the corresponding principal orientation. Notice that, for all vertices $x,y \in X$ and for every hyperplane $J$ of $X$, $J$ separates $x$ and $y$ if and only if $\sigma_x$ and $\sigma_y$ differ on the partition associated to $J$. We deduce from Proposition~\ref{prop:ParaGeod} and Theorem~\ref{thm:Popset} that
$$\begin{array}{lcl} d_M(\sigma_x,\sigma_y) & = & \# \{ \mathcal{P} \in \mathfrak{P} \mid \sigma_x(\mathcal{P}) \neq \sigma_y(\mathcal{P}) \} \\ \\ & = & \# \{ \text{hyperplanes of $X$ separating $x$ and $y$}\} = d_X(x,y). \end{array}$$
In other words, $\iota$ is an isometric embedding. 

\medskip \noindent
Now, let $QM$ be a quasi-median graph containing $X$ as an isometrically embedded subgraph. Given an orientation $\sigma \in M$, we want to construct an orientation $\xi_\sigma$ of $QM$. For every hyperplane $J$ of $QM$, we distinguish two cases: either $J$ does not cross $X$, in which case we define $\xi_\sigma(J)$ as the sector delimited by $J$ that contains $X$; or $J$ crosses $X$, in which case we define $\xi_\sigma(J)$ as the sector delimited by $J$ that contains $\sigma(J)$. It is clear that $\xi_\sigma$ is an orientation of $QM$. Moreover, since $\sigma$ differs from a principal orientation of $X$ on only finitely many hyperplanes, it follows that $\xi_\sigma$ also differs from a principal orientation of $QM$ on only finitely many hyperplanes. Then \cite[Lemma~2.64]{QM} implies that $\xi_\sigma$ is a principal orientation of $QM$. In other words, we can think of $\xi_\sigma$ as a vertex of $QM$, defining a map $M \to QM$ via $\sigma \mapsto \xi_\sigma$. 

\medskip \noindent
First, observe that, for every vertex $x \in X$, if $\sigma_x$ denotes the principal orientation of $X$ given by $x$, then $\xi_{\sigma_x}$ coincides with the principal orientation of $QM$ given by $x$. In other words, thinking of $X$ as living inside both $M$ and $QM$, our map $M \to QM$ restricts to the inclusion $X \hookrightarrow QM$. 

\medskip \noindent
Next, notice that, for all orientations $\mu$ and $\nu$ of $X$, 
$$\begin{array}{lcl} d_{QM}(\xi_\mu,\xi_\nu) & = & \# \{ J \text{ hyperplane of } QM \text{ such that } \xi_\mu(J) \neq \xi_\nu(J) \} \\ \\ & = & \# \{ J \text{ hyperplane of } X \text{ such that } \mu(J) \neq \nu(J) \}= d_M(\mu,\nu). \end{array}$$
Therefore, our map $M \to QM$ turns out to be an isometric embedding. This concludes the proof of the first assertion of our proposition. 

\medskip \noindent
Now we prove $(i)$. Let $C$ be a clique of $X$. If there exists a vertex $x \in M$ that is adjacent to all the vertices in $C$, then, thinking of the vertices in $C \cup \{x\}$ as orientations, it follows from the third item of Theorem~\ref{thm:Popset} that $x$ differs from some orientation $o \in C$ we fix only on the partition $\mathcal{P}$ given by the hyperplane containing $C$. But, if $y \in C$ denotes the vertex of $C$ contained in the sector $x(\mathcal{P})$, then $x$ and $y$ agree on every partition, hence $x=y \in C$. This concludes the proof $(i)$. 

\medskip \noindent
So $\iota$ sends cliques to cliques. Moreover, since $\iota$ is an isometric embedding, it preserves parallelism, that is, $\iota$ sends the hyperplanes of $X$ to hyperplanes of $M$. We claim that this map coincides with the map $\mathcal{P} \mapsto J_\mathcal{P}$ considered in Theorem~\ref{thm:Popset}. Indeed, any partition $\mathcal{P}$ corresponds to a hyperplane $J$ of $X$, and two vertices of $X$ yield an edge of $J$ if and only if $J$ is the only hyperplane separating them, or equivalently, if their principal orientations only differ on $\mathcal{P}$. Therefore $J_\mathcal{P} \cap X = J$. This proves our claim.

\medskip \noindent
Now, Theorem~\ref{thm:Popset} implies that the map $\mathcal{P} \mapsto J_\mathcal{P}$ is surjective, so $(ii)$ follows; it is also injective, so $(iii)$ follows; and finally, it preserves (non-)transversality, so $(iv)$ follows. 
\end{proof}

\section{Unfolding cliques}\label{section:Unfolding}

\subsection{Systems of metrics}

\noindent
In future sections, we will consider Cayley graphs with respect to different sets of generators and `translate' between the different metrics. The context will be as follows:

\medskip \noindent
Let $G$ be a group generated by a collection of subgroups $H_1, \ldots, H_n$. First let $X:=\mathrm{Cay}(G, H_1 \cup \cdots \cup H_n)$ be the Cayley graph of $G$ with generating set $H_1 \cup \cdots \cup H_n$. In the cases we are interested in, the cosets of the $H_i$ correspond to cliques in $X$. Then fix a generating set $S_i$ of $H_i$ for every $1 \leq i \leq n$, and view each clique $C$ of $X$ as a (translate of) $\mathrm{Cay}(H_i,S_i)$ with added edges; each individual $\mathrm{Cay}(H_i,S_i)$ has a metric with respect to $S_i$, and consequently one can endow every clique $C$ with its own metric $\delta_C$ coming from the corresponding $S_i$. One can extend this collection of metrics to a global metric $\delta$ defined on the vertices of $X$ and recover $\mathrm{Cay}(G, S_1 \cup \cdots \cup S_n)$.

\medskip \noindent
Loosely speaking, we ``unfold'' the cliques of $X$. More formally:

\begin{definition}\label{def:delta}
Let $X$ be a graph. A \emph{system of metrics} is a collection of metrics $\{(C,\delta_C) \mid C \text{ clique}\}$ for all cliques. To such a system of metrics we associate the \emph{extended (pseudo-)metric} defined by
$$\delta(x,y) = \inf \left\{ \sum\limits_{i=0}^{n-1} \delta_{C_i}(v_i,v_{i+1}) \right\},$$
where the infimum is taken over all vertices $v_0:=x, v_1, \ldots, v_n:=y$ and all cliques $C_1, \ldots, C_{n-1}$ such that $v_i \in C_i \cap C_{i+1}$ for every $1 \leq i \leq n-1$. 
\end{definition}

\noindent
In paraclique graphs, our systems of metrics will be compatible with projections on cliques in the following sense:

\begin{definition}
Let $X$ be a clique-gated graph. A system of metrics $\{ (C,\delta_C) \mid C \text{ clique}\}$ is \emph{coherent} if, for any two parallel cliques $C_1,C_2 \subset X$, the projection of $C_1$ onto $C_2$ induces an isometry $(C_1,\delta_{C_1}) \to (C_2, \delta_{C_2})$. 
\end{definition}

\noindent
We end this section by showing how the distance given by Definition~\ref{def:delta} can be computed in paraclique graphs. Before stating our proposition, we need to introduce some notation.

\medskip \noindent
Let $X$ be a paraclique graph endowed with a coherent system of metrics $\{ (C,\delta_C) \mid C \text{ clique}\}$. For every hyperplane $J$ and any $x,y \in X$, define 
$$\delta_J(x,y):= \delta_C( \mathrm{proj}_C(x), \mathrm{proj}_C(y)), \text{where $C$ is some/any clique in $J$}.$$ 
Notice that the quantity $\delta_C( \mathrm{proj}_C(x), \mathrm{proj}_C(y))$ does not depend on the choice of clique because the system of metrics is coherent. The $\delta_J$ thus defined are pseudo-metrics. 

\begin{prop}\label{prop:Delta}
Let $X$ be a paraclique graph endowed with a coherent system of metrics $\{(C,\delta_C) \mid C \text{ clique}\}$. Given two vertices $x,y \in X$, fix an arbitrary geodesic $u_0=x, \ldots, u_n=y$ in $X$. For every $0 \leq i \leq n-1$, let $C_i$ denote the clique containing the edge $[u_i,u_{i+1}]$. Then
$$\delta(x,y)= \sum\limits_{i=1}^{n-1} \delta_{C_i}(u_i,u_{i+1}) = \sum\limits_{J \text{ separating $x$ and $y$}} \delta_J(x,y).$$
\end{prop}

\begin{proof}
Fix a path $\gamma$, with vertices $v_0=x, \ldots, v_m=y$, such that
$$\delta(x,y)= \sum\limits_{i=0}^{m-1} \delta_{Q_i} (v_i,v_{i+1}),$$
where $Q_i$ denotes the clique containing the edge $[v_i,v_{i+1}]$ for every $0 \leq i \leq m-1$. Let $\mathcal{J}$ denote the set of the hyperplanes crossed by $\gamma$. We have
$$\delta(x,y)= \sum\limits_{J \in \mathcal{J}} \sum\limits_{i \text{ such that } Q_i \subset J} \delta_{Q_i}(v_i,v_{i+1}).$$
Fix a $J \in \mathcal{J}$ and fix an enumeration $j_1< \cdots < j_k$ of $\{i \mid Q_i \subset J\}$. Notice that, for every $1 \leq s \leq k-1$, the subpath of $\gamma$ connecting $v_{j_s}$ and $v_{j_{s+1}}$ does not cross $J$, hence $\delta_J(v_{j_s},v_{j_s+1}) = \delta_J(v_{j_s},v_{j_{s+1}})$. From the triangle inequality, we deduce that
$$\sum\limits_{i \text{ such that } Q_i \subset J} \delta_{Q_i}(v_i,v_{i+1}) = \sum\limits_{s=1}^{k-1} \delta_J(v_{j_s},v_{j_{s+1}}) \geq \delta_J(v_{j_0},v_{j_k}).$$
This quantity coincides with $\delta_J(x,y)$ since $\gamma$ does not cross $J$ between $x$ and $v_{j_0}$ neither between $v_{j_k}$ and $y$. Hence
$$\delta(x,y) \geq \sum\limits_{J \in \mathcal{J}} \delta_J(x,y) = \sum\limits_{J \text{ separating $x$ and $y$}} \delta_J(x,y),$$
where the last inequality follows from the fact that $\delta_J(x,y)=0$ for every hyperplane $J$ not separating $x$ and $y$. But we know from Proposition~\ref{prop:ParaGeod} that the hyperplanes separating $x$ and $y$ are exactly the hyperplanes crossed by our geodesic $u_0, \ldots, u_n$. More precisely, if $J_i$ denotes the hyperplane containing $C_i$ for every $0 \leq i \leq n-1$, then $J_0,\ldots, J_{n-1}$ are the (pairwise distinct) hyperplanes separating $x$ and $y$. Hence
$$\sum\limits_{J \text{ separating $x$ and $y$}} \delta_J(x,y) = \sum\limits_{i=0}^{n-1} \delta_{J_i} (x,y) = \sum\limits_{i=0}^{n-1} \delta_{C_i}(u_i,u_{i+1}),$$
where the last equality is justified by the fact that $u_i= \mathrm{proj}_{C_i}(x)$ and $u_{i+1}= \mathrm{proj}_{C_i}(y)$ for every $0 \leq i \leq n-1$. In conclusion, we have
$$\delta(x,y) \geq \sum\limits_{J \text{ separating $x$ and $y$}} \delta_J(x,y) = \sum\limits_{i=0}^{n-1} \delta_{C_i}(u_i,u_{i+1}) \geq \delta(x,y),$$
so our proposition follows.
\end{proof}

\subsection{Compatibility with quasi-median closures}

\noindent
Proposition~\ref{prop:QMextension} shows that every paraclique graph embeds canonically into some quasi-median graph. Our next proposition shows that this embedding is compatible with the systems of metrics introduced earlier.

\begin{prop}\label{prop:ExtendingSystemOfMetrics}
Let $X$ be a paraclique graph and let $M$ be its quasi-median extension. A coherent system of metrics $\{ (C,\delta_C) \mid C \subset X \text{ clique}\}$ for $X$ extends uniquely to a coherent system of metrics $\{ (C,\mu_C) \mid C \subset M \text{ clique} \}$ for $M$. Moreover, the inclusion map $X \hookrightarrow M$ induces an isometric embedding $(X,\delta) \hookrightarrow (M,\mu)$. 
\end{prop}

\begin{proof}
Given a clique $C \subset M$, Proposition~\ref{prop:QMextension} implies that there exists a clique $Q \subset X$ parallel to $C$. Therefore, to extend the system of metrics defined on $X$ to a coherent system of metrics on $M$, we set
$$\mu_C : (x,y) \in C^2 \mapsto \delta_Q(\mathrm{proj}_Q(x), \mathrm{proj}_Q(y)).$$
The coherence of our extended system of metrics follows from the coherence of the previous one and from the fact that 
$$\left( \mathrm{proj}_{C_1} \circ \mathrm{proj}_{C_2}\right)_{|C_3} = \left( \mathrm{proj}_{C_1}\right)_{|C_3}$$
for all parallel cliques $C_1,C_2,C_3 \subset M$. 

\medskip \noindent
Next, let $x,y \in X$ be two vertices. Because $X$ is isometrically embedded in $M$, there exists a path $u_0=x, \ldots, u_n=y$ contained in $X$ that is a geodesic in $M$. For every $0 \leq i \leq n-1$, let $C_i$ denote the clique of $X$ (or equivalently, of $M$) that contains the edge $[u_i,u_{i+1}]$. It follows from Proposition~\ref{prop:Delta} that
$$\mu(x,y)= \sum\limits_{i=0}^{n-1} \mu_{C_i} (u_i,u_{i+1}) = \sum\limits_{i=0}^{n-1} \delta_{C_i}(u_i,u_{i+1}) = \delta(x,y).$$
Thus, $(X,\delta)$ is isometrically embedded into $(M,\mu)$, as desired.
\end{proof}

\begin{remark}\label{remark:CanonicalEquivariant}
A consequence of the uniqueness provided by Proposition~\ref{prop:ExtendingSystemOfMetrics} is compatibility with group actions. That is, if a group $G$ acts on a paraclique graph $X$ endowed with a coherent system of metrics $\{(C,\delta_C) \mid C \text{ clique}\}$ in a $G$-invariant way (i.e.\ $\delta_{gC}(gx,gy)=\delta_C(x,y)$ for any clique $C \subset X$, vertices $x,y \in C$, and element $g \in G$), then the action of $G$ on $(X,\delta)$ uniquely extends to $(M,\mu)$. 
\end{remark}

\section{Recognising contracting elements}\label{section:Contracting}

\noindent
For any geodesic metric space $(X,d)$ and subset $Y \subset X$, we let $\pi_Y: X \mapsto Y$ denote the nearest-point projection onto $Y$, i.e.\ $\pi_Y(x)$ represents the points in $Y$ for which the distance between $x$ and $Y$ is minimised. The subset $Y$ is called \emph{contracting} if there exists a constant $C \geq 0$ such that, for any metric ball $B \subset X$ disjoint from $Y$, the diameter of the projection of $B$ onto $Y$ is bounded by $C$, that is, $\mathrm{diam}(\pi_Y(B)) \leq C$.

\medskip \noindent
Now consider a group $G$ acting by isometries on $(X,d)$. An element $g \in G$ is called \emph{contracting in} $X$ if, for some (or equivalently, any) basepoint $o \in X$, the orbit $\langle g \rangle \cdot o$ of $o$ is both quasi-isometrically embedded in $X$ and a $C$-contracting set for some constant $C \geq 0$. 

\medskip \noindent
In this section, we state and prove a sufficient condition for proving that certain isometries of paraclique graphs endowed with coherent systems of metrics are contracting. Before stating our main result, we need some vocabulary.

\begin{definition}\label{def:WellSeparated}
Let $X$ be a paraclique graph endowed with a coherent system of metrics $\mathscr{C}= \{(C,\delta_C) \mid C \subset X \text{ clique}\}$.
\begin{itemize}
	\item A \emph{chain of hyperplanes} is a (possibly finite) sequence $(\ldots, J_1, J_2, \ldots)$ of hyperplanes such that $J_i$ separates $J_{i-1}$ and $J_{i+1}$ for every $i$.
	\item A \emph{facing triple} is the data of three pairwise non-transverse hyperplanes such that none separates the other two.
	\item For every hyperplane $J$ of $X$, its \emph{thickness} $\mathrm{thick}(J)$ is $\mathrm{diam}(C,\delta_C)$ where $C$ is an arbitrary clique in $J$.
	\item For some $L \geq 0$, two hyperplanes $H,K$ are \emph{$L$-well-separated relative to $\mathscr{C}$} if $$\sum\limits_{i=1}^n \mathrm{thick}(J_i) \leq L$$ for every collection $\{J_1, \ldots, J_n\}$ of hyperplanes transverse to both $H$ and $K$ with no facing triple. We say that two hyperplanes are \emph{well-separated} if a constant $L$ as above exists.
\end{itemize}
\end{definition}

\noindent
It is worth mentioning that, if our paraclique graph is locally finite, then the ``no facing triples'' condition in the definition of well-separated hyperplanes can be removed. 

\medskip \noindent
Definition \ref{def:WellSeparated} is motivated by \cite{MR2874959}, in the context of CAT(0) cube complexes, and its subsequent developments \cite{MR4071367, MR4057355}. Our next definition is inspired by \cite{MR2827012}. 

\begin{definition}
Let $X$ be a paraclique graph. An isometry $g \in \mathrm{Isom}(X)$ \emph{skewers} a pair of hyperplanes $J_1,J_2$ if they delimit two sectors $S_1 \supsetneq S_2$ such that $g^nS_1 \subsetneq S_2$ for some integer $n \in \mathbb{Z}$.
\end{definition}

\noindent
The rest of the section is dedicated to the proof of the following statement:

\begin{thm}\label{thm:Contracting}
Let $G$ be a group acting on a paraclique graph $X$ and let $\mathscr{C}=\{(C, \delta_C) \mid C \subset X \text{ clique}\}$ be a coherent and $G$-invariant system of metrics. If $g \in G$ admits an axis in $X$ and skewers a pair of hyperplanes that are well-separated relative to $\mathscr{C}$, then $g$ is contracting in $(X,\delta)$. 
\end{thm}

\noindent
Recall from Remark~\ref{remark:CanonicalEquivariant} that a system of metrics $\mathscr{C}$ is \emph{$G$-invariant} if $\delta_{gC}(gx,gy)= \delta_C(x,y)$ holds for every element $g \in G$, every clique $C \subset X$, and all vertices $x,y \in C$. Also, an \emph{axis} for an isometry refers to a bi-infinite geodesic on which the isometry acts as a non-trivial translation. 

\begin{proof}[Proof of Theorem~\ref{thm:Contracting}.]
By Proposition~\ref{prop:ExtendingSystemOfMetrics}, the system of metrics for $X$ uniquely extends to a coherent and $G$-invariant system of metrics $\{(C,\mu_C) \mid C \text{ clique}\}$ for the quasi-median closure $M$ of $X$. It follows from Proposition~\ref{prop:QMextension} that the element $g$ also skewers a pair of well-separated hyperplanes in $M$. Since $(X,\delta)$ is isometrically embedded into $(M,\mu)$, by Proposition~\ref{prop:ExtendingSystemOfMetrics} it suffices to prove the theorem for quasi-median graphs. From now on, we assume that $X$ is quasi-median.

\medskip \noindent
Let $\gamma$ be an axis of $g$ in $(X,d)$, that is, $\gamma$ is a bi-infinite geodesic with respect to the path metric $d$, and $g$ acts as a translation of length $\tau>0$ on $\gamma$. Let $\gamma^+$ be the path in $(X,\delta)$ obtained from $\gamma$ by connecting any two consecutive vertices $a,b$ in $\gamma$, which necessarily both belong to some clique $C$, by a geodesic in $(C,\delta_C)$. We do this equivariantly so that $g$ acts on $\gamma^+$ as a translation. Then $\gamma^+$ is a geodesic in $(X,\delta)$ (see \cite[Lemma~3.18]{QM}) and, in fact, $\gamma^+$ is an axis of $g$ in $(X,\delta)$. Let $\tau^\ast$ denote the translation length of $g$ with respect to $(X,\delta)$. Set $\tau^+:= \max(\tau,\tau^\ast)$.

\medskip \noindent
We will prove that $\gamma^+$ is contracting in $(X,\delta)$. We first observe that:

\begin{claim}\label{claim:DvsDelta}
For all vertices $a,b \in \gamma$, $d(a,b) \geq \delta(a,b) / \tau^\ast$.
\end{claim}

\noindent
Let $x_0, \ldots, x_\ell$ denote the successive vertices of the subpath of $\gamma$ connecting $a$ and $b$. Note that $\ell = d(a,b)$. For every $0 \leq i \leq \ell-1$, let $C_i$ denote the clique containing $x_i$ and $x_{i+1}$. According to Proposition~\ref{prop:Delta}, 
$$\delta(a,b) = \sum\limits_{i=0}^{\ell-1} \delta_{C_i}(x_i,x_{i+1}) \leq \ell \tau^\ast = d(a,b) \tau^\ast,$$
concluding the proof of Claim~\ref{claim:DvsDelta}. 

\begin{claim}\label{claim:ManySeparatedHyp}
There exist an integer $L \geq 0$ and some hyperplane $J$ such that, up to replacing $g$ with one of its powers, the collection $\{g^kJ \mid k\in \mathbb{Z} \}$ consists of hyperplanes that cross $\gamma$ and that are pairwise $L$-well-separated relative to $\mathscr{C}$. 
\end{claim}

\noindent
By hypothesis, there is an integer $L \geq 0$ and there are two hyperplanes $A$ and $B$ that are $L$-well-separated. Then, up to replacing $g$ with one of its powers, $gA^+ \subsetneq B^+ \subsetneq A^+$ for some sectors $A^+, B^+$ delimited by $A,B$. Notice that 
$$\ldots, g^{-1}A, g^{-1}B, A,B, gA, gB, \ldots$$
defines a chain of hyperplanes that must all cross $\gamma$. Moreover, because $A$ and $B$ both separate $g^{-1}A$ and $gA$, necessarily $g^{-1}A$ and $gA$ are $L$-well-separated relative to $\mathscr{C}$. So $\{ g^{2k}A , k \in \mathbb{Z}\}$ is a collection of hyperplanes that cross $\gamma$ and are pairwise $L$-well-separated relative to $\mathscr{C}$. This concludes the proof of Claim~\ref{claim:ManySeparatedHyp}. 

\medskip \noindent
We fix $L$ and the hyperplane $J$ given by Claim~\ref{claim:ManySeparatedHyp}, and let $J_k:= g^k J$ for $k \in \mathbb{Z}$. Then $(\ldots, J_{-1},J_0,J_1,\ldots)$ is a chain of hyperplanes that are pairwise $L$-well-separated relative to $\mathscr{C}$ and all the $J_k$ cross $\gamma$ along edges at distances $\tau$ in $(X,d)$ and $\tau^\ast$ in $(X,\delta)$.

\medskip \noindent
Fix two vertices $x,y \in X$ and let $a,b$ be two nearest points to $x$ and $y$ on $\gamma^+$, respectively. Assume
$$\delta(a,b) > \tau^\ast \tau ( 6\tau^++4L+3).$$
We prove below that the ball centred at $x$ of radius $\delta(x,y)$, in $(X,\delta)$, intersects $\gamma^+$. This will be enough to conclude that $\gamma^+$ is contracting in $(X,\delta)$. 

\medskip \noindent
As a consequence of Claim~\ref{claim:DvsDelta}, there must exist more than $6\tau^++4L+3$ of the $J_k$'s that separate $a$ and $b$. Up to reindexing the hyperplanes, assume that $J_1, \ldots, J_m$ separate $a$ and $b$, where $m=6\tau^++4L+3$. 

\begin{claim}\label{claim:MidHypOne}
If $r> \tau^* +L$, then $J_r$ does not separate $a$ and $x$. Similarly, if $r< m- \tau^*-L$, then $J_r$ does not separate $b$ and $y$.
\end{claim}

\noindent
The two assertions being symmetric, we only prove the first one. Assume that $J_r$ separates $a$ and $x$ for some $r$. Let $a' \in \gamma^+$ denote the vertex adjacent to $J_r$ that is separated from $a$ by $J_r$. By the definition of $a$, we must have $\delta(x,a') \geq \delta(x,a)$. 

\begin{center}
\includegraphics[width=0.35\linewidth]{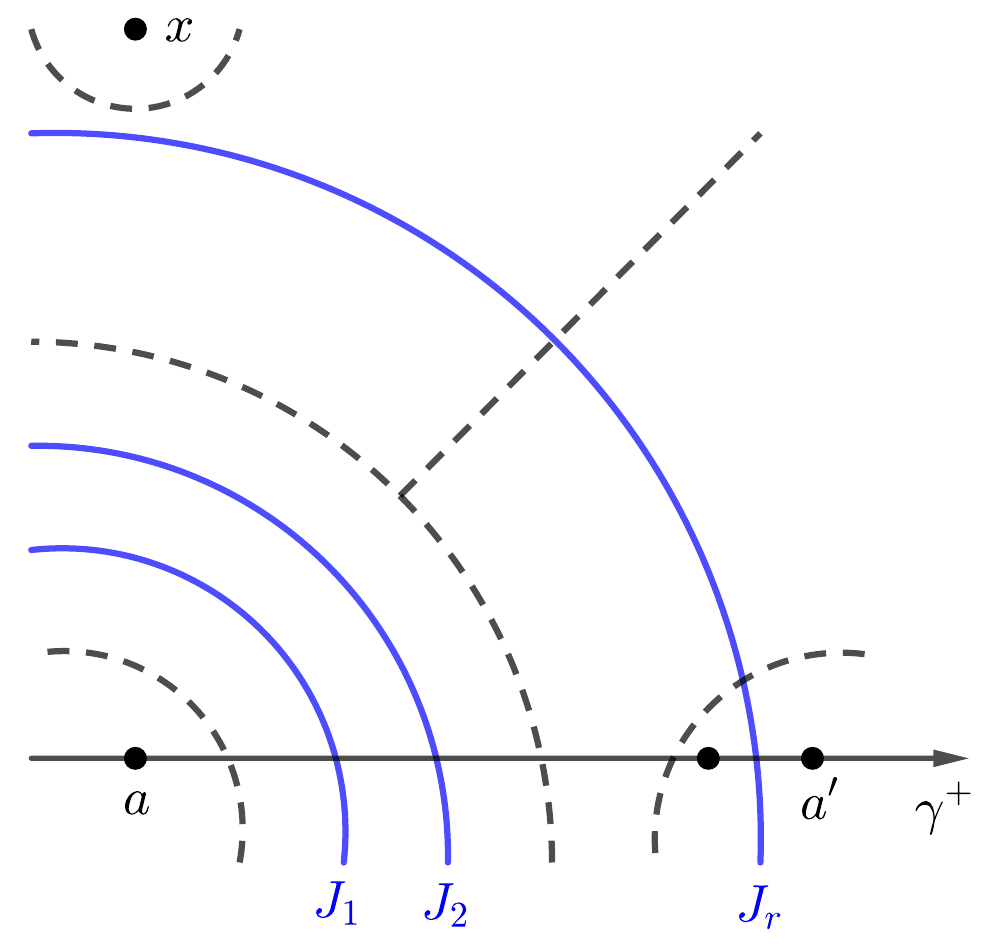}
\end{center}

\noindent
On one hand, we have
$$\delta(x,a)= \sum\limits_{J \text{ sep. $x$ from $\{a,a'\}$}} \delta_J(x,a) + \sum\limits_{J \text{ sep. $a$ from $\{x,a'\}$}} \delta_J(x,a) + \sum\limits_{J \text{ sep. $x,a,a'$}} \delta_J(x,a).$$
At least $r$ hyperplanes, $J_1, \ldots, J_r$, separate $x$ and $a$, and also cross $\gamma^+$, so it follows that
\begin{equation}\label{eq:One}
\delta(x,a) \geq \sum\limits_{J \text{ sep. $x$ from $\{a,a'\}$}} \delta_J(x,a) +r.
\end{equation}
On the other hand, we have
$$\delta(x,a')= \sum\limits_{J \text{ sep. $x$ from $\{a,a'\}$}} \delta_J(x,a') + \sum\limits_{J \text{ sep. $a'$ from $\{x,a\}$}} \delta_J(x,a') + \sum\limits_{J \text{ sep. $x,a,a'$}} \delta_J(x,a').$$
Notice that, for every hyperplane $J$ separating $x$ from $\{a,a'\}$, we have $\delta_J(x,a)=\delta_J(x,a')$. Also, every hyperplane separating $a'$ from $a$ and $x$ must be transverse to $J_r$. For such a hyperplane, there are two possibilities: either it crosses $\gamma^+$ between $J_r$ and $J_{r-1}$; or it is transverse to both $J_{r-1}$ and $J_r$. Therefore,
$$ \sum\limits_{J \text{ sep. $a'$ from $\{x,a\}$}} \delta_J(x,a') + \sum\limits_{J \text{ sep. $x,a,a'$}} \delta_J(x,a') \leq \tau^\ast +L.$$
So far, we have proved that
\begin{equation}\label{eq:Two}
\delta(x,a') \leq \sum\limits_{J \text{ sep. $x$ from $\{a,a'\}$}} \delta_J(x,a) + \tau^\ast +L
\end{equation}
By combining the inequalities (\ref{eq:One}) and (\ref{eq:Two}), we deduce from $\delta(x,a') \geq \delta(x,a)$ that $r \leq \tau^\ast +L$. This concludes the proof of Claim~\ref{claim:MidHypOne}. 

\medskip \noindent
We established that $J_1, \ldots, J_m$ separate $a$ and $b$. Among these hyperplanes, some do not separate $x$ and $y$ by Claim \ref{claim:MidHypOne}. By removing those from consideration and setting $s:=\tau^++L+2$, we get that $J_{s-1},J_s, \ldots, J_{m-s+1}$ do separate $x$ and $y$. Let us exploit this observation to conclude that
$$\delta(x,\gamma^+) < \delta(x,y).$$
For this, fix a vertex $z$ between $J_{s}$ and $J_{s+1}$ that belongs to a geodesic in $(X,\delta)$ connecting $x$ to $y$. Also, fix a vertex $c \in \gamma^+$ between $J_{s}$ and $J_{s+1}$. Finally, let $c^-,c^+ \in \gamma^+$ denote the two vertices at minimal distance that are separated by $J_{s-1}$ and $J_{s+2}$. 

\begin{center}
\includegraphics[width=0.6\linewidth]{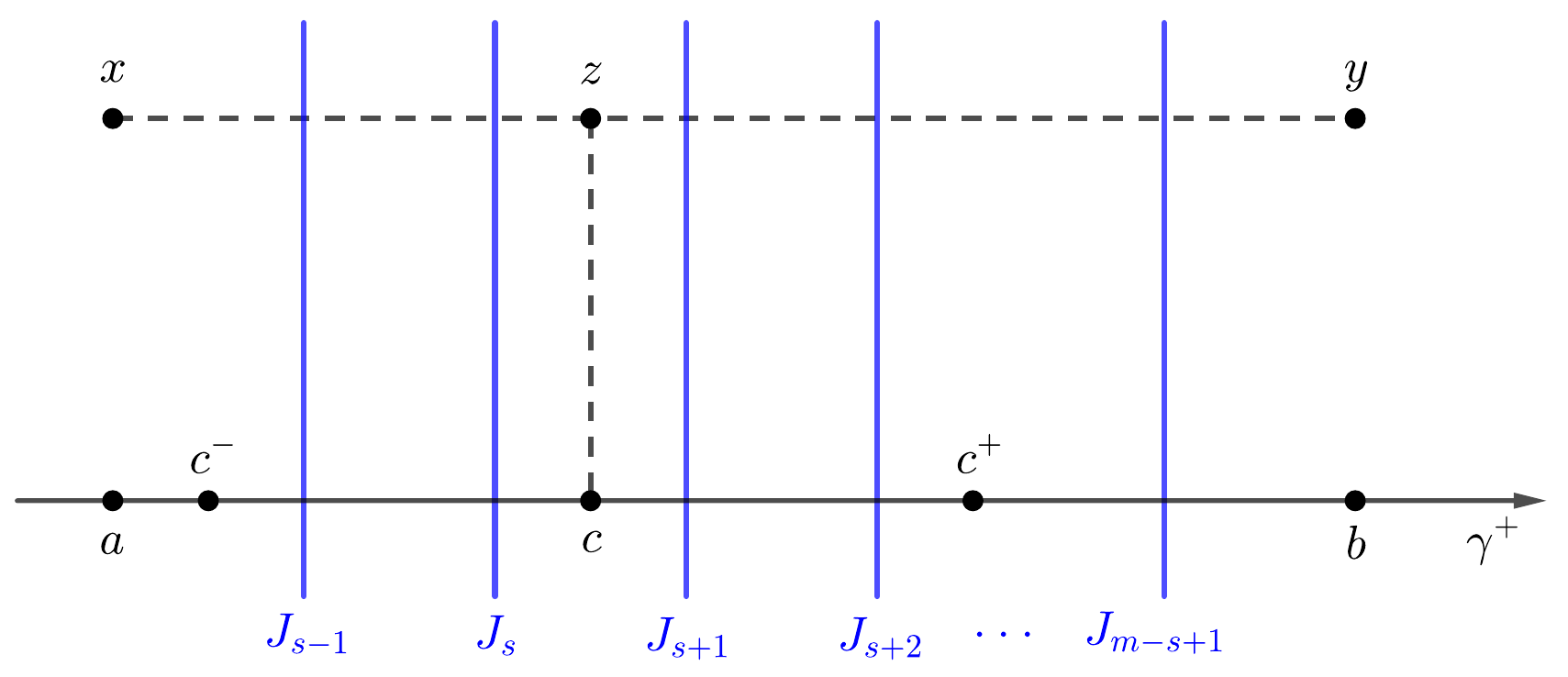}
\end{center}

\noindent
Notice that a hyperplane separating $z$ and $c$ must either be transverse to $J_s$ and $J_{s-1}$, or be transverse to $J_{s+1}$ and $J_{s+2}$, or cross $\gamma^+$ between $J_{s-1}$ and $J_{s+2}$. Therefore, because $J_s$ and $J_{s-1}$, as well as $J_{s+1}$ and $J_{s+2}$, are $L$-well-separated relative to $\mathscr{C}$, we have
$$\delta(c,z) =\sum\limits_{J \text{ sep. $c$ and $z$}} \delta_J(c,z) \leq 2L+ \delta(c^-,c^+) \leq 2(L+2\tau^\ast).$$
But, by definition of $z$, $J_{s+1},\ldots, J_{m-s+1}$ all separate $z$ and $y$, hence
$$\delta(z,y) \geq m-2s+1> 2(L+2 \tau^+) \geq \delta(c,z).$$
We conclude that 
$$\delta(x,\gamma^+) \leq \delta(x,c) \leq \delta(x,z)+\delta(z,c) < \delta(x,z)+\delta(z,y)= \delta(x,y),$$
which concludes the proof of our theorem.
\end{proof}

\noindent
Theorem~\ref{thm:Contracting} applies to paraclique graphs themselves with the path/graph metric $d$, since the metric $\delta$ obtained by endowing every clique of a paraclique graph with the discrete metric coincides with the graph metric $d$. More explicitly:

\begin{cor}\label{cor:contract}
Let $G$ be a group acting on a paraclique graph $X$. If $g \in G$ admits an axis $\gamma$ in $X$ and $\gamma$ intersects two well-separated hyperplanes, then $g$ is contracting in $X$. 
\end{cor} 

\noindent
Here, two well-separated hyperplanes in a paraclique graph correspond to two hyperplanes $H$ and $K$ such that, for some $L \geq 0$, every collection of hyperplanes transverse to both $H,K$ and with no facing triple has cardinality $\leq L$. This is compatible with Definition~\ref{def:WellSeparated}. In practice, we will verify the stronger property that only finitely many hyperplanes are transverse to our two hyperplanes $H$ and $K$.

\section{First applications}\label{section:FirstApplications}

\noindent
In this section, we focus our attention on two important classes of groups: Coxeter groups and graph products of groups. These specific cases will allow us to generalise our results to arbitrary periagroups in Section~\ref{section:periagroups}.

\subsection{Coxeter groups}\label{section:Coxeter}

\noindent
As a first application of Theorem~\ref{thm:Contracting}, we determine when a Coxeter group contains a contracting element with respect to its canonical Cayley graph. 

\begin{thm}\label{thm:CoxeterContracting}
A finitely generated Coxeter group $(W,S)$ has a contracting element with respect to its Cayley graph $\mathrm{Cay}(W,S)$, where $S$ is the standard generating set, if and only if it decomposes as a product $W_1 \times \cdots \times W_n$ of irreducible Coxeter groups such that:
\begin{itemize}	
	\item[(i)] $W_2, \ldots, W_n$ are all finite, and
	\item[(ii)] $W_1$ is either non-affine or an infinite dihedral group.
\end{itemize}
\end{thm}

\noindent
In this section, we assume some familiarity with Davis complexes associated to Coxeter groups. Indeed, any Coxeter group $(W,S)$ acts geometrically on its so-called Davis complex $D(W)$ (\cite{MR2360474}), which is CAT(0), and, if a coarser cellulation is used for $D(W)$, then $\mathrm{Cay}(W,S)$ can be seen as the $1$-skeleton of $D(W)$. Moreover, $\mathrm{Cay}(W,S)$ is mediangle by \cite{Mediangle}, and therefore paraclique; in particular, $\mathrm{Cay}(W,S)$ is endowed with a structure of hyperplanes, and there is a natural bijection between the walls in $D(W)$ and the hyperplanes in $\mathrm{Cay}(W,S)$, which sends a wall in $D(W)$ to the set of the edges in $\mathrm{Cay}(W,S)$ that it crosses (see \cite{MR2360474, MR1983376} for more information). 

\medskip \noindent
For short, our strategy to prove Theorem~\ref{thm:CoxeterContracting} will use the fact, provided by By \cite{MR2585575}, that $W$ contains a rank-one element in $D(W)$, or equivalently, a contracting element for $D(W)$ \cite{MR2507218}. Our goal is to show that contracting elements in $D(W)$ are also contracting in $\mathrm{Cay}(W,S)$. This will be the core of the proof for Theorem \ref{thm:CoxeterContracting}.

\medskip \noindent
The following observation will be used in the proof of Theorem \ref{thm:CoxeterContracting}. It essentially states that, if a geodesic segment has a `long' projection on a contracting geodesic in a CAT(0) space, then the segment and its projection are `close' to each other. We denote by $[x,y]$ a geodesic segment between two points $x$ and $y$.

\begin{lemma}\label{lem:ContractingQuadrangle}
Let $(X,d)$ be a CAT(0) space and $\gamma$ a $B$-contracting geodesic. There exists a constant $\Delta \geq 0$ such that the following holds. For any $x,y \in X$, if the projections $x',y'$ of $x,y$ on $\gamma$ lie at distance $>\Delta$, then $[x',y'] \subset [x,y]^{+\Delta}$, that is, $[x',y']$ lies in the (Hausdorff) $\Delta$-neighbourhood of $[x,y]$. 
\end{lemma}

\begin{proof}
We assume that $d(x',y')$ is sufficiently large compared to the contracting constant of $\gamma$. According to \cite[Corollary~3.4]{MR2507218}, there exist two points $z \in [x,y]$ and $z' \in [x',y']$ such that $d(z,z')$ is bounded above by a constant depending only on the contracting constant of $\gamma$. Because a CAT(0) metric is convex, the geodesic $[x,z']$ (resp.\ $[y,z']$) is contained in a controlled neighbourhood of $[x,z]$ (resp.\ $[y,z]$). Therefore, it suffices to show that $[x',y']$ is contained in a controlled neighbourhood of $[x,z'] \cup [z',y]$. According to \cite[Lemma~3.5]{MR2507218}, there exists $x'' \in [x,z']$ such that $d(x',x'')$ is bounded above by a constant depending only on the contracting constant of $\gamma$. Since a CAT(0) metric is convex, $[x',z']$ is contained in a controlled neighbourhood of $[x'',z']$, and a fortiori in a controlled neighbourhood of $[x,y]$. The same holds for $[y',z']$. This concludes the proof of our lemma.
\end{proof}

\begin{proof}[Proof of Theorem~\ref{thm:CoxeterContracting}.]
First, assume that $W$ contains a contracting element with respect to its Cayley graph. Then $W$ must be infinite and it cannot split as a direct product of two infinite groups. Without loss of generality we may decompose $W$ as a direct product $W_1 \times \cdots \times W_n$ of irreducible Coxeter groups, where $W_1$ is infinite and $W_2, \ldots, W_n$ are all finite. A sufficiently large power of a contracting element will belong to $W_1$, so $W_1$ has to contain a contracting element. Since a group that is virtually free abelian of rank $\geq 2$ cannot contain a contracting element, it follows that, if $W_1$ is affine (virtually abelian), then it must be virtually infinite cyclic. But the only irreducible Coxeter group that is virtually infinite cyclic is the infinite dihedral group.

\medskip \noindent
Conversely, assume that $W$ decomposes as a product $W_1 \times \cdots \times W_n$ of irreducible Coxeter groups such that $W_2, \ldots, W_n$ are all finite and such that $W_1$ is either non-affine or an infinite dihedral group. In order to conclude that $W$ contains a contracting element, it suffices to show that $W_1$ contains a contracting element. Therefore we may assume from now on that $W$ is irreducible and either non-affine or an infinite dihedral group. In fact, since an infinite dihedral group clearly contains a contracting element, we will assume that $W$ is non-affine. 

\medskip \noindent 
Let $D(W)$ be the Davis complex for $(W,S)$, which is CAT(0) (see Remark \ref{rem:Davis}). Fix an element $g \in W$ contracting in $D(W)$ and let $\gamma \subset D(W)$ be a CAT(0) axis of $g$. Let $\tau$ denote the translation length of $g$ along $\gamma$.

\medskip \noindent 
Let $\Delta$ be the constant from Lemma~\ref{lem:ContractingQuadrangle}, and let $N$ denote the maximal number of walls intersecting a ball of radius $\Delta$ in $D(W)$. Such finite $N$ exists because, in $D(W)$, only finitely many walls may intersect a ball of radius $\Delta$: a ball contains a finite number of edges (depending only on the radius of the ball) and there exists a unique wall crossing a given edge.

\begin{claim}\label{claim:OneGoodWall}
There exists a wall $J$ in $D(W)$ that intersects $\gamma$ and whose CAT(0) projection on $\gamma$ is bounded, of diameter $< \max(2N \tau, \Delta)$.
\end{claim}

\noindent
The set of all walls in $D(W)$ cut $D(W)$ into bounded pieces, so there must exist a wall, say $K$, that intersects $\gamma$. We first prove that, if the projection of $K$ on $\gamma$ is sufficiently large, then it has to be stabilized by some non-trivial power of $g$.

\medskip \noindent
Assume that the projection of $K$ on $\gamma$ has diameter $\geq 2 \max(2N\tau, \Delta)$. Then, we can find $p \in \gamma \cap K$ and $q \in K$ such that the distance between $p$ and the projection $q'$ of $q$ on $\gamma$ is $\geq \max(2N \tau, \Delta)$. Up to replacing $g$ with $g^{-1}$, we may assume that $g$ translates $p$ towards $q'$. 

\medskip \noindent
By Lemma~\ref{lem:ContractingQuadrangle}, since $d(p,q') \geq \Delta$, we get $[p,q'] \subset [p,q]^{+\Delta}$. 

\medskip \noindent
Next, for an index $0 \leq k \leq N$, and $o$ the midpoint of $[p,q']$, observe that $g^kp \in [p,o]$:
$$d(p,o) = d(p,q')/2 \geq N \tau \geq k \tau = d(p,g^k p).$$
This implies that
$$o \in [g^k p, g^kq'] = g^k [p,q'] \subset g^k [p,q]^{+\Delta} \subset g^k K^{+\Delta}.$$
So $g^k K$ intersects $B(o,\Delta)$ for all $0 \leq k \leq N$. But since there at most $N$ planes intersecting $B(o,\Delta)$, there must exist distinct $0 \leq i,j \leq N$ such that $g^iK=g^jK$. In other words, a non-trivial power of $g$ stabilizes $K$, as we stated earlier.

\medskip \noindent
Therefore, in order to prove Claim \ref{claim:OneGoodWall}, it suffices to find a wall $J$ that crosses $\gamma$ but that is not stabilized by a non-trivial power of $g$. 

\medskip \noindent
We claim there are only finitely many walls stabilised by non-trivial powers of $g$. Indeed, let $H$ be a wall stabilized by some non-trivial power $g^s$ of $g$. As a convex subspace, $H$ is CAT(0) in its own right. Moreover, $H$ inherits a cellular structure from $D(W)$ with only finitely many isometry types of cells. Therefore, as $g^s$ necessarily has unbounded orbits in $H$, it has an axis $\alpha$ in $H$ which is also an axis in $D(W)$. But $\gamma$ is also an axis for $g^s$, and the projection of $\alpha$ onto $\gamma$ is all of $\gamma$, since $\alpha$ and $\gamma$ are parallel; so it follows from Lemma~\ref{lem:ContractingQuadrangle} that the Hausdorff distance between the two axes of $g^s$ is uniformly bounded. Thus, given a point $x \in \gamma$, all the walls stabilized by a power of $g$ intersect a ball of radius $\Delta$ centered at $x$, which implies that there are only finitely many such walls. 

\medskip \noindent
Now suppose $y,z \in \gamma$ are two points very far away from each other. Since the distance in $D(W)$ coarsely coincides with the number of separating walls, $y$ and $z$ are separated by at least one wall $J$ that is not stabilized by a non-trivial power of $g$. Additionally, $J$ must intersect $\gamma$, because it separates $y$ and $z$. And, according to the argument above, the projection of $J$ on $\gamma$ is bounded above by $\max(2N \tau, \Delta)$, which proves Claim \ref{claim:OneGoodWall}.

\begin{claim}\label{claim:SeparatedWalls}
There exists $r \geq 1$ such that $J$ and $g^rJ$ are disjoint and such that only finitely many walls cross both $J$ and $g^rJ$. 
\end{claim}

\noindent
Let $B$ denote the diameter of the projection of $J$ on $\gamma$. It is finite according to Claim~\ref{claim:OneGoodWall}. Recall that $\tau$ is the translation length of $g$ along $\gamma$. Fix an $r>(2B+\Delta)/ \tau$.

\medskip \noindent
First, notice that $J$ and $g^rJ$ are disjoint. Fix a point $p \in J \cap \gamma$. If $J$ and $g^rJ$ intersect, then so do their projections on $\gamma$. Therefore, if $z$ is a point in this intersection of projections, then
$$r \tau = d(p,g^rp) \leq d(p,z)+ (z,g^rp) \leq 2B,$$
contradicting our choice of $r$. Therefore, $J$ and $g^rJ$ are indeed disjoint.

\medskip \noindent
Now, let $K$ be a wall crossing both $J$ and $g^r J$. We will prove that $K$ intersects the ball $B(o,\Delta)$, where $o \in \gamma$ denotes the midpoint of $p$ and $g^rp$. Since there are only finitely many walls intersecting a given ball, our claim will follow.

\medskip \noindent
Fix two points $a \in J \cap K$ and $b \in g^r J \cap K$. The projection $a'$ of $a$ on $\gamma$ must be at distance $\leq B$ from $p$; and, similarly, the projection $b'$ of $b$ on $\gamma$ must be at distance $\leq B$ from $g^rp$. Because $d(p,g^rp)=r \tau>2B$, the midpoint $o$ must lie between $a'$ and $b'$ on $\gamma$. Moreover, because
$$d(a',b') \geq d(p,g^rp) - d(a',p)-d(b',g^rp) \geq r \tau - 2B >\Delta,$$
Lemma~\ref{lem:ContractingQuadrangle} applies and shows that $[a,b] \subset K$ intersects $B(o,\Delta)$, as desired. This concludes the proof of Claim~\ref{claim:SeparatedWalls}. 

\medskip \noindent
Finally, because of the bijective correspondence between walls in $D(W)$ and hyperplanes in $Cay(W,S)$, it follows from Claim~\ref{claim:SeparatedWalls} that Corollary \ref{cor:contract} applies, so $g$ is contracting in $Cay(W,S)$. This concludes the proof of the theorem. 
\end{proof}

\subsection{Graph products}\label{section:GraphProduct}

\noindent
As a second application of Theorem~\ref{thm:Contracting}, we construct contracting elements in graph products of finitely generated groups. Recall that, given a graph $\Gamma$ and a collection of groups $\mathcal{G}= \{ G_u \mid u \in \Gamma \}$ indexed by $\Gamma$, their \emph{graph product} is 
$$\Gamma \mathcal{G}=\langle G_u \ (u\in \Gamma) \mid [G_u,G_v]=1 \ (\{u,v\} \in E(\Gamma)) \rangle,$$
where $E(\Gamma)$ denotes the edge-set of $\Gamma$ and where $[G_u,G_v]=1$ is shorthand for $[g,h]=1$ for all $g \in G_u$, $h \in G_v$. The groups of $\mathcal{G}$ are referred to as \emph{vertex-groups}. Unless explicitly stated, vertex-groups are not assumed to be finite.

\medskip \noindent
\textbf{Convention.} We assume that the groups in $\mathcal{G}$ are non-trivial. This is not restrictive, since a graph product with some trivial vertex groups is a graph product over a smaller graph, all of whose vertex groups are non-trivial.

\medskip \noindent
A \emph{word} in $\Gamma \mathcal{G}$ is a product $g_1 \cdots g_n$, where $n \geq 0$, and for every $1 \leq i \leq n$, $g_i \in G$ for some $G \in \mathcal{G}$; the $g_i$'s are the \emph{syllables} of the word, and $n$ is the \emph{length} of the word. The following elementary operations on a word do not modify the element of $\Gamma \mathcal{G}$ it represents:
\begin{description}
	\item[Cancellation:] delete the syllable $g_i$ if $g_i=1$;
	\item[Amalgamation:] if $g_i,g_{i+1} \in G$ for some $G \in \mathcal{G}$, replace the two syllables $g_i$ and $g_{i+1}$ by the single syllable $g_ig_{i+1} \in G$;
	\item[Shuffling:] if $g_i$ and $g_{i+1}$ belong to two adjacent vertex-groups, shuffle them.
\end{description}
A word is \emph{graphically reduced} if its length cannot be shortened by applying these elementary moves. Every element of $\Gamma \mathcal{G}$ can be represented by a graphically reduced word, and this word is unique up to the shuffling operation. The subgraph of $\Gamma$ induced by the vertices whose vertex-groups contain the syllables of such a word is the \emph{support} of our element. A word is \emph{cyclically graphically reduced} if its length cannot be shortened by applying these elementary moves and cyclic permutations. Every element $g$ of $\Gamma \mathcal{G}$ has a conjugate that can be reprensented by a cyclically graphically reduced word, and the support of such a conjugate is referred to as the \emph{essential support} of $g$. 
For more information on graphically reduced words, we refer to \cite{GreenGP} (see also \cite{HsuWise,VanKampenGP}). 

\medskip\noindent
Given an induced subgraph $\Lambda$ of $\Gamma$, we can define the graph product based on the vertex-groups associated to $\Lambda$, and for simplicity will write $\langle \Lambda \rangle$ for this instead of $\Lambda\mathcal{G}$.

\medskip \noindent
Recall that a \emph{join} refers to a subgraph $\Lambda$ of $\Gamma$  that is the union of two non-empty subgraphs $\Lambda_1, \Lambda_2$ with every vertex in $\Lambda_1$ adjacent to every vertex in $\Lambda_2$, so $\langle \Lambda_1 \sqcup \Lambda_2\rangle =\langle \Lambda_1 \rangle \times \langle \Lambda_2 \rangle$. Such a join is \emph{large} if the subgroups $\langle \Lambda_1 \rangle$ and $\langle \Lambda_2 \rangle$ of $\Gamma \mathcal{G}$ are both infinite (so either each of $\Lambda_1$ and $\Lambda_2$ has an infinite vertex-group, or has two non-adjacent vertices).  

\medskip \noindent
The rest of the section is dedicated to the proof of the following statement:

\begin{thm}\label{thm:ContractingGP}
Let $\Gamma$ be a finite graph and $\mathcal{G}=\{G_u \mid u \in V(\Gamma)\}$ a collection of finitely generated groups. For every $u \in V(\Gamma)$, fix a finite generating set $S_u \subset G_u$. Set $S:= \bigcup_{u \in V(\Gamma)} S_u$. If the essential support of an element $g \in \Gamma \mathcal{G}$ is neither complete nor contained in a large join, then it is contracting in $\mathrm{Cay}(\Gamma \mathcal{G},S)$.
\end{thm}

\noindent
The connection between graph products and quasi-median graphs, which will allow us to apply Theorem~\ref{thm:Contracting}, is made by the following:

\begin{thm}[{\cite[Proposition~8.2, Corollary~8.7]{QM}}]
Let $\Gamma$ be a graph and $\mathcal{G}$ a collection of groups indexed by $V(\Gamma)$. The Cayley graph of $\Gamma \mathcal{G}$ with respect to the generating set consisting of all non-trivial elements in the vertex-groups, 
$$\mathrm{QM}(\Gamma, \mathcal{G}):= \mathrm{Cay} \left( \Gamma \mathcal{G}, \bigcup\limits_{G \in \mathcal{G}} G \backslash \{1 \} \right),$$ 
is a quasi-median graph of cubical dimension $\mathrm{clique}(\Gamma)= \max \{ \# V(\Lambda ) \mid \Lambda \subset \Gamma \ \text{clique} \}$. 
\end{thm}

\noindent
The graph product $\Gamma \mathcal{G}$ acts naturally by isometries on $\mathrm{QM}(\Gamma, \mathcal{G})$ via left-multiplication; as Cayley graph, each edge of $\mathrm{QM}(\Gamma, \mathcal{G})$ is labelled by a generator, but also by a vertex of $\Gamma$ (corresponding to the vertex-group containing the generator). If two edges of $\mathrm{QM}(\Gamma, \mathcal{G})$ belong to the same hyperplane, they must be labelled by the same vertex of $\Gamma$ (see \cite[Lemma~8.9]{QM}), which implies that the hyperplanes of $\mathrm{QM}(\Gamma, \mathcal{G})$ are also naturally labelled by vertices of $\Gamma$. An easy observation that will be needed later is \cite[Lemma~8.12]{QM}, namely:

\begin{lemma}\label{lem:LabelHyp}
Let $\Gamma$ be a simplicial graph and $\mathcal{G}$ a collection of groups indexed by $V(\Gamma)$. Two transverse hyperplanes in $\mathrm{QM}(\Gamma, \mathcal{G})$ are labelled by adjacent vertices of $\Gamma$. 
\end{lemma}

\noindent
The construction of the quasi-median graph $\mathrm{QM}(\Gamma, \mathcal{G})$ leads to the following description of its geodesics \cite[Lemma~8.3]{QM}:

\begin{lemma}\label{lem:DistInX}
Let $\Gamma$ be a graph and $\mathcal{G}$ be a collection of groups indexed by $V(\Gamma)$. Fix two elements $g,h \in \Gamma \mathcal{G}$ and write $g^{-1}h$ as a graphically reduced word $u_1 \cdots u_n$. Then the sequence of vertices $$g,gu_1,gu_1u_2, \ldots, gu_1 \cdots u_n=h$$ defines a geodesic between $g$ and $h$ in $\mathrm{QM}(\Gamma, \mathcal{G})$. Conversely, any geodesic between $g$ and $h$ is labelled by a graphically reduced word representing $g^{-1}h$.
\end{lemma}

\noindent
The cliques and prisms of $\mathrm{QM}(\Gamma, \mathcal{G})$ are described as follows (see \cite[Lemma~8.6 and Corollary~8.7]{QM} or \cite[Lemmas~2.4 and~2.6]{MR4295519}):

\begin{lemma}\label{lem:QMcliques}
Let $\Gamma$ be a simplicial graph and $\mathcal{G}$ a collection of groups indexed by $V(\Gamma)$. The cliques of $\mathrm{QM}(\Gamma, \mathcal{G})$ coincide with the cosets of vertex-groups. That is, every clique $C$ corresponds to $h \langle u \rangle$, for some $h \in \Gamma \mathcal{G}$ and $u \in V(\Gamma)$.
\end{lemma}

\begin{lemma}\label{lem:QMprisms}
Let $\Gamma$ be a simplicial graph and $\mathcal{G}$ a collection of groups indexed by $V(\Gamma)$. The prisms of $\mathrm{QM}(\Gamma, \mathcal{G})$ coincide with the cosets of the $\langle \Lambda \rangle$, where $\Lambda \subset \Gamma$ is a complete subgraph.
\end{lemma}

\noindent
We are now ready to prove the main theorem of this section.

\begin{proof}[Proof of Theorem~\ref{thm:ContractingGP}.]
We start by endowing $\mathrm{QM}(\Gamma, \mathcal{G})$ with a system of metrics. Let $C$ be a clique of $\mathrm{QM}(\Gamma, \mathcal{G})$. By Lemma \ref{lem:QMcliques}, each element in $C$ can be uniquely written as a (graphically) reduced word of the form $h w$, where $h \in \Gamma \mathcal{G}$ is fixed, and $w \in \langle u \rangle$ for a fixed vertex $u \in \Gamma$.  
  Define the local metric
$$\delta_C : (hx,hy) \mapsto d_{S_u} (x,y),$$
where $d_{S_u}$ is the word metric over $\langle u \rangle$ given by the generating set $S_u$. By \cite[Claim~8.24]{QM}, $\mathscr{C}:= \{ (C,\delta_C) \mid C \text{ clique}\}$ is a coherent and $\Gamma \mathcal{G}$-invariant system of metrics, and the metric $\delta$ coincides with the word metric given by the generating set~$S$.

\medskip \noindent
We will prove that $g$ skewers a pair of well-separated hyperplanes relative to $\mathscr{C}$. According to Theorem~\ref{thm:Contracting}, this will suffice to conclude that $g$ is contracting in $\mathrm{Cay}(\Gamma \mathcal{G},S)$.

\medskip \noindent
Let $\Lambda$ denote the essential support of $g$. Up to conjugating $g$, we can assume that $g$ is cyclically graphically reduced and that $g \in \langle \Lambda \rangle$. We may also assume that $\Lambda$ is not a join. Suppose $\Lambda$ is a join $\Lambda_1 \ast \cdots \ast \Lambda_n$, with $\Lambda_1, \ldots, \Lambda_n$ not joins. If $n \geq 2$, then since by hypothesis the essential support of $g$ is neither complete nor contained in a large join, up to reindexing, $\Lambda_2, \ldots, \Lambda_n$ are complete subgraphs labelled by finite groups. A sufficiently high power of $g$ then has its support in $\Lambda_1$, which is not complete. We can therefore assume that the essential support $\Lambda$ of $g$ is not a join.

\medskip \noindent
According to \cite[Proposition~4.24]{AutGP}, fixing a geodesic $[1,g]$ between the vertices $1$ and $g$, the concatenation
$$\gamma:= \cdots \cup g^{-1}[1,g] \cup [1,g] \cup g [1,g] \cup \cdots$$
defines a geodesic in the subgraph $\langle \Lambda \rangle$, and a fortiori in $\mathrm{QM}(\Gamma, \mathcal{G})$; also, for every hyperplane $J$ crossing $\gamma$, up to replacing $g$ with one of its powers, $J$ and $gJ$ are \emph{strongly separated} (i.e.\ there doesn't exist a third hyperplane transverse to both) in $\langle \Lambda \rangle$. To obtain the theorem it suffices to verify that $J$ and $g^4J$ are well-separated relative to $\mathscr{C}$. 

\medskip \noindent
Let $K$ be a hyperplane transverse to both $J$ and $g^4J$. Since $\Lambda$ is the essential support of $g$, for every vertex $u \in \Lambda$ there exists a syllable of $g$ that belongs to $\langle u \rangle$, hence an edge of the path $g^2[1,g]$ labelled by $u$. Let $H$ denote the hyperplane containing this edge. Because $g^2[1,g]$ lies between $gJ$ and $g^3J$, and that $J$ and $gJ$, as well as $g^3J$ and $g^4J$, are strongly separated in $\langle \Lambda \rangle$, it follows that $H$ is not transverse to $J$ neither to $g^4J$ (otherwise, it would be transverse to both $J$ and $gJ$ or to both $g^3J$ and $g^4J$). In other words, $H$ separates $J$ and $g^4J$. This implies that $H$ is transverse to $K$. According to Lemma~\ref{lem:LabelHyp}, $K$ is labelled by a vertex adjacent to $u$. Since this is true for every vertex of $\Lambda$, we conclude that a hyperplane transverse to both $J$ and $g^4J$ is labelled by a vertex of $\mathrm{link}(\Lambda)$. 

\medskip \noindent
Now let $N(J)$ denote the \emph{carrier} of $J$, that is, is the subgraph of $\mathrm{QM}(\Gamma, \mathcal{G})$ generated (or induced) by $J$.
Let $P \subset N(J)$ denote the projection of $N(g^4J)$ on $N(J)$. According to \cite[Proposition~2.33]{QM}, the hyperplanes crossing $P$ are exactly the hyperplanes transverse to both $J$ and $g^4J$. It follows from our previous observation that all the edges of $P$ are labelled by vertices in $\mathrm{link}(\Lambda)$, which amounts to saying that $P$ is contained in a coset of $\langle \mathrm{link}(\Lambda) \rangle$. But, as a consequence of our assumptions on $\Lambda$, $\mathrm{link}(\Lambda)$ must be a complete graph all of whose vertices are labelled by finite groups. Hence
$$\sum\limits_{K \text{ transverse to $J$ and $g^4J$}} \mathrm{thick}(K) \leq \mathrm{clique}(\Gamma) \cdot \max \{ |\langle u \rangle|, \ \langle u \rangle \text{ finite}\}.$$
Thus, $J$ and $g^4J$ are well-separated relative to $\mathscr{C}$, and we conclude from Theorem~\ref{thm:Contracting} that $g$ is contracting in $\mathrm{Cay}(\Gamma \mathcal{G}, S)$. 
\end{proof}

\noindent
As a by-product of Theorem~\ref{thm:ContractingGP}, since contracting geodesics turn out to be Morse \cite[Lemma~3.3]{MR3175245}, we can deduce the following characterisation of Morse elements in finitely generated graph products. Recall that, given a finitely generated group $G$, an infinite-order element $g \in G$ is \emph{Morse} if, for all $A>0$ and $B \geq 0$, there exists some $C \geq 0$ such that every $(A,B)$-quasi-geodesic connecting two points in $\langle g \rangle$ stays in the $C$-neighbourhood of $\langle g \rangle$. 

\begin{cor}\label{cor:Morse}
Let $\Gamma$ be a finite graph and $\mathcal{G}=\{G_u \mid u \in V(\Gamma)\}$ a collection of finitely generated groups. An element $g \in \Gamma \mathcal{G}$ is Morse if and only if one of the following conditions hold:
\begin{itemize}
	\item the essential support of $g$ is neither complete nor contained in a large join;
	\item there exists a vertex $u \in \Gamma$ whose link is complete and labelled by only finite groups, and $g$ decomposes as $ab$ for some Morse element $a \in \langle u \rangle$ and some $b \in \mathrm{link}(u)$. 
\end{itemize}
\end{cor}

\begin{proof}
Let $\Lambda$ denote the essential support of $g$, and decompose it as a join $\Lambda_1 \ast \cdots \ast \Lambda_n$ such that $\Lambda_1, \ldots, \Lambda_n$ are not joins themselves. First we eliminate two clear cases when $g$ cannot be Morse: 
\begin{itemize}
	\item If two of the $\Lambda_i$ are infinite, so either contain at least two vertices or are a single vertex labelled by an infinite group, then $g$ is not Morse as it is contained in a subgroup that splits as a product of two infinite groups. 
	\item If all the $\Lambda_i$ are single vertices labelled by finite groups, then $g$ is a finite-order element, and so cannot be Morse. 
\end{itemize}
From now on, up to reindexing of subgraphs in $\Lambda$, assume that $\Lambda_1$ either contains at least two vertices or is a single vertex labelled by an infinite group, and assume that $\Lambda_2, \ldots, \Lambda_n$ are single vertices labelled by finite groups (with the possibility that $n=1$). 

\medskip \noindent
Assume first that $\Lambda_1$ contains at least two vertices. If $\Lambda$ is not contained in a large join, then Theorem~\ref{thm:ContractingGP} applies and shows that $g$ is contracting, and therefore Morse. Otherwise, $g$ belongs to a subgroup that splits as a product of two infinite subgroups, and so cannot be Morse. 

\medskip \noindent
Next, assume that $\Lambda_1$ is a single vertex, say $u$, labelled by an infinite group. Notice that $g$ decomposes uniquely as a product $ab$ with $a \in \langle u \rangle$ and $b \in \langle \Lambda_2 \cup \cdots \cup \Lambda_n \rangle \leq \langle \mathrm{link}(u) \rangle$. Since $\Lambda_2, \ldots, \Lambda_n$ are single vertices labelled by finite groups, $b$ has finite order, so $g$ has a power, say $g^N$, that belongs to $\langle u\rangle$. If $\mathrm{link}(u)$ is not complete or if it is complete but has at least one vertex labelled by an infinite group, then $g^N$ is not Morse since it belongs to a subgroup that splits a product of two infinite subgroups. A fortiori, $g^N$ is not Morse either. Otherwise, if $\mathrm{link}(u)$ is complete and all its vertices are labelled by finite groups, then 
$$\Gamma \mathcal{G}= \langle \mathrm{star}(u) \rangle \underset{\langle \mathrm{link}(u) \rangle}{\ast} \langle \Gamma \backslash \{u\} \rangle$$
is a splitting over a finite subgroup. Clearly, $g^N$ is Morse in $\Gamma \mathcal{G}$ if and only if it is Morse in $\langle \mathrm{star}(u) \rangle$, or equivalently in $\langle u \rangle$. Since $g^N$ is also a power of $a$, this amounts to requiring $a$ to be a Morse element in $\langle a \rangle$. 
\end{proof}

\section{Application to periagroups}\label{section:periagroups}

\subsection{Mediangle geometry}\label{section:MediangleGeometry}

\noindent
We first record the preliminaries about mediangle graphs necessary for Section~\ref{section:ContractingInPeriagroups}. 
\paragraph{Mediangle graphs.}  Our definition of \emph{mediangle graphs} is a variation of the definition of quasi-median graphs as weakly modular graphs with no induced copy of $K_{3,2}$ and $K_4^-$; see Section~\ref{section:QM}. Roughly speaking, we replace the $4$-cycles in Proposition \ref{prop:QMweaklyModular} with convex even cycles. Recall that a \emph{convex} subgraph $Y$ of a graph $X$ is one that contains all the geodesics between its vertices, that is, $I(x,y) \subseteq Y$ for all $x,y \in V(Y)$.

\begin{definition}\label{def:Mediangle}
A connected graph $X$ is \emph{mediangle} if the following are satisfied:
\begin{description}
	\item[(Triangle Condition)] For all vertices $o,x,y \in X$ satisfying $d(o,x)=d(o,y)$ and $d(x,y)=1$, there exists a common neighbour $z \in X$ of $x,y$ such that $z \in I(o,x) \cap I(o,y)$.
	\item[(Intersection of Triangles)] $X$ does not contain an induced copy of $K_4^-$.
	\item[(Cycle Condition)] For all vertices $o ,x,y,z \in X$ satisfying $d(o,x)=d(o,y)=d(o,z)-1$ and $d(x,z)=d(y,z)=1$, there exists a convex cycle of even length that contains the edges $[z,x],[z,y]$ and such that the vertex opposite to $z$ belongs to $I(o,x) \cap I(o,y)$.
	\item[(Intersection of Even Cycles)] The intersection between any two convex cycles of even lengths contains at most one edge. 
\end{description}
\end{definition}

\noindent
As mentioned in Section~\ref{section:Paraclique}, examples of mediangle graphs include (quasi-)median graphs, Cayley graphs of Coxeter graphs (e.g.\ the one-skeleton of the regular tiling of the plane by hexagons), one-skeleta of some small cancellation polygonal complexes, and hypercellular graphs. We refer to \cite{Mediangle} for more details. 

\medskip \noindent
In \cite[Section 3.3]{Mediangle}, hyperplanes in mediangle graphs are defined as the transitive closure of the relation that assumes two edges in a $3$-cycle or opposite in a convex even cycle are equivalent. Lemma \ref{lem:hyper} shows this definition is equivalent to Definition \ref{def:hyper}, which uses parallelism; this will allow us to apply the results from \cite{Mediangle} safely. 

\begin{lemma}\label{lem:hyper}
Let $X$ be a mediangle graph. Two edges $e$ and $f$ belong to the same hyperplane if and only if there exists a sequence of edges
$$a_0=e, \ a_1, \ldots,  \ a_{n-1}, \ a_n= f$$
such that, for every $0 \leq i \leq n-1$, $a_i$ and $a_{i+1}$ either belong to a common $3$-cycle or are opposite edges in a convex even cycle.
\end{lemma}

\begin{proof}
Let $C$ denote the clique containing $f$. As a consequence of \cite[Corollary~3.18]{Mediangle}, there exists a sequence of edges as described by our lemma if and only if the projection of $e$ on $C$ is an edge. But this also amounts to saying that the cliques containing $e$ and $f$ are parallel, i.e.\ $e$ and $f$ belong to the same hyperplane. 
\end{proof}

\noindent
Let us verify that, as claimed before, mediangle graphs are indeed paraclique.

\begin{prop}\label{prop:MediangleParaclique}
Mediangle graphs are paraclique.
\end{prop}

\noindent
We start by proving the following observation, which will be also useful later:

\begin{lemma}\label{lem:SameHypIso}
Let $X$ be a mediangle graph and $Y,Z \subset X$ two gated subgraphs. If $Y$ and $Z$ are crossed by exactly the same hyperplanes, then the projection of $Y$ on $Z$ induces an isometry $Y \to Z$. 
\end{lemma}

\begin{proof}
Let $a,b \in Y$ be two vertices. According to \cite[Corollary~3.18]{Mediangle}, the hyperplanes separating the projections of $a$ and $b$ on $Z$ are exactly the hyperplanes separating $a$ and $b$ that cross $Z$. But the hyperplanes separating $a$ and $b$ necessarily cross $Y$, and consequently $Z$. Thus, $a$ and $b$ are separated by the same hyperplanes as their projections on $Z$. In particular, the same number of hyperplanes separates $a$ and $b$ and their projections on $Z$, which proves that the projection of $Y$ on $Z$ is an isometric embedding. 

\medskip
\noindent It remains to verify that every vertex of $Z$ is the projection of some vertex of $Y$. Given a vertex $z \in Z$, let $y$ denote its projection on $Y$. According to \cite[Lemma~3.17]{Mediangle}, the hyperplanes separating $z$ from $y$ separate $z$ from $Z$. Consequently, the hyperplanes separating $y$ and $z$ coincide with the hyperplanes separating $Y$ and $Z$. Thus $y$ and $z$ minimise the distance between $y$ and $z$, so $z$ must be the projection of $y$ on $Z$.
\end{proof}

\begin{proof}[Proof of Proposition~\ref{prop:MediangleParaclique}.]
Cliques in mediangle graphs are gated according to \cite[Lemma~3.12]{Mediangle}. Now, let $C_1,C_2$ be two cliques such that the projection of $C_1$ on $C_2$ is not reduced to a single vertex. It follows from \cite[Corollary~3.18]{Mediangle} that $C_1$ and $C_2$ belong to the same hyperplane, and then Lemma~\ref{lem:SameHypIso} implies that the projection of $C_1$ on $C_2$ is bijective. 
\end{proof}

\noindent
In the sequel, we will need to know which isometries of mediangle graphs admit an axis (i.e.\ a bi-infinite geodesic line on which they act as a non-trivial translation). This is not always the case among isometries with unbounded orbits. Our next statement provides a sufficient condition for admitting an axis. The criterion is far from optimal, but we focus on the specific case that will be relevant to us.

\begin{prop}\label{prop:AxisSometimes}
Let $X$ be a mediangle graph and $g \in \mathrm{Isom}(X)$ some isometry. Assume that every vertex of $X$ belongs to only finitely many cliques and that $g$ skewers a pair of hyperplanes $(J_1,J_2)$ satisfying:
\begin{itemize}
	\item[(i)] every hyperplane transverse to both $J_1$ and $J_2$ delimits only finitely many sectors;
	\item[(ii)] there are only finitely many hyperplanes transverse to both $J_1$ and $J_2$.
\end{itemize}
Then some power of $g$ admits an axis in $X$.
\end{prop}

\noindent
We start by proving the following elementary observation:

\begin{lemma}\label{lem:UnionSectorsConvex}
Let $X$ be a mediangle graph and $J$ a hyperplane. A union of sectors delimited by $J$ is convex in $X$.
\end{lemma}

\begin{proof}
We already know from Proposition~\ref{prop:HypCliqueGated} that sectors delimited by $J$ are convex. Therefore, it suffices to show that, given two vertices $x,y \in X$ that belong to two distinct sectors $A,B$ delimited by $J$, every geodesic connecting $x$ to $y$ is contained in $A \cup B$. But we know from Proposition~\ref{prop:ParaGeod} that any such geodesic cannot cross $J$ twice, so it cannot pass through a third sector. 
\end{proof}

\begin{cor}\label{cor:SectorConvexHull}
Let $X$ be a mediangle graph and $S \subset X$ a subgraph. A sector intersects the convex hull of $S$ if and only if it intersects $S$. 
\end{cor}

\begin{proof}
It is clear that a sector intersecting $S$ also intersects the convex hull of $S$. Conversely, if a sector $D$ does not intersect $S$, then $S$ is contained in the complement $D^c$ of $D$, which is a union of sectors. Therefore, Lemma~\ref{lem:UnionSectorsConvex} implies that $D^c$ is convex, and we conclude that the convex hull of $S$ must be contained in $D^c$, and consequently disjoint from $D$. 
\end{proof}

\begin{proof}[Proof of Proposition~\ref{prop:AxisSometimes}.]
Fix an arbitrary vertex $o \in X$ and let $A$ denote the convex hull of the orbit $\langle g \rangle \cdot o$. We claim that $A$ is a locally finite quasi-line on which $g$ acts. This will imply  that some power of $g$ acts as a non-trivial translation on some bi-infinite geodesic contained in $A$; see for instance \cite[Theorem~A.2]{TransLengthQM}. 

\medskip \noindent
Up to replacing $g$ with some of its powers, we can assume that $g^nJ_2$ separates $g^nJ_1$ and $g^{n+1}J_1$ for every $n \in \mathbb{Z}$. If $A$ is not locally finite, then, because every vertex of $X$ belongs to only finitely many cliques, $A$ must contain an infinite complete subgraph. Let $J$ denote the hyperplane containing this subgraph. Clearly, $A$ intersects infinitely many sectors delimited by $J$, so Corollary~\ref{cor:SectorConvexHull} implies that $\langle g \rangle \cdot o$ intersects infinitely many sectors delimited by $J$. But, if $o$ lies between $g^sJ_1$ and $g^sJ_1$ for some $s \in \mathbb{Z}$, then $g^ko$ lies between $g^{s+k}J_1$ and $g^{s+k}J_1$, which implies that $J$ must be transverse to $g^iJ_1$ and $g^{i+1}J_1$ (and a fortiori $g^iJ_2$) for infinitely many $i \in \mathbb{Z}$. But, due to the fact that $J$ delimits infinitely many sectors, there cannot be an index $i \in \mathbb{Z}$ such that $g^i J_1$ and $g^iJ_2$ are both crossed by $J$. We conclude that $A$ must be locally finite.

\medskip \noindent
Then, let us verify that $A$ is contained in a neighbourhood of $\langle g \rangle \cdot o$. Fix an arbitrary vertex $x \in A$. Up to replacing our vertices and hyperplanes by $\langle g \rangle$-translates, assume for ease of notation that $o$ and $x$ both lie between $J_1$ and $gJ_1$. We claim that $d(o,x) \leq 2N+d(J_1,g^3J_1)$ where $N$ denotes the number of hyperplanes transverse to both $J_1$ and $J_2$. Among the hyperplanes separating $o$ and $x$, we know that at most $2N$ of them may be transverse to both $g^{-1}J_1$ and $J_1$ or to both $gJ_1$ and $g^2J_1$. Also, at most $d(g^{-1}J_1,g^2J_1)=d(J_1,g^3J_1)$ of them may separate $g^{-1}J_1$ and $g^2J_1$. The possible remaining hyperplanes separate $x$ from $g^{-1}J_1$ and $g^2J_1$, and consequently from $\langle g \rangle \cdot o$, which is impossible since $x$ belongs to the convex hull of $\langle g \rangle \cdot o$. 
\end{proof}

\subsection{Periagroups}\label{sec:periagroups} Periagroups were first defined in \cite{Mediangle}, as follows.

\begin{definition}\label{def:periagroup}[\cite{Mediangle}]
Let $\Gamma$ be a graph, $\lambda : E(\Gamma) \to \mathbb{N}_{\geq 2}$ an edge labelling, and $\mathcal{G}=\{G_u \mid u \in V(\Gamma) \}$ a collection of non-trivial groups. We assume that $\lambda(\{u,v\})=2$ for any edge $\{u,v \} \in E(\Gamma)$ satisfying $\max(|G_u|, |G_v|)\geq 3$  (so $\lambda(\{u,v\})>2$ means $|G_u|=|G_v|=2$). The \emph{periagroup} $\Pi(\Gamma, \lambda, \mathcal{G})$ admits
$$\left\langle G_u, \ u \in V(\Gamma) \mid \langle G_u,G_v \rangle^{\lambda(\{u,v\})} = \langle G_v,G_u \rangle^{\lambda(\{u,v\})}, \ \{u,v\} \in E(\Gamma) \right\rangle$$
as a relative presentation. Here, $\langle a,b \rangle^k$ refers to the word obtained from $ababab \cdots$ by keeping only the first $k$ letters; and $\langle G_u, G_v \rangle^k = \langle G_v,G_u \rangle^k$ is a shorthand for: $\langle a,b \rangle^k= \langle b,a \rangle^k$ for all non-trivial $a \in G_u$ and $b \in G_v$. 
\end{definition}

\begin{ex}\label{ex:periagroup}
Let $\Gamma$ be the labelled graph with vertices $v_1, v_2, v_3, v_4$ and corresponding vertex groups $G_1, G_2, G_3, G_4$, where $G_i=\langle x_i\mid x_i^2=1\rangle$, $i=1,2$, $G_3=\langle x_3 \mid x_3^3=1 \rangle$ and $G_4=\langle x_4, x_5 \mid - \rangle \cong F_2$, as pictured:
 \[
  \xymatrix{\stackrel{G_1}{\bullet}\ar@{-}[r]^5 &\stackrel{G_2}{\bullet} \ar@{-}[r]^2&\stackrel{G_3}{\bullet}\ar@{-}[r]^2 &\stackrel{G_4}{\bullet} }.
 \]
 The periagroup $\Pi(\Gamma, \lambda, \mathcal{G})$ based on $\Gamma$ has the presentation $$\langle x_1,x_2,x_3,x_4 \mid \langle x_1,x_2\rangle^5=\langle x_2,x_1\rangle^5, [x_2,x_3]=[x_3,x_4]=[x_3, x_5]=1, x_1^2=x_2^2=x_3^3=1 \rangle.$$

 \end{ex}
\noindent
Periagroups of cyclic groups of order two coincide with Coxeter groups; and, if $\lambda \equiv 2$, all the relations are commutations and one retrieves graph products of groups. Thus, periagroups can be thought of as an interpolation between Coxeter groups and graph products of groups. Periagroups of cyclic groups, a generalisation of Coxeter groups and right-angled Artin groups, aslo appear in \cite{MR1076077}, and are studied geometrically in \cite{MR4735237} under the name \emph{Dyer groups}. 

\medskip \noindent
Let $\Pi = \Pi(\Gamma, \lambda, \mathcal{G})$ be a periagroup. A \emph{word} in $\Pi$ is a product $g_1 \cdots g_n$ with $n \geq 0$ and where, for every $1 \leq i \leq n$, $g_i$ belongs to $G_i$ for some $G_i \in \mathcal{G}$; the $g_i$'s are the \emph{syllables} of the word, and $n$ is the \emph{length} of the word. Clearly, the following operations on a word do not modify the element of $\Pi$ it represents:
\begin{description}
	\item[(reduction)] remove the syllable $g_i$ if $g_i=1$;
	\item[(fusion)] if $g_i,g_{i+1} \in G$ for some $G \in \mathcal{G}$, replace the two syllables $g_i$ and $g_{i+1}$ by the single syllable $g_ig_{i+1} \in G$;
	\item[(dihedral relation)] if there exist $\{u,v\} \in E(\Gamma)$ such that $g_i\cdots g_{i+\lambda(u,v)-1}= \langle a,b \rangle^{\lambda(u,v)}$ for some $a \in G_u$, $b \in G_v$, then replace this subword with $\langle b,a \rangle^{\lambda(u,v)}$.
\end{description}
A word is \emph{graphically reduced} if its length cannot be shortened by applying these elementary moves. Clearly, every element of $\Pi$ can be represented by a graphically reduced word. Moreover, according to \cite[Proposition~5.8]{Mediangle}, such a word is unique up to applying dihedral relations.

\paragraph{Mediangle geometry of periagroups.} Generalising graph products, \cite[Theorem~1.1]{Mediangle} shows how periagroups and mediangle graphs are related. In particular:

\begin{thm}
Let $\Pi = \Pi(\Gamma, \lambda,\mathcal{G})$ be a periagroup. The Cayley graph 
$$\mathrm{M}(\Gamma, \lambda, \mathcal{G}):= \mathrm{Cay} \left( \Pi, \bigcup\limits_{G \in \mathcal{G}} G \backslash \{1 \} \right)$$ 
is a mediangle graph. 
\end{thm}

\noindent
The characterisation of graphically reduced words as minimal words given by \cite[Proposition~5.8]{Mediangle} immediately implies that:

\begin{prop}\label{prop:GeodPeriagroups}
Let $\Pi = \Pi(\Gamma, \lambda,\mathcal{G})$ be a periagroup. Fix two elements $g,h \in \Pi$ and write $g^{-1}h$ as a graphically reduced word $u_1 \cdots u_n$. Then the sequence of vertices $$g,gu_1,gu_1u_2, \ldots, gu_1 \cdots u_n=h$$ defines a geodesic between $g$ and $h$ in $\mathrm{M}(\Gamma, \lambda,\mathcal{G})$. Conversely, any geodesic between $g$ and $h$ is labelled by a graphically reduced word representing $g^{-1}h$.
\end{prop}

\noindent
In mediangle graphs, it is possible to define \emph{angles} between transverse hyperplanes \cite{Mediangle}. Then, a hyperplane $J$ is \emph{right} whenever $\measuredangle(J,H) = \pi/2$ for every hyperplane $H$ transverse to $J$. Alternatively, a hyperplane is right if the only convex cycles it crosses have length four. Contrary to hyperplanes in quasi-median graphs, carriers of hyperplanes may not be gated in mediangle graphs. However, \cite[Lemma~2.26]{RARotation} shows that this does not happen for right hyperplanes:

\begin{lemma}\label{lem:RightHypGated}
In mediangle graphs, carriers of right hyperplanes are gated. 
\end{lemma}

\noindent
A remarkable property satisfied by right hyperplanes is given by \cite[Lemma~2.25]{RARotation} below. Recall that, given a group $G$ acting on a mediangle graph $X$ and a hyperplane $J$, the \emph{rotative-stabiliser} $\mathrm{stab}_\circlearrowright(J)$ refers to the subgroup of $\mathrm{stab}(J)$ that stabilises (setwise) each clique in $J$. 

\begin{lemma}\label{lem:TransverseThenStab}
Let $G$ be a group acting on a mediangle graph $X$, and let $J,H$ be two transverse hyperplanes. If $J$ is right, then $\mathrm{stab}_\circlearrowright(H)$ stabilises $J$. 
\end{lemma}

\medskip \noindent
The structure of right hyperplanes in mediangle graphs associated to periagroups is described by the following statement:

\begin{prop}\label{prop:RightHyp}
Let $\Pi= \Pi(\Gamma,\lambda, \mathcal{G})$ be a periagroup. The following statements hold in $M(\Gamma, \lambda, \mathcal{G})$. 
\begin{itemize}
	\item[(i)] For every vertex $u \in \Gamma$, the hyperplane $J_u$ containing the clique $\langle u \rangle$ is right if and only if all the edges of $\Gamma$ containing $u$ have label $2$.
	\item[(ii)] If $J_u$ is right, its carrier is $\langle \mathrm{star}(u) \rangle$. As a consequence, $\mathrm{stab}(J_u) = \langle \mathrm{star}(u) \rangle$ and $\mathrm{stab}_\circlearrowright(J)= \langle u \rangle$. 
	\item[(iii)] All the edges of a right hyperplane are labelled by the same vertex of $\Gamma$. The labels of two transverse right hyperplanes are adjacent in $\Gamma$.
	\item[(iv)] Two right hyperplanes whose carriers intersect are transverse if and only if their labels are adjacent.
\end{itemize}
\end{prop}

\begin{proof}
Item $(i)$ is given by \cite[Lemma~2.28]{RARotation}. First part of $(ii)$ is given by \cite[Lemma~2.29]{RARotation}. It clearly implies that $\mathrm{stab}(J_u)= \langle \mathrm{star}(u) \rangle$. Since $\langle \mathrm{star}(u) \rangle$ decomposes as a product $\langle \mathrm{link}(u) \rangle \times \langle u \rangle$, the cliques of $J$ must be given by the cosets $g \langle u \rangle$ with $g \in \langle \mathrm{link}(u) \rangle$. Hence
$$\mathrm{stab}_\circlearrowright(J)= \bigcap\limits_{g \in \langle \mathrm{link}(u) \rangle} g \langle u \rangle g^{-1}= \bigcap\limits_{g \in \langle \mathrm{link}(u) \rangle} \langle u \rangle = \langle u \rangle,$$
proving item $(ii)$. Item $(iii)$ is given by \cite[Lemma~2.30]{RARotation}.

\medskip \noindent
Finally, consider two right hyperplanes whose carriers intersect. Up to translating by an element of $\Pi$, we can assume that our two hyperplanes are $J_u$ and $J_v$ for some vertices $u,v \in \Gamma$ only contained in edges labelled $2$. If $u$ and $v$ are not adjacent, then $J_u$ and $J_v$ are not transverse according to $(iii)$. Otherwise, if $u$ and $v$ are adjacent, the cliques $\langle u \rangle$ and $\langle v \rangle$ span a prism, namely $\langle u,v \rangle = \langle u \rangle \times \langle v \rangle$, which shows that $J_u$ and $J_v$ are transverse. This proves $(iv)$.  
\end{proof}

\noindent
We record a final observation about hyperplanes in mediangle graphs of periagroups:

\begin{lemma}\label{lem:ParallelEdgesSameOrbit}
Let $\Pi = \Pi(\Gamma,\lambda,\mathcal{G})$ be a periagroup. Let $J$ be a hyperplane in $M(\Gamma, \lambda,\mathcal{G})$, $C\subset J$ a clique, and $e \subset J$ an edge. There exists $g \in \mathrm{stab}(J)$ such that $ge$ is the projection of $e$ on $C$.
\end{lemma}

\begin{proof}
As a consequence of Claim~\ref{claim:SequenceCycles} below, we can assume without loss of generality that there exists a convex even cycle whose opposite edges are $e$ and the projection of $e$ on $C$. The convex even cycles of $M(\Gamma,\lambda,\mathcal{G})$ are completely characterised by \cite[Claim~4.13]{Mediangle}; namely, they are $\Pi$-translates of the cycles of the form 
$$1, \ a, \ ab, \ aba, \ldots, \ \langle a,b \rangle^k = \langle b,a \rangle^k, \ldots, \ bab, \ ba, \ b$$
where $a$ and $b$ are two generators that belong to two vertex-groups corresponding to two vertices of $\Gamma$ connected by an edge labelled by $k$. In other words, they correspond to $2k$-cycles on which a dihedral subgroup $\langle a,b \rangle$ of size $2k$ acts canonically. Since such a dihedral group contains a reflection that inverts a given pair of opposite edges, the desired conclusion follows.

\begin{claim}\label{claim:SequenceCycles}
Let $X$ be a mediangle graph, $C \subset X$ a clique, and $e$ an edge contained in the same hyperplane as $C$. There exists a sequence of edges
$$a_0=e, \ a_1, \ldots, \ a_{n-1}, \ a_n$$
such that $a_n$ is the projection of $e$ on $C$ such that $a_i,a_{i+1}$ are opposite edges in some convex even cycle for every $0 \leq i \leq n-1$.
\end{claim}

\noindent
We argue by induction over the distance between $e$ and $C$. Let $x,y \in X$ denote the endpoints of $e$ and let $x',y' \in C$ denote their projections on $C$. Notice that $d(x,x')=d(y,y')$. This is due to the fact that the edges $[x,y]$ and $[x',y']$ belongs to the same hyperplane, which implies that the hyperplanes separating $x$ and $x'$, or equivalently $y$ and $y'$, coincide with the hyperplanes separating $[x,y]$ and $C$, and consequently that the same number of hyperplanes separates $x$ from $x'$ and $y$ from $y'$. If $d(x,x')=0$, then $[x,y]$ is contained in $C$ and there is nothing to prove. 

\medskip \noindent
\begin{minipage}{0.3\linewidth}
\includegraphics[width=0.8\linewidth]{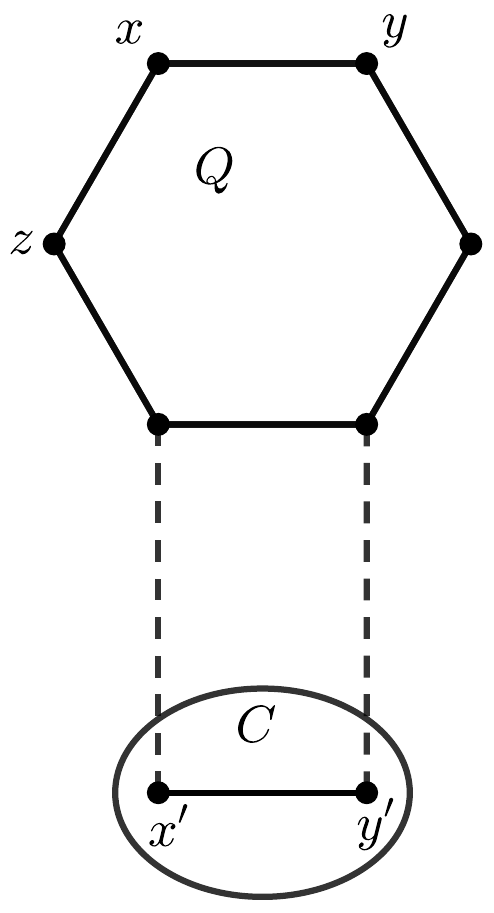}
\end{minipage}
\begin{minipage}{0.69\linewidth}
Otherwise, we can fix a neighbour $z$ of $x$ that belongs to a geodesic connecting $x$ to $x'$. Notice that
$$d(y',y) = d(x',x)=d(x',z)+1 = d(y',x').$$
Moreover, $d(y',x')= 1+ d(x,x')> d(x,x')$. Therefore, we can apply the cycle condition, i.e.\ there exists a convex even cycle $Q$ spanned by the edges $[x,y]$ and $[x,z]$ such that the vertex $v \in Q$ opposite to $x$ belongs to $I(y',y) \cap I(y',z)$. Then the edge of $Q$ opposite to $[x,y]$ is closer to $C$ and has the same projection on $C$ as $[x,y]$. Hence the desired conclusion. 
\end{minipage}
\end{proof}

\paragraph{Parabolic subgroups.} Given a periagroup $\Pi=\Pi(\Gamma,\lambda,\mathcal{G})$, a \emph{standard parabolic subgroup} is a subgroup of the form $\langle \Lambda \rangle$ where $\Lambda \subset \Gamma$ is a subgraph. A \emph{parabolic subgroup} is a subgroup that is conjugate to a standard parabolic subgroup. It turns out that the intersection of two parabolic subgroups is again a parabolic subgroup. More precisely, the proof of \cite[Theorem~1.7]{Mediangle} shows that:

\begin{thm}\label{thm:InterParabolicPeriagroups}
Let $\Pi = \Pi(\Gamma, \lambda, \mathcal{G})$ be a periagroup. For all $g,h \in \Pi$ and $\Phi,\Psi \subset \Gamma$, 
$$g \langle \Phi \rangle g^{-1} \cap h \langle \Psi \rangle h^{-1}= k \langle \Xi \rangle k^{-1},$$
where $k \in \Pi$ belongs to the projection of $h \langle \Psi \rangle$ on $g \langle \Phi \rangle$ and where $\Xi$ is the subgraph of $\Gamma$ induced by the vertices labelling the edges of this projection. In particular, the intersection of two parabolic subgroups is again a parabolic subgroup.
\end{thm}

\noindent
It is worth mentioning that the projections involved in the theorem are well-defined as cosets of standard parabolic subgroups are gated according to \cite[Corollary~6.6]{Mediangle}. 

\medskip \noindent
We emphasize that, given a periagroup $\Pi = \Pi(\Gamma,\lambda, \mathcal{G})$, there may exist distinct subgraphs $\Lambda_1,\Lambda_2 \subset \Gamma$ such that the parabolic subgroups $\langle \Lambda_1 \rangle$ and $\langle \Lambda_2 \rangle$ are conjugate. This is already the case in the finite dihedral groups $D_{2n}$ with $n$ odd, where the two canonical generators are conjugate. However, the subgraphs $\Lambda_1$ and $\Lambda_2$ cannot be too different from each other: they must have the same size.

\begin{prop}\label{prop:ConjParabolicSize}
Let $\Pi= \Pi(\Gamma, \lambda, \mathcal{G})$ be a periagroup. For all subgraphs $\Phi,\Psi \subset \Gamma$, if $\langle \Phi \rangle$ and $\langle \Psi \rangle$ are conjugate in $\Pi$, then $|\Phi |= |\Psi|$. 
\end{prop}

\noindent
In order to prove the proposition, the following elementary observation will be needed: 

\begin{lemma}\label{lem:HypParbabolicStab}
Let $\Pi = \Pi(\Gamma, \lambda, \mathcal{G})$ be a periagroup, $\Lambda \subset \Gamma$ a subgraph, and $J$ a hyperplane of $M=M(\Gamma,\lambda, \mathcal{G})$. If $J$ crosses the subgraph $\langle \Lambda \rangle$, then $\mathrm{stab}_\circlearrowright (J)$ stabilises $\langle \Lambda \rangle$. 
\end{lemma}

\begin{proof}
Because $J$ crosses $\langle \Lambda \rangle$, there exist an element $g \in \langle \Lambda \rangle$ and a vertex $u \in \Lambda$ such that the clique $g \langle u \rangle$ of $\langle \Lambda \rangle$ is contained in $J$. On the one hand, 
$$\mathrm{stab}_\circlearrowright(J) \leq \mathrm{stab}(g \langle u \rangle) = g \langle u \rangle g^{-1};$$
on the other hand, $g \langle u \rangle g^{-1}$ clearly stabilises $\langle \Lambda \rangle$. Hence the desired conclusion. 
\end{proof}

\begin{proof}[Proof of Proposition~\ref{prop:ConjParabolicSize}.]
Fix a $g \in \Pi$ such that $\langle \Phi \rangle = g \langle \Psi \rangle g^{-1}$. We claim that the gated subgraphs $\langle \Phi \rangle$ and $g \langle \Psi \rangle$ in $M$ are crossed by exactly the same hyperplanes. Indeed, if $J$ were a hyperplane crossing $\langle \Phi \rangle$ but not $g \langle \Psi \rangle$, then $\mathrm{stab}_\circlearrowright(J)$ would stabilise $\langle \Phi \rangle$ (acccording to Lemma~\ref{lem:HypParbabolicStab}) but not $g \langle \Psi \rangle$, contradicting the equality $\langle \Phi \rangle = g \langle \Psi \rangle g^{-1}$. For the same reason, every hyperplane crossing $g \langle \Psi \rangle$ has to cross $\langle \Phi \rangle$ as well. As a consequence of this observation, we can apply Lemma~\ref{lem:SameHypIso} and deduce that $\langle \Phi \rangle$ and $g \langle \Psi \rangle$ are isomorphic graphs. Now, notice that every vertex of $\langle \Phi \rangle$ (resp.\ $g \langle \Psi \rangle$) belongs to exaclty $|\Phi|$ (resp.\ $|\Psi|$) cliques.  Hence the desired equality. 
\end{proof}

\noindent
As a consequence of Proposition~\ref{prop:ConjParabolicSize}, for any (possibly infinite) $\Gamma$, we can define the \emph{height} $h(P)$ of a parabolic subgroup $P \leq \Pi$ as the number of vertices in the subgraph $\Lambda \subset \Gamma$, where $P$ is conjugate to $\langle \Lambda \rangle$. A parabolic subgroup is \emph{of finite type} if it has finite height. 
A property of interest is that parabolic subgroups of finite type satisfy the descending chain condition:

\begin{cor}\label{cor:FiniteType}
Let $\Pi= \Pi(\Gamma,\lambda, \mathcal{G})$ be a periagroup and $P,Q \leq \Pi$ two parabolic subgroups. If $Q \subset P$ then $h(Q) \leq h(P)$. Moreover, if $h(P)$ is finite, then $h(P)=h(Q)$ if and only if $P=Q$. 
\end{cor}

\begin{proof}
Write $P=g \langle \Phi \rangle g^{-1}$ and $Q= h \langle \Psi \rangle h^{-1}$. As a consequence of Theorem~\ref{thm:InterParabolicPeriagroups}, we can write $P \cap Q = Q$ as $k \langle \Xi \rangle k^{-1}$ for some $k \in g \langle \Phi \rangle$ and $\Xi \subset \Phi$. Of course, this inclusion implies that
$$h(Q)= |\Xi| \leq |\Phi| = h(P).$$
If $h(P)=|\Phi|$ is finite and $h(P)=h(Q)$, then $|\Xi|=|\Phi|$ implies that $\Xi = \Phi$. Hence $Q= k \langle \Xi \rangle k^{-1}= g \langle \Phi \rangle g^{-1}=P$. 
\end{proof}

\begin{cor}\label{cor:ChainParabolicTerminate}
In a periagroup $\Pi= \Pi(\Gamma, \lambda, \mathcal{G})$, a decreasing sequence $P_1 \supset P_2 \supset \cdots$ of parabolic subgroups, with $P_1$ of finite type, must terminate.
\end{cor}

\begin{proof}
It follows from Corollary~\ref{cor:FiniteType} that $h(P_1)>h(P_2)>\cdots$, which cannot be infinite since $h(P_1)$ is finite by assumption.
\end{proof}

\subsection{A lemma for Coxeter groups}

\noindent
In this section, we give a preliminary result about Coxeter groups that will be needed later. Before stating it, we introduce some notation.

\begin{definition}\label{def:JoinTwo}
Given two graphs $\Gamma_1,\Gamma_2$ whose edges are labelled by integers $\geq 2$, we denote by $\Gamma_1 \ast_2 \Gamma_2$  the labelled graph obtained from $\Gamma_1 \sqcup \Gamma_2$ by connecting every vertex of $\Gamma_1$ with every vertex of $\Gamma_2$ by an edge labelled $2$. The subgraphs $\Gamma_1$ and $\Gamma_2$ are the \emph{ $\ast_2$-factors} of $\Gamma_1 \ast_2 \Gamma_2$. A labelled graph is \emph{$\ast_2$-irreducible} whenever it does not decompose as a $\ast_2$-product of two non-empty labelled graphs. 
\end{definition}

\noindent
The rest of the section is dedicated to the proof of the following statement. We use a different notation for Coxeter groups compared to previous sections: we write $C(\Gamma)$ for the Coxeter group with defining graph $\Gamma$, instead of $(W,\Gamma)$.

\begin{prop}\label{prop:InterCosetCoxeter}
Let $\Psi$ be a labelled graph and $C(\Psi)$ its corresponding Coxeter group. For $\Lambda_1, \Lambda_2 \subset \Psi$, there exists $g \in C(\Psi)$ such that $g \langle \Lambda_1 \rangle \cap \langle \Lambda_2 \rangle = \emptyset$ if and only if (at least) one $\ast_2$-factor of $\Psi$ is contained in neither $\Lambda_1$ nor $\Lambda_2$. 
\end{prop}

\noindent
That is, if $\Psi$ is $\ast_2$-reducible, then at least one of its factors is in $\Psi \setminus (\Lambda_1 \cup \Lambda_2)$. 
In order to prove this proposition, we start by stating and proving an elementary observation about mediangle graphs. Recall from \cite[Lemma~3.17]{Mediangle} that:

\begin{lemma}\label{lem:SepFromProj}
Let $X$ be a mediangle graph, $x \in X$ a vertex, and $Y \subset X$ a gated subgraph. Every hyperplane separating $x$ from its projection on $Y$ must separate $x$ from~$Y$. 
\end{lemma}

\noindent
From this lemma, we deduce that:

\begin{cor}\label{cor:ProjectionThreeInter}
Let $X$ be a mediangle graphs and $A,B, C \subset X$ three gated subgraphs. If $A \cap B, B \cap C \neq \emptyset$, then the projection of a vertex in $A \cap B$ onto $C$ must belong to $B$. 
\end{cor}

\begin{proof}
Fix a vertex $x \in A \cap B$ and let $p \in C$ denote its projection on $C$. If $p$ does not belong to $B$, then it follows from Lemma~\ref{lem:SepFromProj} that there exists some hyperplane $J$ separating $p$ from $B$. Since $x$ belongs to $B$, in particular $J$ separates $x$ and $p$. Then, Lemma~\ref{lem:SepFromProj} implies that $J$ cannot cross $C$. Hence we get a contradiction, since $J$ separates $p$ from $B \cap C$. 
\end{proof}

\noindent
We can now prove Proposition~\ref{prop:InterCosetCoxeter} for irreducible Coxeter groups, from which the general result follows easily. 

\begin{lemma}\label{lem:DisjointCosets}
Let $\Psi$ be an $\ast_2$-irreducible labelled graph. For $\Lambda_1, \Lambda_2 \subset \Psi$, there exists $g \in C(\Psi)$ such that $g\langle \Lambda_1 \rangle \cap \langle \Lambda_2 \rangle = \emptyset$ if and only if $\Lambda_1, \Lambda_2 \neq \Psi$. 
\end{lemma}

\begin{proof}
It is clear that, if $\Lambda_1$ or $\Lambda_2$ coincides with $\Psi$ entirely, then $g\langle \Lambda_1 \rangle \cap \langle \Lambda_2 \rangle \neq \emptyset$. Conversely, assume that $\Lambda_1, \Lambda_2 \neq \Psi$. If $\Lambda_1= \Lambda_2$, then $g \langle \Lambda_1 \rangle \cap \langle \Lambda_2 \rangle= \emptyset$ for every $g \notin \langle \Lambda_1 \rangle$. From now on, we assume that $\Lambda_1 \neq \Lambda_2$. Consequently, we can find distinct vertices $a_1,a_2 \in \Psi$ such that $a_1 \notin \Lambda_1$ and $a_2 \notin \Lambda_2$. Because $\Psi$ is $\ast_2$-irreducible, we can find a sequence of pairwise distinct vertices
$$x_1:=a_2, \ x_2, \ldots, \ x_{k-1}, \ x_k:= a_1$$
such that, for every $1 \leq i \leq k-1$, $x_i$ and $x_{i+1}$ are not connected by an edge labelled $2$ in $\Psi$. (Indeed, since $\Psi$ is $\ast_2$-irreducible, the graph with the same vertices as $\Psi$, and whose edges connect two vertices whenever they are not connected by an edge labelled $2$ in $\Gamma$, is connected. Then our sequence of vertices can be taken as a geodesic connecting $a_2$ to $a_1$ in this new graph.) We claim that
$$g \langle \Psi \backslash \{a_1\} \rangle \cap \langle \Psi \backslash \{a_2\} \rangle = \emptyset \text{ where } g:=x_1 \cdots x_k.$$
This is sufficient in order to conclude the proof of our lemma since 
$$g \langle \Lambda_1 \rangle \cap \langle \Lambda_2 \rangle \subset g \langle \Psi \backslash \{a_1\} \rangle \cap \langle \Psi \backslash \{a_2\} \rangle.$$
Notice that no relation can be applied to the word $x_1 \cdots x_k$, which amounts to the path $\xi$ in $\mathrm{Cay}(C(\Psi), V(\Psi))$ starting at $1$ and labelled by $x_1 \cdots x_k$ being a convex geodesic (see for instance \cite[Lemmas~3.3 and~3.4]{Mediangle}). It follows from this convexity that $\xi$ must pass through the projection $p$ of $1$ on $g \langle \Psi \backslash \{a_1\} \rangle$. If $p$ is not the terminus of $\xi$, then the word $x_1 \cdots x_k$ must have a suffix that belongs to $\langle \Psi \backslash \{a_1\} \rangle$; but the last letter in the word is $x_k=a_1$, so this is not possible. Thus, $\xi$ connects $1$ to its projection on $g \langle \Psi \backslash \{a_1\} \rangle$. If $g \langle \Psi \backslash \{a_1\} \rangle$ intersects $\langle \Psi \backslash \{a_2\} \rangle$, then by Corollary~\ref{cor:ProjectionThreeInter}  $p$ must belong to $\langle \Psi \backslash \{a_2 \} \rangle$, hence $\xi \subset \langle \Psi \backslash \{a_2\} \rangle$ by convexity. So the word $x_1 \cdots x_k$ must belong to $\langle \Psi \backslash \{a_2\} \rangle$, which is impossible since the first letter is $x_1=a_2$. 
\end{proof}

\begin{proof}[Proof of Proposition~\ref{prop:InterCosetCoxeter}.]
Decompose $\Psi$ as a $\ast_2$-product $\Psi_1 \ast_2 \cdots \ast_2 \Psi_n$ of $\ast_2$-irreducible factors. So $C(\Psi)$ decomposes as $C(\Psi_1) \times \cdots \times C(\Psi_n)$. For $i=1,2$ and for every $1 \leq j \leq n$, set $\Lambda_i^j := \Lambda_i \cap \Psi_j$. For every $g \in C(\Psi)$, we have
$$g \langle \Lambda_1 \rangle \cap \langle \Lambda_2 \rangle = \left( g_1 \langle \Lambda_1^1 \rangle \cap \langle \Lambda_2^1 \rangle \right) \times \cdots \times \left( g_n \langle \Lambda_1^n \rangle \cap \langle \Lambda_2^n \rangle \right),$$
where $(g_1, \ldots, g_n)$ is the (direct product) decomposition of $g$. Consequently, there exists some $g \in C(\Psi)$ such that $g \langle \Lambda_1 \rangle \cap \langle \Lambda_2 \rangle = \emptyset$ if and only if there exist $1 \leq i \leq n$ and $g_i \in C(\Psi_i)$ such that $g \langle \Lambda_1^i \rangle \cap \langle \Lambda_2^i \rangle =\emptyset$. According to Lemma~\ref{lem:DisjointCosets}, this means that there exists $1 \leq i \leq n$ such that $\Lambda_1^i, \Lambda_2^i \neq \Psi_i$, or equivalently, that at least one $\ast_2$-factor of $\Psi$ is contained in neither $\Lambda_1$ nor $\Lambda_2$. 
\end{proof}

\subsection{Finding contracting elements}\label{section:ContractingInPeriagroups}

\noindent
In this section, our goal is to construct contracting elements in finitely generated periagroups. Clearly, given finitely many groups $G_1, \ldots, G_n$ with generating sets $S_1, \ldots, S_n$, respectively, the product $G_1 \times \cdots \times G_n$ contains a contracting element with respect to $S_1 \cup \cdots \cup S_n$ if and only if there exists some $1 \leq i \leq n$ such that $\mathrm{Cay}(G_j,S_j)$ is finite for every $j \neq i$ and such that $G_i$ contains a contracting element with respect to $S_i$. Therefore, it is sufficient to restrict ourselves to groups that do not decompose as products. For periagroups, this amounts to focusing on the following subfamily:

\begin{definition}
A periagroup $\Pi(\Gamma,\lambda,\mathcal{G})$ is \emph{irreducible} if the labelled graph $\Gamma$ does not split as a non-trivial $\ast_2$-join.
\end{definition}

\noindent
If the labelled graph $\Gamma$ is a $\ast_2$-join, as defined by Definition~\ref{def:JoinTwo}, then the periagroup $\Pi(\Gamma, \lambda, \mathcal{G})$ decomposes as a product of periagroups over the $\ast_2$-factors of~$\Gamma$. 

\medskip \noindent
In some sense, it is possible to break a periagroup into pieces that look either like graph products or like Coxeter groups. Since we understand how to construct contracting elements in graph products (Section~\ref{section:GraphProduct}) and Coxeter groups (Section~\ref{section:Coxeter}), one can expect to obtain contracting elements in  periagroups from such a decomposition. In practice, this is more subtle, but such a \emph{GP-Cox decomposition} (graph product - Coxeter) will be central in our arguments. Formally:

\begin{definition}
Let $\Pi(\Gamma, \lambda, \mathcal{G})$ be a periagroup. A \emph{GP-Cox decomposition} is a graph partition $\Gamma= \Gamma_1 \sqcup \Gamma_2$ such that:
\begin{itemize}
	\item all the edges in $\Gamma$ with an endpoint in $\Gamma_1$ have label $2$, and
	\item all the vertices in $\Gamma_2$ index cyclic groups of order two.
\end{itemize}
The vertices in $\Gamma_1$ (resp.\ $\Gamma_2$) are \emph{of type GP} (resp.\ \emph{of type Cox}). 
\end{definition}

\begin{ex}
Consider the Example \ref{ex:periagroup}. There we can take $\Gamma_1=\{v_1, v_2\}$ to be of type GP and $\Gamma_2=\{v_3,v_4\}$ of type Cox.
\end{ex}
\noindent
A periagroup may admit several GP-Cox decompositions. Indeed, a vertex indexed by a cyclic group of order $2$ that is only incident to edges labelled by $2$ can be of type either GP or Cox. 
For instance, if the periagroup $\Pi(\Gamma, \lambda, \mathcal{G})$ is a right-angled Coxeter group, i.e.\ all the groups in $\mathcal{G}$ are cyclic of order $2$ and $\lambda \equiv 2$, then $\Gamma = \Gamma \sqcup \emptyset$ and $\Gamma = \emptyset \sqcup \Gamma$ are two natural but distinct GP-Cox decompositions, highlighting the group structure as either a graph product or as a Coxeter group. The results obtained in this section do not depend on the specific GP-Cox decomposition we use.

\medskip \noindent
The rest of the section is dedicated to the proof of the following statement:

\begin{thm}\label{thm:PeriagroupsContracting}
Let $\Pi:=\Pi(\Gamma, \lambda, \mathcal{G})$ be an irreducible periagroup with a fixed GP-Cox decomposition $(\Psi^c,\Psi)$, where $\Gamma$ is finite with at least two vertices, and each vertex group $G \in \mathcal{G}$ is endowed with a finite generating set $S_G$. If one of the following conditions is satisfied, then $\Pi$ has a contracting element with respect to the generating set $\bigcup_{G \in \mathcal{G}} S_G$:
\begin{itemize}
	\item[(i)] $\Psi^c = \emptyset$ and $C(\Psi)$ is a Coxeter group that is neither spherical nor affine;
	\item[(ii)] $\Psi= \emptyset$ and $\Psi^c$ is not a large join;
	\item[(iii)] $\Psi$ and $\Psi^c$ are both non-empty.
\end{itemize}
If none of these conditions is satisfied, then $\Pi$ is neither virtually infinite cyclic nor acylindrically hyperbolic. 
\end{thm}

\noindent
Cases $(i)$ and $(ii)$ correspond to Coxeter groups and graph products as studied in Sections~\ref{section:Coxeter} and ~\ref{section:GraphProduct}. Case $(iii)$ is the one we must settle. 

\medskip \noindent
Associated to a GP-Cox decomposition $(\Psi^c,\Psi)$ of $\Gamma$, the periagroup $\Pi= \Pi(\Gamma, \lambda, \mathcal{G})$ can be decomposed as a semi-direct product $\Omega \mathcal{J} \rtimes C(\Psi)$ of a graph product and a Coxeter group. This follows from \cite[Theorem~6.1]{Mediangle}, and we explain the decomposition now. Let $\Delta$ denote the gated subgraph $\langle \Psi \rangle$ in the mediangle graph $M=M(\Gamma, \lambda,\mathcal{G})$ and let $\mathcal{J}$ denote the set of the hyperplanes of $M$ that are tangent to $\Delta$ (i.e.\ that are not separated from $\Delta$ by other hyperplanes but that do not cross $\Delta$ themselves). Then 
$$\Pi = \underset{=: \mathrm{Rot}}{\underbrace{\langle \mathrm{stab}_\circlearrowright(J), \ J \in \mathcal{J} \rangle}} \rtimes \underset{=C(\Psi)}{\underbrace{\mathrm{stab}(\Delta)}}.$$
The decomposition essentially comes from the fact that $\mathrm{Rot}$ acts on $M$ with $\Delta$ as a fundamental domain, which is proved thanks to a ping-pong argument. By looking at how $\mathrm{Rot}$ acts on $\Delta$, we can show that it decomposes as a graph product $\Omega \mathcal{J}$. Here, $\Omega$ is the crossing graph of $\mathcal{J}$, i.e.\ the graph whose vertices are the hyperplanes in $\mathcal{J}$ and whose edges connect two hyperplanes whenever they are transverse. The vertex groups given by a hyperplane $J$ of $\mathcal{J}$, thought of as a vertex of $\Omega$, is $\mathrm{stab}_\circlearrowright (J)$ (which is a conjugate of a vertex group of $\Pi$). 

\medskip \noindent
This decomposition, and its geometric description, will be often used in the sequel. 

\medskip \noindent
In order to prove Theorem~\ref{thm:PeriagroupsContracting}, we need a better understanding of the structure of the graph $\Omega$. The next result shows that $\Omega$ can be thought of as a ``blow up'' of $\Psi^c$. 

\begin{prop}\label{prop:LabelMapOmega}
There is a surjective map $\ell : \Omega \twoheadrightarrow \Psi^c$ such that:
\begin{itemize}
	\item[(i)] for every $u \in \Psi^c$, $\ell^{-1}(u)$ is a collection of pairwise non-adjacent vertices of size $[C(\Psi): \langle \mathrm{link}(u) \cap \Psi \rangle]$;
	\item[(ii)] for all non-adjacent $u,v \in \Psi^c$, no vertex in $\ell^{-1}(u)$ is adjacent to some vertex in $\ell^{-1}(v)$;
	\item[(iii)] for all adjacent $u,v \in \Psi^c$, $\ell^{-1}(u) \cup \ell^{-1}(v)$ is a bipartite complete graph if and only if every $\ast_2$-component of $\Psi$ is contained in $\mathrm{link}(u)$ or $\mathrm{link}(v)$.
\end{itemize}
Moreover, in the graph product $\Omega \mathcal{J}$, two vertices of $\Omega$ with the same image under $\ell$ index isomorphic groups.
\end{prop}

\begin{proof}
By definition, $\Omega$ is the crossing graph of the hyperplanes (necessarily of type GP) that are tangent to $\Delta$. Let $\ell : \Omega \to \Psi^c$ be the labelling map; this is surjective because, for every vertex $u \in \Psi^c$, the hyperplane containing the edge $[1,u]$ belongs to $\Omega$ and is labelled by $u$. Notice also that, regarding the graph product $\Omega \mathcal{J}$, a vertex of $\Omega$ indexes the rotative-stabiliser of the corresponding hyperplane tangent to $\Delta$. Since the rotative-stabiliser of a hyperplane of type GP labelled by a vertex $u$ is always conjugate to $\langle u \rangle$ (see Proposition~\ref{prop:RightHyp}), it follows that two vertices of $\Omega$ with the same impage under $\ell$ index isomorphic groups, as desired.

\medskip \noindent
In order to prove (i), notice that $C(\Psi)$ acts transitively on $\ell^{-1}(u)$, where $u \in \Psi^c$ is a vertex we fix. Indeed, let $J$ and $H$ be two hyperplanes tangent to $\Delta$ both labelled by $u$. Then $J$ (resp.\ $H$) must contain an edge of the form $[a,au]$ (resp.\ $[b,bu]$) for some $a \in \Delta$ (resp.\ $b \in \Delta$). Since $ab^{-1}$ belongs to $C(\Psi)$ and sends $[b,bu]$ to $[a,au]$, this yields the desired conclusion. As a consequence, the size of $\ell^{-1}(u)$ must coincide with the index of the stabiliser of the hyperplane containing $[1,u]$ in $C(\Psi)$, which is $\langle \mathrm{link}(u) \rangle \cap C(\Psi)= \langle \mathrm{link}(u) \cap \Psi \rangle$ according to Proposition~\ref{prop:RightHyp}. 

\medskip \noindent
By Proposition~\ref{prop:RightHyp}, two transverse hyperplanes of type GP have adjacent labels, so this concludes the proof of (i) and also implies (ii).

\medskip \noindent
Finally, let $u,v \in \Psi^c$ be two adjacent vertices. For convenience, let $J_u$ (resp.\ $J_v$) denote the hyperplane containing the edge $[1,u]$ (resp.\ $[1,v]$). We saw that $C(\Psi)$ acts transitively on all the hyperplanes having the same label, so $\ell^{-1}(u)$ (resp.\ $\ell^{-1}(v)$) corresponds to all the hyperplanes $gJ_u$ (resp.\ $gJ_v$) for $g \in C(\Psi)$. For all $g,h \in C(\Psi)$, it follows from Proposition~\ref{prop:RightHyp} that $gJ_u$ and $hJ_v$ are transverse if and only if their carriers intersect in $M$, or equivalently in $\Delta$ (as a consequence of the Helly property for gated subgraphs). Therefore, $gJ_u$ and $hJ_v$ are transverse if and only 
$$g \langle \mathrm{link}(u) \cap \Psi \rangle \cap h \langle \mathrm{link}(v) \cap \Psi \rangle \neq \emptyset.$$
According to Proposition~\ref{prop:InterCosetCoxeter}, we deduce that $\ell^{-1}(u) \cup \ell^{-1}(v)$ is a bipartite complete graph if and only if every $\ast_2$-component of $\Psi$ is contained in $\mathrm{link}(u)$ or $\mathrm{link}(v)$. 
\end{proof}

\begin{cor}\label{cor:DoubleGammaJoin}
If $\Pi$ is irreducible, then $\Omega$ is not a join.
\end{cor}

\begin{proof}
Assume that $\Omega$ decomposes non-trivially as a join $\Omega_1 \ast \Omega_2$. Our goal is to show that $\Pi$ is reducible, which amounts to saying that the graph $\Xi$, with the same vertices as $\Gamma$ and whose edges connect two vertices of $\Gamma$ whenever they are not connected by an edge labelled $2$ in $\Gamma$, is disconnected. 

\medskip \noindent
Let us verify that $\ell(\Omega_1)$ and $\ell(\Omega_2)$ lie in distinct connected components in $\Xi$. If it is not the case, we can find a path $u_1, \ldots, u_n$ in $\Xi$ with $u_1 \in \ell(\Omega_1)$ and $u_n \in \ell(\Omega_2)$. By choosing this path of minimal length, we can assume that $\{u_2, \ldots, u_{n-1}\}$ is disjoint from $\ell(\Omega_1) \cup \ell(\Omega_2) = \Psi^c$. Let $\Psi_0$ denote the $\ast_2$-component of $\Psi$ that contains $u_2$. Necessarily, $u_2, \ldots, u_{n-1}$ all belong to $\Psi_0$. Since $\ell^{-1}(\{u_1,u_n\})$ decomposes as a join $\ell^{-1}(u_1) \ast \ell^{-1}(u_n)$ in $\Omega$, it follows from Proposition~\ref{prop:LabelMapOmega} that $\Psi_0$ must be contained in $\mathrm{link}_\Gamma(u_1)$ or $\mathrm{link}_\Gamma(u_n)$. Consequently, in $\Xi$, either there is no edge connecting $u_1$ to $\Psi_0$ or there is no edge connecting $u_n$ to $\Psi_0$. But we know that $u_1$ is adjacent to $u_2 \in \Psi_0$ and that $u_n$ is adjacent to $u_{n-1} \in \Psi_0$, a contradiction.
\end{proof}

\noindent
We need a last preliminary lemma before turning to the proof of Theorem~\ref{thm:PeriagroupsContracting}. 

\begin{lemma}\label{lem:CentraliserOmegaJ}
The centraliser of $\Omega \mathcal{J}$ in $C(\Psi)$ is $\langle \Lambda \rangle$, where $\Lambda$ is the union of the $\ast_2$-factors of $\Psi$ contained in $\mathrm{link}(\Psi^c)$.
\end{lemma}

\begin{proof}
From the solution to the word problem in periagroups, we deduce that, for every $u \in \Psi^c$, the centraliser of $G_u$ in $C(\Psi)$ is $\langle \mathrm{link}(u) \cap \Psi \rangle$. Since $\Omega \mathcal{J}= \langle g G_u g^{-1}, u \in \Psi^c, g\in C(\Psi) \rangle$, we have
$$\begin{array}{lcl} Z(\Omega \mathcal{J}) \cap C(\Psi) & = & \displaystyle \bigcap\limits_{g \in C(\Psi), u \in \Psi^c} Z(gG_ug^{-1}) \cap C(\Psi) = \bigcap\limits_{g \in C(\Psi)} g \left( \bigcap\limits_{u \in \Psi^c} \langle \mathrm{link}(u) \cap \Psi \rangle \right) g^{-1} \\ \\ & = & \displaystyle \bigcap\limits_{g \in C(\Psi)} g \langle \mathrm{link}(\Psi^c) \cap \Psi \rangle g^{-1}. \end{array}$$
We conclude the proof thanks to

\begin{claim}
For every subgraph $\Lambda \subset \Psi$, the intersection 
$$\bigcap\limits_{g \in C(\Psi)} g \langle \Lambda \rangle g^{-1}$$
is equal to $\langle \Xi \rangle$, where $\Xi$ is the union (possibly empty) of the $\ast_2$-factors contained in~$\Lambda$.
\end{claim}

\noindent
The intersection must be a normal parabolic subgroup, say $\langle \Xi \rangle$. For all $u \in \Xi$ and $v \notin \Xi$, we must have $vuv \in \langle \Xi \rangle$, which is only possible if $u$ and $v$ commute. Thus, $\Xi$ must be a union of $\ast_2$-factors. As $\langle \Xi \rangle \leq \langle \Lambda \rangle$, these $\ast_2$-factors must be contained in $\Lambda$. Conversely, for every $\ast_2$-factor $\Xi_0$ contained in $\Lambda$, a conjugate of $\langle \Xi_0 \rangle$ must be contained in $\langle \Lambda \rangle$. 
\end{proof}

\begin{proof}[Proof of Theorem~\ref{thm:PeriagroupsContracting}.]
Because cliques in $M$ are gated, every clique contains a unique vertex at minimal distance from $1$. In other words, every clique $C$ of $M$ can be uniquely written as $gG$ where $G \in \mathcal{G}$ and where $g$ is the unique shortest element in $gG$. Then, endow the clique $C$ with the local metric
$$\delta_C : (ga,gb) \mapsto \|a^{-1}b\|_{S_G}$$
given by the word length $\| \cdot \|_{S_G}$ of $G$ associated to the generating set $S_G$. Let $\mathscr{C}$ denote the corresponding system of metrics.

\begin{claim}\label{claim:CayleyPeriagroups}
The system of metrics $\mathscr{C}$ is coherent and $\Pi$-invariant. Moreover, $(M,\delta)$ is $\Pi$-equivariantly isometric to $\mathrm{Cay}(\Pi, \bigcup_{G \in \mathcal{G}} S_G)$.
\end{claim}

\noindent
We start by verifying that $\mathscr{C}$ is $\Pi$-invariant. So let $C$ be a clique, $x,y \in C$ two vertices, and $g \in \Pi$ an element. If $h \in C$ denotes the projection of $1$ to $C$, we can write $x=ha$ and $y=hb$ for some generators $a,b \in G$ where $G$ is the group of $\mathcal{G}$ that indexes $C$, and we have $\delta_C(x,y)= \| b^{-1}a\|_{S_G}$. If $k$ denotes the projection of $1$ to $gC$, then we can write $gh=kw$ for some $w \in S_G$. Since $gx=kwa$ and $gy = kwb$ where $wa,wb \in G$, we conclude that
$$\delta_{gC}(gx,gy)= \| b^{-1}a\|_{S_G} = \delta_C(x,y),$$
as desired. The fact that $\mathscr{C}$ is coherent also follows from the fact that $\mathscr{C}$ is $\Pi$-equivariant as a consequence of Lemma~\ref{lem:ParallelEdgesSameOrbit}.

\medskip \noindent
Finally, let us verify that $\delta(x,y)= \|y^{-1}x\|_S$ for all $x,y \in M$, which will conclude the proof of our claim. Notice that, according to Propositions~\ref{prop:Delta} and~\ref{prop:GeodPeriagroups}, $\delta(x,y)$ coincides with the sum of the lengths of the syllables of a graphically reduced word representing $y^{-1}x$. But a word written over $S \cup S^{-1}$ can be reduced, and applying a reduction, a fusion, or a dihedral relation does not increase the sum of the lengths of the syllables. Therefore, a word of length $\|y^{-1}x\|_S$ written over $S\cup S^{-1}$ can be reduced into a graphically reduced word such that the lengths of its syllables is $\leq \|y^{-1}x\|_S$. This proves the desired equality and concludes the proof of Claim~\ref{claim:CayleyPeriagroups}. 

\medskip \noindent
As a consequence of Claim~\ref{claim:CayleyPeriagroups}, finding contracting elements in $\Pi$ amounts to finding contracting isometries in $\Pi$ with respect to its action on $(M,\delta)$. Our goal will be to apply Theorem~\ref{thm:Contracting}. 

\medskip \noindent
Let $\Pi = \Omega \mathcal{J} \rtimes C(\Psi)$ be the decomposition of $\Pi$ associated to our fixed GP-Cox decomposition. If $\Omega$ is empty, then $\Psi^c$ is empty, which amounts to saying that our periagroup $\Pi$ coincides with the Coxeter group $C(\Psi)$. So Theorem~\ref{thm:CoxeterContracting} applies. If $\Psi$ is empty, then our periagroup coincides with a graph product and Theorem~\ref{thm:ContractingGP} applies. From now on, we assume that $\Omega$ and $\Psi$ are both non-empty. Observe that $\Pi$ naturally acts on the quasi-median graph $\mathrm{QM}(\Omega, \mathcal{J})$, since it is a Cayley graph of $\Omega \mathcal{J}$ given by a generating set that is invariant under the action of $C(\Psi)$ on $\Omega \mathcal{J}$ by conjugation (which only permutes the vertex-groups). With respect to this action, vertex-stabilisers coincide with the conjugates of $C(\Psi)$.

\begin{claim}\label{claim:FiniteInterStab}
If there are two strongly separated hyperplanes $J_1$ and $J_2$ in $\mathrm{QM}(\Omega, \mathcal{J})$ such that $\mathrm{stab}(J_1) \cap \mathrm{stab}(J_2)$ is finite, then $\Pi$ contains a contracting element.
\end{claim}

\noindent
Recall that two hyperplanes are strongly separated whenever they are not both transverse to a third hyperplane. 

\medskip \noindent
First, notice that there is a natural $\Pi$-equivariant map $\eta$ from our mediangle graph $M$ to the quasi-median graph $\mathrm{QM}(\Omega,\mathcal{J})$ since $\Pi$ acts freely on $M$ and vertex-transitively on $\mathrm{QM}(\Omega,\mathcal{J})$. Explicitly, a vertex $x$ of $M$ corresponds to a unique element $g \in \Pi$ and we denote $\eta(x)$ as the $g$-translate of the vertex of $\mathrm{QM}(\Omega,\mathcal{J})$ given by the neutral element of $\Omega \mathcal{J}$. Our GP-Cox decomposition induces a partition of the generators of $\Pi$ from which $M$ is constructed into generators of type GP and generators of type Cox. Notice that an edge of $M$ is sent by $\eta$ to a single vertex if it is of type Cox and to an edge if it is of type GP. Since two opposite edges of a $4$-cycle in $M$ always have the same type, it follows that $\eta$ sends hyperplanes of type GP to hyperplanes and preserves transversality. 

\medskip \noindent
Now, because $\eta$ is surjective, we can fix two hyperplanes of type GP $\bar{J}_1$ and $\bar{J}_2$ respectively sent to $J_1$ and $J_2$ by $\eta$. Let us verify that $\bar{J}_1$ and $\bar{J}_2$ are well-separated relative to $\mathscr{C}$. In fact, since $J_1$ and $J_2$ are strongly separated in $\mathrm{QM}(\Omega, \mathcal{J})$, we know that no hyperplane of type GP can be transverse to both $\bar{J}_1$ and $\bar{J}_2$. Consequently, a hyperplane transverse to both $\bar{J}_1$ and $\bar{J}_2$ must be of type Cox, so it delimits only two sectors. It follows that it suffices to verify that only finitely many hyperplanes may be transverse to both $\bar{J}_1$ and $\bar{J}_2$. Given a hyperplane of $M$ transverse to both $\bar{J}_1$ and $\bar{J}_2$, we know from Lemma~\ref{lem:TransverseThenStab} that its rotative-stabiliser stabilises both $\bar{J}_1$ and $\bar{J}_2$, and consequently both $J_1$ and $J_2$. Hence
$$\langle \mathrm{stab}_\circlearrowright(J), \ J \text{ transverse to both } \bar{J}_1 \text{ and } \bar{J}_2 \rangle \leq \mathrm{stab}(J_1) \cap \mathrm{stab}(J_2).$$
But the intersection $\mathrm{stab}(J_1) \cap \mathrm{stab}(J_2)$ is finite by assumption, whence the desired conclusion. 

\medskip \noindent
Next, fix a sector $\bar{J}_2^+$ delimited by $\bar{J}_2$ that is disjoint from $\bar{J}_1$ and let $\bar{J}_1^+$ denote the sector delimited by $\bar{J}_1$ containing $\bar{J}_2^+$. Notice that, if $a \in \mathrm{stab}_\circlearrowright(\bar{J}_2)$ sends $\bar{J}_1$ in $\bar{J}_2^+$ and $b \in \mathrm{stab}_\circlearrowright(a \bar{J}_1)$ is non-trivial, then $g:=ba$ satisfies 
$$g \bar{J}_1^+ \subset \bar{J}_2^+ \subset \bar{J}_1^+.$$
Thus, $g$ skewers a pair of hyperplanes in $M$ that are well-separated relative to $\mathscr{C}$. The same property holds for all the non-trivial powers of $g$, and it follows from Proposition~\ref{prop:AxisSometimes} that $g$ has a power admitting an axis. Thus, we can apply Theorem~\ref{thm:Contracting} and deduce that $\Pi$ contains an element that is contracting in $(M,\delta)$, which amounts to saying that $\Pi$ contains a contracting element in the desired Cayley graph according to Claim~\ref{claim:CayleyPeriagroups}. This concludes the proof of Claim~\ref{claim:FiniteInterStab}. 

\medskip \noindent
In view of Claim~\ref{claim:FiniteInterStab}, we first look for strongly separated hyperplanes in $\mathrm{QM}(\Omega, \mathcal{J})$. 

\medskip \noindent
Notice that $\Omega$ is not a join and contains at least two vertices. The former assertion is justified by Corollary~\ref{cor:DoubleGammaJoin}. Next, if $\Omega$ is reduced to a single vertex, then it follows from Proposition~\ref{prop:LabelMapOmega} that $\Psi^c$ is reduced to a single vertex, say $u$, and that $\mathrm{link}(u) \cap \Psi = \Psi$. But this contradicts the fact that our periagroup is irreducible. 

\medskip \noindent
The fact that $\Omega$ is not a join and contains at least two vertices essentially means that $\mathrm{QM}(\Omega, \mathcal{J})$ contains strongly separated hyperplanes. But we need to find two such hyperplanes such that the intersection of their stabilisers is finite, which is not always possible. An obstruction comes from the centraliser of $\Omega \mathcal{J}$ in $C(\Psi)$.

\medskip \noindent
Notice that, if the centraliser of $\Omega \mathcal{J}$ in $C(\Psi)$ is infinite, then $\Pi$ cannot be acylindrically hyperbolic. Indeed, as already said, if $\Pi$ is acylindrically hyperbolic, then $\Omega \mathcal{J}$ must be acylindrically hyperbolic as well. As a consequence, it must contain an element whose centraliser is virtually infinite cyclic, which is not possible if the centraliser of $\Omega \mathcal{J}$ in $C(\Psi)$ is infinite. So, from now on, we assume that the centraliser of $\Omega \mathcal{J}$ in $C(\Psi)$ is finite. According to Lemma~\ref{lem:CentraliserOmegaJ}, this amounts to saying that every $\ast_2$-factor of $\Psi$ contained in $\mathrm{link}(\Psi^c)$ defines a finite Coxeter group. 

\medskip \noindent
Our goal now is to find two strongly separated hyperplanes in $\mathrm{QM}(\Omega, \mathcal{J})$ whose stabilisers have a finite intersection. We start by looking for two distinct vertices whose stabilisers have a finite intersection. 

\begin{claim}\label{claim:FixatorParabolic}
For every finite set of vertices $V \subset \mathrm{QM}(\Omega, \mathcal{J})$, $\mathrm{Fix}(V)$ is a parabolic subgroup of finite type in $\Pi$.
\end{claim}

\noindent
The stabiliser of $1 \in \mathrm{QM}(\Omega, \mathcal{J})$ is clearly $C(\Psi)$. Since $\Omega \mathcal{J}$ acts transitively on $\mathrm{QM}(\Omega, \mathcal{J})$, it follows that the stabilisers of vertices are all conjugates of $C(\Psi)$. Since $\mathrm{Fix}(V)$ coincides with the intersection of the stabilisers of all the vertices in $V$, it follows from Theorem~\ref{thm:InterParabolicPeriagroups} that $\mathrm{Fix}(V)$ is a parabolic subgroup. Moreover, as $\Psi$ is finite, we deduce from Corollary~\ref{cor:FiniteType} that $\mathrm{Fix}(V)$ has finite type. 

\begin{claim}\label{claim:FiniteOrbits}
With respect to its action on $\mathrm{QM}(\Omega, \mathcal{J})$, every orbit of $C(\Psi)$ is finite.
\end{claim}

\noindent
Given a vertex $x$ of $\mathrm{QM}(\Omega,\mathcal{J})$, thought of as a word $w$ consisting of syllables in the graph product $\Omega \mathcal{J}$, all the vertices in the orbit $C(\Psi) \cdot x$ can be represented by words obtained by permuting the syllables of $w$. Since there exist only finitely many such words, the claim follows.

\begin{claim}\label{claim:FiniteFixatorPair}
There exist two distinct vertices $x,y \in \mathrm{QM}(\Omega, \mathcal{J})$ such that $\mathrm{stab}(x) \cap \mathrm{stab}(y)$ is finite.
\end{claim}

\noindent
Fix an enumeration $\{x_1,x_2, \ldots\}$ of $\mathrm{QM}(\Omega, \mathcal{J})$, which is possible because $\Omega$ is countable and all the groups from $\mathcal{J}$ are finitely generated. According to Claim~\ref{claim:FixatorParabolic}, $(\bigcap_{i \leq n} \mathrm{stab}(x_i))_n$ is a non-increasing sequence of parabolic subgroups of finite type in $\Pi$. According to Corollary~\ref{cor:ChainParabolicTerminate}, such a sequence has to stabilise. Therefore, there exists a finite subset $V \subset \mathrm{QM}(\Omega, \mathcal{J})$ such that $\mathrm{Fix}(V)$ coincides with the kernel of the action of $\Pi$ on $\mathrm{QM}(\Omega, \mathcal{J})$. Notice that this kernel is contained in $C(\Psi)$. Indeed, an element $g$ of this kernel can be written as $g=ab$ for some $a \in \Omega \mathcal{J}$ and $b \in C(\Psi)$. Then,
$$1 = g \cdot 1 = ab \cdot 1 = a \cdot 1 = a \text{ implies } g=b \in C(\Psi).$$
But the action of $C(\Psi)$ on $\mathrm{QM}(\Omega, \mathcal{J})$, algebraically speaking, coincides with the action of $C(\Psi)$ on $\Omega \mathcal{J}$ by conjugation, so the kernel we are looking for is the centraliser of $\Omega \mathcal{J}$ in $C(\Psi)$, which is assumed to be finite.

\medskip \noindent
So far, we have proved that there exists a finite subset $V \subset \mathrm{QM}(\Omega, \mathcal{J})$ such that $\mathrm{Fix}(V)$ is finite. Fix two distinct vertices $x,y \in V$. (If this is not possible, i.e.\ if $V$ is reduced to a single vertex, then $C(\Psi)$ must be finite; but, in this case, our claim is clear.) As a consequence of Claim~\ref{claim:FiniteOrbits}, $\mathrm{stab}(x) \cap \mathrm{stab}(y)$ has finite index in $\mathrm{Fix}(V)$. This proves Claim~\ref{claim:FiniteFixatorPair}.

\medskip \noindent
Now, we apply Lemma~\ref{lem:StronglySeparatedEverywhere} below and find two strongly separated hyperplanes $A$ and $B$ such that the two vertices $x$ and $y$ from Claim~\ref{claim:FiniteFixatorPair} belong to a geodesic connecting the two vertices, say $x'$ and $y'$, minimising the distance between $N(A)$ and $N(B)$. Notice that $\mathrm{stab}(A) \cap \mathrm{stab}(B)$ contains a finite-index subgroup in $\mathrm{stab}(x) \cap \mathrm{stab}(y)$. Indeed, $\mathrm{stab}(A) \cap \mathrm{stab}(B)$ fixes $x'$ and $y'$, and there exist only finitely many geodesics between two fixed vertices in a quasi-median graph. Since $\mathrm{stab}(x) \cap \mathrm{stab}(y)$ is finite, we deduce that $\mathrm{stab}(A) \cap \mathrm{stab}(B)$ must be finite as well. We conclude from Claim~\ref{claim:FiniteInterStab} that $\Pi$ contains a contracting element. 
\end{proof}

\noindent
During the proof of the theorem, we used the following assertion, which we prove now:

\begin{lemma}\label{lem:StronglySeparatedEverywhere}
Let $\Omega$ be a countable graph with at least two vertices, and let $\mathcal{J}$ a collection of groups indexed by $\Omega$. Assume that $\Omega$ is not a join, and $\mathrm{clique}(\Omega)<\infty$. Then, for any two distinct vertices $x,y \in \mathrm{QM}(\Omega, \mathcal{J})$, there exist two strongly separated hyperplanes $A$ and $B$ such that $x$ and $y$ belong to a geodesic connecting the two vertices minimising the distance between $N(A)$ and $N(B)$.
\end{lemma}

\noindent
The \emph{clique-number} $\mathrm{Clique}(\Gamma)$ of a graph $\Gamma$ refers to the maximal size of a complete subgraph in $\Gamma$. Therefore, $\mathrm{Clique}(\Gamma)$ is finite when there is a uniform upper bound on the sizes of the complete subgraphs in $\Gamma$. In the proof below, we will also refer to the \emph{opposite graph} $\Gamma^{\mathrm{opp}}$ of $\Gamma$ as the graph with the same vertices as $\Gamma$ and whose edges connect two vertices whenever they are not adjacent in $\Gamma$. Notice that $\Gamma^\mathrm{opp}$ is connected if and only if $\Gamma$ is not a join.

\begin{proof}[Proof of Lemma~\ref{lem:StronglySeparatedEverywhere}.]
Let $x,y \in \mathrm{QM}(\Omega, \mathcal{J})$ be two distinct vertices. Fix a geodesic $[x,y]$ connecting $x$ and $y$ and let $v_x$ (resp.\ $v_y$) denote the vertex of $\Omega$ labelling the edge of $[x,y]$ containing $x$ (resp.\ $y$). Let $\sigma_x$ (resp.\ $\sigma_y$) be an infinite ray $\Omega^\mathrm{opp}$ that passes through each vertex infinitely times and that starts from a vertex distinct from and not adjacent to $v_x$ (resp.\ $v_y$) in $\Omega$. We extend a geodesic $[x,y]$ between $x$ and $y$ by a ray $\gamma^+$ starting from $y$ whose edges are successively labelled by the vertices of $\sigma_y$ and by a ray $\gamma^-$ starting from $x$ whose edges are successively labelled by the vertices of $\sigma_x$. For convenience, we denote by $A_1,A_2, \ldots$ (resp.\ $A_{-1},A_{-2},\ldots$) the hyperplanes successively crossed by $\gamma^+$ (resp.\ $\gamma^-$). A key observation is that, by construction, the hyperplanes $A_1,A_2, \ldots$ (resp.\ $A_{-1},A_{-2},\ldots$) are pairwise non-transverse. 

\begin{claim}\label{claim:GammaGeodesic}
The line $\gamma:= \gamma^- \cup [x,y] \cup \gamma^+$ is a geodesic.
\end{claim}

\noindent
Assume by contradiction that $\gamma$ is not geodesic, that is, $\gamma$ crosses some hyperplane twice. First, assume that some hyperplane $A_{-n}$ crosses both $\gamma^-$ and $[x,y]$, and choose $n$ as small as possible. Notice that $n=1$, since otherwise $A_{-1}$ would have to be transverse to $A_{-n}$. Notice that all the hyperplanes crossed by $[x,y]$ between $x$ and $A_{-1}$ must be transverse to $A_{-1}$. Either $A_{-1}$ contains the first of $[x,y]$, which is impossible since $A_{-1}$ is not labelled by $v_x$; or $A_{-1}$ contains the first edge of $[x,y]$, and $A_{-1}$ must be transverse to the hyperplane containing the first edge of $[x,y]$, which is impossible since the label of $A_{-1}$ is not adjacent to $v_x$. Thus, no hyperplane crosses both $\gamma^-$ and $[x,y]$. Similarly, no hyperplane crosses both $\gamma^+$ and $[x,y]$. So there must be some hyperplane crossing both $\gamma^-$ and $\gamma^+$, i.e.\ there exist $k, \ell \geq 1$ such that $A_{-k}=A_\ell$. We choose $k$ as small as possible. Notice that $k=1$ since otherwise $A_{-k}$ would be transverse to $A_{-1}$. Then, notice that $A_{-1}$ must be transverse to all the hyperplanes separating $x$ and $y$, including the hyperplane containing the first edge of $[x,y]$, which is impossible since the label of $A_{-1}$ is not adjacent to $v_x$. 

\begin{claim}\label{claim:SepartingXandY}
For every $i \geq 1$ sufficiently large, the hyperplanes separating $x$ and $y$ separate $A_{-i}$ and $A_i$. 
\end{claim}

\noindent
Given a hyperplane $J$ separating $x$ and $y$, $J$ cannot be transverse to all the $A_i$, since we know by construction that some $A_i$ has the same label as $J$. Therefore, because the $A_i$ are pairwise non-transverse, $J$ is non-transverse to the $A_i$ for $i \geq 1$ sufficiently large. The same observation applies to the $A_{-i}$. Since this holds for every hyperplane separating $x$ and $y$, the desired conclusion follows.

\begin{claim}\label{claim:AijStronglySeparated}
For every $i \geq 1$, there exists $j \geq i$ sufficiently large, $A_i$ and $A_j$ (resp.\ $A_{-i}$ and $A_{-j}$) are strongly separated. 
\end{claim}

\noindent
Otherwise, there would exist some $i \geq 1$ such that $A_i$ and $A_j$ (or $A_{-i}$ and $A_{-j}$, but this case is symmetric so we do not consider it) are nevery strongly separated for $j \geq i$, i.e.\ there exists some hyperplane $H_j$ transverse to both $A_i$ and $A_j$. Let $u_1, \ldots, u_n \in \Omega$ be a maximal collection of pairwise adjacent vertices. By construction of $\sigma_y$, we can find some $j \geq i$ sufficiently large such that $A_i$ and $A_j$ are separated by hyperplanes labelled by $u_1, \ldots, u_n$. Necessarily, $H_j$ is transverse to all these hyperplanes, so its label is adjacent to all $u_1, \ldots, u_n$, a contradiction. This concludes the proof of Claim~\ref{claim:AijStronglySeparated}.

\medskip \noindent
As a consequence of Claims~\ref{claim:SepartingXandY} and~\ref{claim:AijStronglySeparated}, we can find $k \geq \ell \geq 1$ and $1 \leq p \leq q$ such that $A_{-k},A_{-\ell}$ are strongly separated, such that $A_p,A_q$ are strongly separated, and such that the hyperplanes separating $x$ and $y$ all separate $A_{-\ell}$ and $A_p$. Let $x'$ (resp.\ $y'$) denote the projection of $x$ on $N(A_{-k})$ (resp.\ $N(A_q)$). Fix a geodesic $\alpha^-$ connecting $x'$ to $x$ and a geodesic $\alpha^+$ connecting $y$ to $y'$.

\begin{claim}
The path $\alpha:= \alpha^- \cup [x,y] \cup \alpha^+$ is a geodesic.
\end{claim}

\noindent
A hyperplane crossing both $\alpha^-$ and $[x,y]$ would have to cross $\gamma^-$ as $A_{-\ell}$ and $A_{-m}$ are strongly separated. So it would cross $\gamma$ twice, contracting Claim~\ref{claim:GammaGeodesic}. Similarly, a hyperplane crossing both $\alpha^-$ and $\alpha^+$ would have to cross both $\gamma^-$ and $\gamma^+$, contradicting Claim~\ref{claim:GammaGeodesic} again. Thus, $\alpha$ crosses each hyperplane at most once, proving that it is a geodesic.

\begin{claim}\label{claim:MinimalDistance}
The vertices $x'$ and $y'$ minimise the distance between $N(A_{-k})$ and $N(A_r)$.
\end{claim}

\noindent
Notice that a hyperplane $J$ crossing $\alpha^-$ separates $A_{-k}$ and $A_q$. Indeed, if $J$ crosses $\alpha^-$, then it cannot be transverse to $A_q$ since $A_{-k}$ and $A_q$ are strongly separated; and, since it separates $x$ from its projection on $N(A_{-k})$, it cannot be transvse to $A_{-k}$ either. Similarly, every hyperplane crossing $\alpha^+$ separates $A_{-k}$ and $A_q$. And we already know that the hyperplanes crossing $[x,y]$ separate $A_{-\ell}$ and $A_p$. Thus, the hyperplanes crossing $\alpha$, or equivalently the hyperplanes separating $x'$ and $y'$, all separate $A_{-k}$ and $A_q$. This concludes the proof of Claim~\ref{claim:MinimalDistance}.

\medskip \noindent
Thus, we have proved that $x$ and $y$ both lie on a geodesic $\alpha$ that connects two vertices $x' \in N(A_{-k})$ and $y' \in N(A_q)$, where $A_{-k}$ and $A_q$ are strongly separated. This concludes the proof of the lemma. 
\end{proof}

\section{On acylindrical hyperbolicity}\label{section:AcylHyp}

\noindent
A group admitting a proper action on a metric space with contracting isometries must be acylindrically hyperbolic or virtually cyclic by \cite{MR3415065}, so Theorem~\ref{thm:PeriagroupsContracting} allows us to determine precisely when a periagroup is acylindrically hyperbolic. Theorem \ref{thm:PeriaAcylHyp} states that a periagroup is acylindrically hyperbolic whenever it is not virtually cyclic and it does not virtually split as a product of two infinite groups in an `obvious way'. 

\begin{thm}\label{thm:PeriaAcylHyp}
A finitely generated periagroup $\Pi:= \Pi(\Gamma, \lambda, \mathcal{G})$ is acylindrically hyperbolic if and only if $\Gamma$ decomposes as a $\ast_2$-join $\Gamma_1 \ast_2 \cdots \ast_2 \Gamma_n$ of $\ast_2$-irreducible labelled graphs such that:
\begin{itemize}
	\item[(i)] for every $2 \leq i \leq n$, the labelled graph $\Gamma_i$ is a single vertex indexed by a finite group, or $\Gamma_i$ defines a spherical Coxeter group, and
	\item[(ii)] the labelled graph $\Gamma_1$ either
\begin{itemize}
	\item is a single vertex indexed by an acylindrically hyperbolic group,
	\item or defines a graph product $\neq \mathbb{Z}_2 \ast \mathbb{Z}_2$ with at least two vertices,
	\item or defines a Coxeter group that is neither spherical nor affine,
	\item or contains at least one vertex indexed by $\mathbb{Z}_2$ and at least one vertex indexed by a group of order $\geq 3$.
\end{itemize}
\end{itemize}
\end{thm}

\begin{proof}
By the assumption that $\Gamma$ is a $\ast_2$-join, the group $\Pi$ decomposes as
$$\Pi = \langle \Gamma_1 \rangle \times \cdots \times \langle \Gamma_n \rangle.$$
Since an acylindrically hyperbolic groups cannot decompose as a product of two infinite groups (see \cite[Corollary~7.2]{MR3430352}), and the product of an acylindrically hyperbolic group and a finite group remains acylindrically hyperbolic, we know that, up to permuting the $\ast_2$-factors of $\Gamma$, $\Pi$ is acylindrically hyperbolic if and only if $\langle \Gamma_1 \rangle$ is acylindrically hyperbolic and $\langle \Gamma_2 \rangle, \ldots, \langle \Gamma_n \rangle$ are all finite. Now each $\langle \Gamma_i \rangle$ is finite if and only if $\Gamma_i$ is a single vertex indexed by a finite group or it defines a spherical Coxeter group. It remains to investigate when $\langle \Gamma_1 \rangle$ is acylindrically hyperbolic. 

\medskip \noindent
Consider the GP-Cox decomposition $\Gamma_1 = \Gamma_1' \sqcup \Gamma_1''$, where $\Gamma_1''$ contains all the vertices of $\Gamma_1$ indexed by $\mathbb{Z}_2$. By Theorem~\ref{thm:PeriagroupsContracting}, $\langle \Gamma_1 \rangle$ is acylindrically hyperbolic if and only if 
\begin{itemize}
	\item $\Gamma_1$ is a single vertex indexed by an acylindrically hyperbolic group,
	\item or $\Gamma_1'= \emptyset$ (i.e.\ $\Gamma_1$ defines a Coxeter group) and $C(\Gamma_1'')$ is a Coxeter group that is neither spherical nor affine,
	\item or $\Gamma_1''= \emptyset$ and $\Gamma_1'$ defines a graph product $\neq \mathbb{Z}_2 \ast \mathbb{Z}_2$ with at least two vertices,
	\item or $\Gamma_1'$ and $\Gamma_1''$ are both non-empty, which means that $\Gamma_1$ contains at least one vertex indexed by $\mathbb{Z}_2$ and at least one vertex indexed by a group of order $\geq 3$.
\end{itemize}
This concludes the proof of the theorem. 
\end{proof}

\noindent
When restricted to Coxeter groups, Theorem~\ref{thm:PeriaAcylHyp} implies:

\begin{cor}\label{cor:CoxAcylHyp}
A finitely generated Coxeter group is acylindrically hyperbolic if and only if it decomposes as a product $W_1 \times \cdots \times W_n$ of irreducible Coxeter groups such that $W_2, \ldots, W_n$ are finite and such that $W_1$ is neither spherical nor affine.
\end{cor}

\noindent
The same result can also be deduced from the construction of rank-one elements in \cite{MR2585575}. It is worth mentioning that, since we used \cite{MR2585575} in order to prove Theorem~\ref{thm:CoxeterContracting}, we do not really get an alternative proof of Corollary~\ref{cor:CoxAcylHyp}.  

\medskip \noindent
Then, when restricted to graph products, Theorem~\ref{thm:PeriaAcylHyp} implies:

\begin{cor}
Let $\Gamma$ be a finite graph, and let $\mathcal{G}$ be a collection of finitely generated groups indexed by $V(\Gamma)$. The graph product $\Gamma \mathcal{G}$ is acylindrically hyperbolic if and only if $\Gamma$ decomposes as a join $\Gamma_1 \ast \cdots \ast \Gamma_n$ satisfying:
\begin{itemize}
	\item[(i)] for every $2 \leq i \leq n$, $\Gamma_i$ is a complete graph whose vertices are indexed by finite groups, and
	\item[(ii)] the graph $\Gamma_1$ either
\begin{itemize}
	\item is a single vertex indexed by an acylindrically hyperbolic group,
	\item or has at least two vertices but is distinct from $\mathbb{Z}_2 \ast \mathbb{Z}_2$. 
\end{itemize}
\end{itemize}
\end{cor}

\noindent
The same result can be deduced from \cite{MR3368093}. Here, our arguments differ: in \cite{MR3368093}, generalised loxodromic elements are obtained by a careful analysis of the actions of graph products on their natural Bass-Serre trees; in our proof the contracting elements are obtained from the quasi-median geometry of graph products.

\medskip \noindent
Recall that a \emph{Dyer group} \cite{MR1076077} is a periagroup of cyclic groups. When restricted to Dyer groups, Theorem~\ref{thm:PeriaAcylHyp} implies:

\begin{cor}\label{cor:DyerAcylHyp}
A finitely generated Dyer group $D:= \Pi(\Gamma, \lambda, \mathcal{G})$ is acylindrically hyperbolic if and only if $\Gamma$ decomposes as a $\ast_2$-join $\Gamma_1 \ast_2 \cdots \ast_2 \Gamma_n$ of $\ast_2$-irreducible labelled graphs such that:
\begin{itemize}
	\item[(i)] for every $2 \leq i \leq n$, the labelled graph $\Gamma_i$ is reduced to a single vertex indexed by a finite cyclic group, or $\Gamma_i$ defines a spherical Coxeter group, and
	\item[(ii)] the labelled graph $\Gamma_1$ either
\begin{itemize}
	\item defines a graph product $\neq \mathbb{Z}_2 \ast \mathbb{Z}_2$ with at least two vertices,
	\item or defines a Coxeter group that is neither spherical nor affine,
	\item or contains at least one vertex indexed by $\mathbb{Z}_2$ and at least one vertex indexed by a cyclic group of order $\geq 3$.
\end{itemize}
\end{itemize}
\end{cor}

\noindent
The same result can be found in \cite{DyerAcylHyp}. There, the theorem is obtained by combining Corollary~\ref{cor:CoxAcylHyp} (i.e.\ \cite{MR2585575}) with the fact a Dyer group is naturally a finite-index subgroup of a Coxeter group \cite{MR4735237} (following \cite{MR1783167}).

\section{Applications to conjugacy growth}\label{section:ConjGrowth}

\subsection{Conjugacy and standard growth}\label{subsec:conjgeo}

For a group $G$ with finite generating set $X$ and words $u,v \in X^{\ast}$, we use $u=v$ to denote equality of words, and $u =_{G} v$ to denote equality of the group elements represented by $u$ and $v$. The \emph{(word) length} of an element $g\in G$, denoted $|g|_X$, is the length of a shortest word in $X$ that represents $g$, i.e.\ $|g|_X := \min \{ |w| \mid w\in X^*, w =_G g\}$. We let $s_{G,X}(n):=\#\{g\in G\mid |g|_X=n\}$ be the \emph{spherical standard growth} function of $G$ with respect to $X$, and similarly let the \emph{cumulative standard growth} be $S_{G,X}(n):=\#\{g\in G\mid |g|_X \leq n\}$. We define the \emph{growth rate} of the group to be
\[ \alpha_{G,X} = \lim_{n \to \infty} \sqrt[n]{s_{G,X}(n)}.\]

\medskip \noindent
We write $g \sim h$ to express that $g$ and $h$ are conjugate, and write $[g]$ for the conjugacy class of $g$. The \emph{length} of the conjugacy class $[g]$, denoted $|[g]|$, is the shortest length among all elements in $[g]$, i.e.\ $|[g]| = \min \{ |h| \mid h \sim g \}$. 
We define the \emph{(spherical) conjugacy growth function} $c(n) = c_{G,X}(n)$ 
of $G$ with respect to $X$ to be the number of conjugacy classes whose length is $= n$, that is, \ 
\[ c(n) = \#\{[g] \mid |[g]| = n\},\]
and the \emph{conjugacy growth rate} as $\lim_{n \rightarrow \infty}\sqrt[n]{c_{G,X}(n)}$.
We can similarly define the cumulative conjugacy growth function, but for ease of computation we shall work only with the spherical version. 
The \emph{conjugacy growth series} $C(z) = C_{G,X}(z)$ is defined to be the (ordinary) generating function of $c(n)$, so
\[ C(z) = \sum\limits_{n = 0}^{\infty} c(n) z^n. \]
All results in this paper can be easily extended to the cumulative version of the conjugacy growth function and series (see \cite{AC2017}). Similar to the standard growth, we define the \emph{conjugacy growth rate} of the group as $\gamma_{G,X} = \limsup_{n \to \infty} \sqrt[n]{c_{G,X}(n)}$. This limsup doesn't exist in general, but it does for the groups we consider in this paper.

\medskip \noindent
We call a formal power series $f(z)$ \emph{rational} if it can be expressed (formally) as the ratio of two polynomials with integral coefficients, or equivalently, the coefficients of $f(z)$ satisfy a finite linear recursion. In the language of polynomial rings, this is to say $f(z) \in \mathbb{Q}(z)$. Furthermore, $f(z)$ is \emph{irrational} if it is not rational. 
A formal power series is \emph{algebraic} if it is in the algebraic closure of $\mathbb{Q}(z)$, i.e.\ it is the solution to a polynomial equation with coefficients from $\mathbb{Q}(z)$. It is called \emph{transcendental} if it is not algebraic.

\medskip \noindent
For $G$ generated by $X$, let $\alpha=\alpha_{G,X}$ be the growth rate of the group $G$, as defined above. The main tool used in \cite{AC2017, CiobanuE2020, GeYang, CC25} to prove that the conjugacy growth series $C_{G,X}(z)$ is transcendental was to show that the conjugacy growth function $c_{G,X}(n)$ has asymptotics of the form $\sim \frac{\alpha^n}{n}$. We say that a sequence $a_n$, $n \geq 1$, has asymptotics $f(n)$, denoted as $a_n \sim f(n)$, if there exist two positive constants $c_1, c_2$ such that $c_1 f(n) \leq a_n \leq c_2 f(n)$ for all $n \geq 1$. Sequences with asymptotics of the form $\sim \frac{\alpha^n}{n}$ have transcendental generating functions by \cite[Thm. D]{Flajolet1987}.

\subsection{Conjugacy growth in direct products}

The conjugacy growth series, like the standard growth series, of direct products, is the product of each factor's series \cite[Obs. 14.8]{Rivin2010}:

\begin{lemma}\label{lem:dirprod}
Let $A_i$, $1\leq i \leq n$, be groups with (disjoint) generating sets $S_i$. The conjugacy growth series of $A_1\times \dots \times A_n$ with respect to the generating set $\cup_i S_i$ satisfies
\[C_{A_1\times \dots \times A_n, \cup_i S_i}(z)= \prod_i C_{A_i,S_i}(z).\]
\end{lemma}

\medskip \noindent
The asymptotics for the conjugacy growth of a direct product has a similar behaviour to that for standard growth.

\begin{lemma} \label{lem:conjasp}
Let $A_i$, $1\leq i \leq n$, have (disjoint) generating sets $S_i$. Suppose the conjugacy growth function and rate of $A_i$ with respect to $S_i$ are $c_i(n)$ and $\gamma_i=\gamma_{A_i,S_i}$, respectively. 
\begin{enumerate}
\item The conjugacy growth rate of the direct product satisfies $\gamma_{A_1\times \dots \times A_n, \cup_i S_i} = \max_i \gamma_i$. 
\item If $\gamma_1>\gamma_i$ for all $i>1$, and $c_1(n) \sim \frac{\gamma_1^n}{n}$, then $c_{A_1\times \dots \times A_n, \cup_i S_i}(n) \sim \frac{\gamma_1^n}{n}.$
\end{enumerate}
\end{lemma}

\medskip \noindent
This translates into the following statement for the conjugacy growth series.

\begin{cor}\label{cor:dirprod}
Let $A_i$, $1\leq i \leq n$, have (disjoint) generating sets $S_i$, suppose $A_n$ is virtually abelian, and $A_1, \dots, A_{n-1}$ have exponential growth and a contracting element in $\mathrm{Cay}(A_i,S_i)$. Furthermore, assume the growth rates satisfy $\gamma_1 > \gamma_2 \geq  \dots \geq \gamma_{n-1}$. 

Then the conjugacy growth series $C(z)=C_{A_1\times \dots \times A_n, \cup_i S_i}(z)$ is transcendental.
\end{cor}

\begin{proof}
By Lemma \ref{lem:dirprod}, $C(z)$ is the product of $C_{A_1\times \dots \times A_{n-1}, \cup_i S_i}(z)$ and $C_{A_n,S_n}(z)$. Since the latter is rational by \cite{Evetts2019}, the complexity of $C(z)$ is determined by the first $n-1$ factors, each of which has asymptotics of the form $\frac{\gamma_i^n}{n}$ by \cite{GeYang}. Then Lemma \ref{lem:conjasp}(2.) implies that the product of the first $n-1$ factors has conjugacy asymptotics $\frac{\gamma_1^n}{n}$, and therefore a transcendental conjugacy growth series by  \cite[Thm. D]{Flajolet1987}.
\end{proof}

\begin{remark}
Corollary \ref{cor:dirprod} can be strengthened in certain ways. For example, if $A=B$ is of exponential growth and contains a contracting element, then the conjugacy growth series of $A \times B$ (with respect to the standard generating set) is the square of the conjugacy growth series of $A$, so is transcendental. That is, we can obtain transcendental series even if one of the groups doesn't have a strictly larger growth rate than the others. 

However, there exist periagroups (right-angled Artin groups) that are not isomorphic but have the same growth rate, in which case the arguments above do not apply.  
\end{remark}

\begin{remark}
Theorem \ref{thm:PeriaAcylHyp}, together with \cite[Theorem 1.5]{AC2017}, implies that, for any generating set, no language of (geodesic) conjugacy representatives of a non-elementary periagroup can be regular; in fact, no such language can be unambiguous context-free. We refer to \cite{AC2017} for background and an in-depth discussion on the languages of conjugacy representatives in acylindrically hyperbolic groups.
\end{remark}

\section*{Acknowledgments}
The first-named author would like to thank Matt Cordes and Alex Martin for helpful discussions.

\addcontentsline{toc}{section}{References}

\bibliographystyle{alpha}
{\footnotesize\bibliography{ContractingPeria}}

\begin{thebibliography}{BMW94}

\bibitem[AC17]{AC2017}
Y.~Antol\'in and L.~Ciobanu.
\newblock Formal {Conjugacy} {Growth} in {Acylindrically} {Hyperbolic}
  {Groups}.
\newblock {\em International Mathematics Research Notices}, 2017(1):121--157,
  2017.

\bibitem[ACT15]{MR3404665}
G.~Arzhantseva, C.~Cashen, and J.~Tao.
\newblock Growth tight actions.
\newblock {\em Pacific J. Math.}, 278(1):1--49, 2015.

\bibitem[Ban84]{MR766499}
H.-J. Bandelt.
\newblock Retracts of hypercubes.
\newblock {\em J. Graph Theory}, 8(4):501--510, 1984.

\bibitem[BBF15]{MR3415065}
M.~Bestvina, K.~Bromberg, and K.~Fujiwara.
\newblock Constructing group actions on quasi-trees and applications to mapping
  class groups.
\newblock {\em Publ. Math. Inst. Hautes \'Etudes Sci.}, 122:1--64, 2015.

\bibitem[BC12]{MR2874959}
J.~Behrstock and R.~Charney.
\newblock Divergence and quasimorphisms of right-angled {A}rtin groups.
\newblock {\em Math. Ann.}, 352(2):339--356, 2012.

\bibitem[BF09]{MR2507218}
M.~Bestvina and K.~Fujiwara.
\newblock A characterization of higher rank symmetric spaces via bounded
  cohomology.
\newblock {\em Geom. Funct. Anal.}, 19(1):11--40, 2009.

\bibitem[BMW94]{MR1297190}
H.-J. Bandelt, H.~M. Mulder, and E.~Wilkeit.
\newblock Quasi-median graphs and algebras.
\newblock {\em J. Graph Theory}, 18(7):681--703, 1994.

\bibitem[CC25]{CC25}
L.~Ciobanu and G.~Crowe.
\newblock Conjugacy geodesics and growth in dihedral artin groups.
\newblock {\em New York J. Math}, 31:465--507, 2025.

\bibitem[CEH20]{CiobanuE2020}
L.~Ciobanu, A.~Evetts, and Meng-Che Ho.
\newblock The conjugacy growth of the soluble {B}aumslag-{S}olitar groups.
\newblock {\em New York Journal of Mathematics}, 26:473--495, 2020.

\bibitem[CF10]{MR2585575}
P.-E. Caprace and K.~Fujiwara.
\newblock Rank-one isometries of buildings and quasi-morphisms of {K}ac-{M}oody
  groups.
\newblock {\em Geom. Funct. Anal.}, 19(5):1296--1319, 2010.

\bibitem[CHM23]{Ciobanu2023}
L.~Ciobanu, S.~Hermiller, and V.~Mercier.
\newblock Formal conjugacy growth in graph products {I}.
\newblock {\em Groups, Geometry, and Dynamics}, 17(2):427--457, 2023.

\bibitem[Cou23]{CoulonErg}
R.~Coulon.
\newblock Ergodicity of the geodesic flow for groups with a contracting
  element.
\newblock {\em arxiv:2303.01390}, 2023.

\bibitem[Cou24]{MR4803663}
R.~Coulon.
\newblock Patterson-{S}ullivan theory for groups with a strongly contracting
  element.
\newblock {\em Ergodic Theory Dynam. Systems}, 44(11):3216--3271, 2024.

\bibitem[CS11]{MR2827012}
P.-E. Caprace and M.~Sageev.
\newblock Rank rigidity for {CAT}(0) cube complexes.
\newblock {\em Geom. Funct. Anal.}, 21(4):851--891, 2011.

\bibitem[CS15]{MR3339446}
R.~Charney and H.~Sultan.
\newblock Contracting boundaries of {$\rm CAT(0)$} spaces.
\newblock {\em J. Topol.}, 8(1):93--117, 2015.

\bibitem[Dav08]{MR2360474}
M.~Davis.
\newblock {\em The geometry and topology of {C}oxeter groups}, volume~32 of
  {\em London Mathematical Society Monographs Series}.
\newblock Princeton University Press, Princeton, NJ, 2008.

\bibitem[DJ00]{MR1783167}
M.~Davis and T.~Januszkiewicz.
\newblock Right-angled {A}rtin groups are commensurable with right-angled
  {C}oxeter groups.
\newblock {\em J. Pure Appl. Algebra}, 153(3):229--235, 2000.

\bibitem[Dye90]{MR1076077}
M.~Dyer.
\newblock Reflection subgroups of {C}oxeter systems.
\newblock {\em J. Algebra}, 135(1):57--73, 1990.

\bibitem[Eve19]{Evetts2019}
A.~Evetts.
\newblock Rational growth in virtually abelian groups.
\newblock {\em Illinois Journal of Mathematics}, 63(4):513 -- 549, 2019.

\bibitem[Fla87]{Flajolet1987}
P.~Flajolet.
\newblock Analytic models and ambiguity of context-free languages.
\newblock {\em Theoretical Computer Science}, 49(2):283--309, 1987.

\bibitem[Gen17]{QM}
A.~Genevois.
\newblock Cubical-like geometry of quasi-median graphs and applications to
  geometric group theory.
\newblock {\em PhD thesis, arxiv:1712.01618}, 2017.

\bibitem[Gen18]{AutGP}
A.~Genevois.
\newblock Automorphisms of graph products of groups and acylindrical
  hyperbolicity.
\newblock {\em arxiv:1807.00622, to appear in Memoirs of the AMS}, 2018.

\bibitem[Gen19a]{MR4057355}
A.~Genevois.
\newblock Hyperbolicities in {${\rm CAT}(0)$} cube complexes.
\newblock {\em Enseign. Math.}, 65(1-2):33--100, 2019.

\bibitem[Gen19b]{VanKampenGP}
A.~Genevois.
\newblock On the geometry of van {K}ampen diagrams of graph products of groups.
\newblock {\em arXiv:1901.04538}, 2019.

\bibitem[Gen20]{MR4071367}
A.~Genevois.
\newblock Contracting isometries of {${\rm CAT}(0)$} cube complexes and
  acylindrical hyperbolicity of diagram groups.
\newblock {\em Algebr. Geom. Topol.}, 20(1):49--134, 2020.

\bibitem[Gen22a]{Mediangle}
A.~Genevois.
\newblock Rotation groups, mediangle graphs, and periagroups: a unified point
  of view on {C}oxeter groups and graph products groups.
\newblock {\em arxiv:2212.06421}, 2022.

\bibitem[Gen22b]{TransLengthQM}
A.~Genevois.
\newblock Translation lengths in crossing and contact graphs of (quasi-)median
  graphs.
\newblock {\em arXiv:2209.06441}, 2022.

\bibitem[Gen24]{RARotation}
A.~Genevois.
\newblock Rotation groups virtually embed into right-angled rotation groups.
\newblock {\em arXiv:2404.15652}, 2024.

\bibitem[GM19]{MR4295519}
A.~Genevois and A.~Martin.
\newblock Automorphisms of graph products of groups from a geometric
  perspective.
\newblock {\em Proc. Lond. Math. Soc. (3)}, 119(6):1745--1779, 2019.

\bibitem[Gre90]{GreenGP}
E.~Green.
\newblock Graph products of groups.
\newblock {\em PhD Thesis}, 1990.

\bibitem[GS10]{Guba2010}
V.~Guba and M.~Sapir.
\newblock On the conjugacy growth functions of groups.
\newblock {\em Illinois Journal of Mathematics}, 54(1):301--313, 2010.

\bibitem[GY22]{GeYang}
I.~Gekhtman and W.~Yang.
\newblock Counting conjugacy classes in groups with contracting elements.
\newblock {\em Journal of Topology}, 15(2):620--665, 2022.

\bibitem[HK96]{MR1420527}
J.~Hagauer and S.~Klav\v{z}ar.
\newblock Clique-gated graphs.
\newblock {\em Discrete Math.}, 161(1-3):143--149, 1996.

\bibitem[HO13]{Hull2013}
M.~Hull and D.~Osin.
\newblock Conjugacy growth of finitely generated groups.
\newblock {\em Advances in Mathematics}, 235:361--389, 2013.

\bibitem[HW99]{HsuWise}
T.~Hsu and D.~Wise.
\newblock On linear and residual properties of graph products.
\newblock {\em Michigan Math. J.}, 46:251--259, 1999.

\bibitem[Mer17]{Mercier2016}
V.~Mercier.
\newblock Conjugacy growth series of some wreath products, 2017.
\newblock \url{http://arxiv.org/abs/1610.07868}.

\bibitem[MO15]{MR3368093}
A.~Minasyan and D.~Osin.
\newblock Acylindrical hyperbolicity of groups acting on trees.
\newblock {\em Math. Ann.}, 362(3-4):1055--1105, 2015.

\bibitem[NR03]{MR1983376}
G.~Niblo and L.~Reeves.
\newblock Coxeter groups act on {${\rm CAT}(0)$} cube complexes.
\newblock {\em J. Group Theory}, 6(3):399--413, 2003.

\bibitem[Osi16]{MR3430352}
D.~Osin.
\newblock Acylindrically hyperbolic groups.
\newblock {\em Trans. Amer. Math. Soc.}, 368(2):851--888, 2016.

\bibitem[Riv10]{Rivin2010}
I.~Rivin.
\newblock Growth in free groups and other stories.
\newblock {\em Illinois Journal of Mathematics}, 54(1):327--370, 2010.

\bibitem[Sis18]{MR3849623}
A.~Sisto.
\newblock Contracting elements and random walks.
\newblock {\em J. Reine Angew. Math.}, 742:79--114, 2018.

\bibitem[Soe24]{MR4735237}
M.~Soergel.
\newblock A generalization of the {D}avis-{M}oussong complex for {D}yer groups.
\newblock {\em J. Comb. Algebra}, 8(1-2):209--249, 2024.

\bibitem[Sul14]{MR3175245}
H.~Sultan.
\newblock Hyperbolic quasi-geodesics in {CAT}(0) spaces.
\newblock {\em Geom. Dedicata}, 169:209--224, 2014.

\bibitem[SV24]{DyerAcylHyp}
M.~Soergel and N.~Vaskou.
\newblock Dyer groups: Centres, hyperbolicity, and acylindrical hyperbolicity.
\newblock {\em arxiv:2410.22464}, 2024.

\bibitem[WXY25]{MR4891022}
R.~Wan, X.~Xu, and W.~Yang.
\newblock Marked length spectrum rigidity in groups with contracting elements.
\newblock {\em J. Lond. Math. Soc. (2)}, 111(4):Paper No. e70146, 2025.

\bibitem[Yan20]{MR4081104}
W.-Y. Yang.
\newblock Genericity of contracting elements in groups.
\newblock {\em Math. Ann.}, 376(3-4):823--861, 2020.

\end{thebibliography}

%

\Address
\end{document}